\documentclass[letterpaper,11pt,oneside,reqno]{amsart}

\usepackage[sorting=nyt,style=alphabetic,backend=bibtex,hyperref=true,doi=false,maxbibnames=9,maxcitenames=4]{biblatex}
\makeatletter
\def\blx@maxline{77}
\makeatother
\usepackage{amsmath,amssymb,amsthm,amsfonts}
\usepackage{array}
\usepackage{graphicx}
\usepackage[mathscr]{euscript}
\usepackage{color,tikz}
\usetikzlibrary{shapes,arrows,positioning,decorations.markings}
\usepackage[DIV12]{typearea}
\usepackage{adjustbox}
\allowdisplaybreaks
\numberwithin{equation}{section}
\usepackage{upgreek}
\usepackage{hyperref}
\usepackage{cleveref}
\usepackage{enumerate}

\newcommand{\Z}{\mathbb{Z}}
\newcommand{\C}{\mathbb{C}}
\newcommand{\R}{\mathbb{R}}

\renewcommand{\i}{\mathbf{i}}
\DeclareMathOperator{\EE}{\mathbb{E}}

\newcommand{\al}{\alpha}
\newcommand{\la}{\lambda}
\newcommand{\be}{\beta}
\newcommand{\ka}{\varkappa}

\newcommand{\md}{\,|\,}
\DeclareMathOperator{\prob}{\mathbb{P}}
\renewcommand{\Re}{\mathop{\mathrm{Re}}}
\renewcommand{\Im}{\mathop{\mathrm{Im}}}

\newcommand{\HeightFunction}{\mathfrak{h}}
\newcommand{\HeightFunctionBernoulli}{\mathfrak{h}^{\textnormal{Ber}}}
\newcommand{\widehatHeightFunctionBernoulli}{\widehat{\mathfrak{h}}^{\textnormal{Ber}}}
\newcommand{\HeightFunctionGenBernoulli}{{\mathfrak{h}}^{\textnormal{gBer}}}

\newcommand{\HSOp}{\mathscr{D}}
\newcommand{\WhittakerOp}{\mathscr{W}}
\newcommand{\MacdonaldOp}{\mathscr{M}}
\newcommand{\ShiftOp}{\mathscr{T}}
\newcommand{\WhittakerMeas}{\mathbf{W}}
\newcommand{\SchurMeas}{\mathbf{S}}
\newcommand{\WhittakerProc}{\mathbf{WP}}

\newcommand{\nua}{\mathsf{c}}

\newcommand{\FProb}{\mathsf{F}^{\textnormal{stoch}}}
\newcommand{\FDenominators}[1]{{\mathsf{F}}^{\scriptscriptstyle(#1)}}

\newcommand{\aind}{{\upalpha}}
\newcommand{\bind}{{\upbeta}}

\newcommand{\Part}{\mathbb{P}}
\newcommand{\Partt}{\mathbb{P}'}

\newcommand{\cont}{\boldsymbol\gamma}

\newcommand{\SSS}{\mathcal{S}}
\newcommand{\CCC}{\mathcal{C}}
\newcommand{\SSST}{\widetilde{\mathcal{S}}}
\newcommand{\CCCT}{\widetilde{\mathcal{C}} }

\newcommand{\ttime}{\mathsf{t}}
\newcommand{\pgeom}{\mathbf{p}}
\newcommand{\x}{\mathbf{x}}
\newcommand{\X}{\mathbf{X}}
\newcommand{\y}{\mathbf{y}}
\newcommand{\xhs}{\mathsf{x}}
\newcommand{\xhsp}{\mathsf{x}'}
\newcommand{\xhspvec}{\vec{\mathsf{x}}'}
\newcommand{\yhs}{\mathsf{y}}
\newcommand{\yhsgb}{\mathsf{y}^{\GB{.12}}}

\newcommand{\yhsbg}{\mathsf{y}^{\BG{.12}}}
\newcommand{\yhsgbvec}{{\vec{\mathsf{y}}}^{\GB{.12}}}
\newcommand{\yhsbgvec}{{\vec{\mathsf{y}}}^{\,\BG{.12}}}
\newcommand{\yhsp}{\mathsf{y}^{\scriptscriptstyle{\boldsymbol\dagger}}}
\newcommand{\LLL}{\mathsf{L}}

\newcommand{\SP}{\mathsf{s}}

\newcommand{\GB}[1]{\scalebox{#1}{\begin{tikzpicture}
[scale=1, line width=8pt]
\draw[->] (0,0)--++(1,0)--++(0,1);
\end{tikzpicture}}}
\newcommand{\BG}[1]{\scalebox{#1}{\begin{tikzpicture}
[scale=1, line width=8pt]
\draw[->] (0,0)--++(0,1)--++(1,0);
\end{tikzpicture}}}

\def\epsx{.12}
\newcommand{\vertexoo}[1]{\scalebox{#1}{\begin{tikzpicture}
[scale=1, very thick]
\node (i1) at (0,-1) {$g$};
\node (j1) at (-1,0) {$0$};
\node (i2) at (0,1) {$g$};
\node (j2) at (1,0) {$0$};
\draw[densely dotted] (j1) -- (j2);
\foreach \shi in {(0,0), (\epsx,0), (-\epsx,0)}
{\begin{scope}[shift=\shi]
\node (shi1) at (0,-1) {\phantom{$g$}};
\node (shi2) at (0,1) {\phantom{$g$}};
\draw[->] (shi1) -- (shi2);
\draw[->] (shi1) --++ (0,.5);
\end{scope}}
\end{tikzpicture}}}

\newcommand{\vertexol}[1]{\scalebox{#1}{\begin{tikzpicture}
[scale=1, very thick]
\node (i1) at (0,-1) {$g$};
\node (j1) at (-1,0) {$0$};
\node (i2) at (0,1) {$g-1$};
\node (j2) at (1,0) {$1$};
\foreach \shi in {(0,0), (-\epsx,0)}
{\begin{scope}[shift=\shi]
\node (shi1) at (0,-1) {\phantom{$g$}};
\node (shi2) at (0,1) {\phantom{$g$}};
\draw[->] (shi1) -- (shi2);
\draw[->] (shi1) --++ (0,.5);
\end{scope}}
\foreach \shi in {(\epsx,0)}
{\begin{scope}[shift=\shi]
\node (shi1) at (0,-1) {\phantom{$g$}};
\node (shi2) at (0,1) {\phantom{$g$}};
\draw[->] (shi1) -- (0,0) -- (j2);
\draw[->] (shi1) --++ (0,.5);
\end{scope}}
\draw[densely dotted] (j1) -- (j2);
\end{tikzpicture}}}

\newcommand{\vertexll}[1]{\scalebox{#1}{\begin{tikzpicture}
[scale=1, very thick]
\node (i1) at (0,-1) {$g$};
\node (i1shm1) at (-\epsx,-1) {\phantom{$g$}};
\node (i1sh1) at (\epsx,-1) {\phantom{$g$}};
\node (j1) at (-1,0) {$1$};
\node (i2) at (0,1) {$g$};
\node (i2shm1) at (-\epsx,1) {\phantom{$g$}};
\node (i2sh1) at (\epsx,1) {\phantom{$g$}};
\node (j2) at (1,0) {$1$};
\draw[densely dotted] (j1) -- (j2);
\draw[densely dotted] (i1) -- (i2);
\draw[->] (j1) -- (-\epsx*1.5,0)--++(.5*\epsx,.5*\epsx) -- (i2shm1);
\draw[->] (i1shm1) -- (-\epsx,-\epsx) -- (0,\epsx) -- (i2);
\draw[->] (i1) -- (0,-\epsx) -- (\epsx,\epsx) -- (i2sh1);
\draw[->] (i1sh1) -- (\epsx,-.5*\epsx)--++(.5*\epsx,.5*\epsx) -- (j2);
\draw[->] (j1) --++ (.5,0);
\draw[->] (i1) --++ (0,.5);
\draw[->] (i1sh1) --++ (0,.5);
\draw[->] (i1shm1) --++ (0,.5);
\end{tikzpicture}}}

\newcommand{\vertexlo}[1]{\scalebox{#1}{\begin{tikzpicture}
[scale=1, very thick]
\node (i1) at (0,-1) {$g$};
\node (j1) at (-1,0) {$1$};
\node (i2) at (0,1) {$g+1$};
\node (i2shm2) at (-3/2*\epsx,1) {\phantom{$g$}};
\node (j2) at (1,0) {$0$};
\draw[densely dotted] (j1) -- (j2);
\draw[->] (j1) -- (-3/2*\epsx,0) -- (i2shm2);
\foreach \shi in {(1/2*\epsx,0), (3/2*\epsx,0), (-1/2*\epsx,0)}
{\begin{scope}[shift=\shi]
\node (shi1) at (0,-1) {\phantom{$g$}};
\node (shi2) at (0,1) {\phantom{$g$}};
\draw[->] (shi1) -- (shi2);
\draw[->] (shi1) --++ (0,.5);
\end{scope}}
\draw[->] (j1) --++ (.5,0);
\end{tikzpicture}}}

\newtheorem{proposition}{Proposition}[section]
\newtheorem{lemma}[proposition]{Lemma}
\newtheorem{corollary}[proposition]{Corollary}
\newtheorem{theorem}[proposition]{Theorem}
\newtheorem{introtheorem}[proposition]{Theorem}
\theoremstyle{definition}
\newtheorem{definition}[proposition]{Definition}
\newtheorem{remark}[proposition]{Remark}
\begin{document}
\title{Stochastic Higher Spin Six Vertex Model and $q$-TASEP\lowercase{s}}

\author[D. Orr]{Daniel Orr}
\address{D. Orr,
Department of Mathematics (MC 0123), 460 McBryde Hall, Virginia Tech,
225 Stanger Street, Blacksburg, VA 24061 USA}
\email{dorr@vt.edu}

\author[L. Petrov]{Leonid Petrov}
\address{L. Petrov, University of Virginia, Department of Mathematics,
141 Cabell Drive, Kerchof Hall,
P.O. Box 400137,
Charlottesville, VA 22904, USA,
and Institute for Information Transmission Problems,
Bolshoy Karetny per. 19, Moscow, 127994, Russia}
\email{lenia.petrov@gmail.com}

\date{}

\begin{abstract}
  We present two new connections between the inhomogeneous stochastic higher spin six vertex model in a quadrant and integrable stochastic systems from the Macdonald processes hierarchy. First, we show how Macdonald $q$-difference operators with $t=0$ (an algebraic tool crucial for studying the corresponding Macdonald processes) can be utilized to get $q$-moments of the height function $\mathfrak{h}$ in the higher spin six vertex model first computed in \cite{BorodinPetrov2016inhom} using Bethe ansatz. This result in particular implies that for the vertex model with the step Bernoulli boundary condition, the value of $\mathfrak{h}$ at an arbitrary point $(N+1,T)\in\mathbb{Z}_{\ge2}\times\mathbb{Z}_{\ge1}$ has the same distribution as the last component $\lambda_N$ of a random partition under a specific $t=0$ Macdonald measure. On the other hand, it is known that $\mathbf{x}_N:=\lambda_N-N$ can be identified with the location of the $N$th particle in a certain discrete time $q$-TASEP started from the step initial configuration. The second construction we present is a coupling of this $q$-TASEP and the higher spin six vertex model (with the step Bernoulli boundary condition) along time-like paths providing an independent probabilistic explanation of the equality of $\mathfrak{h}(N+1,T)$ and $\mathbf{x}_N+N$ in distribution. Combined with the identification of averages of observables between the stochastic higher spin six vertex model and Schur measures (which are $t=q$ Macdonald measures) obtained recently in \cite{borodin2016stochastic_MM}, this produces GUE Tracy--Widom asymptotics for a discrete time $q$-TASEP with the step initial configuration and special jump parameters.
\end{abstract}

\maketitle

\setcounter{tocdepth}{1}
\tableofcontents
\setcounter{tocdepth}{3}

% \begin{comment}

\section{Introduction} % (fold)
\label{sec:introduction}

\subsection{Bethe ansatz and $q$-difference operators} % (fold)
\label{sub:bethe_ansatz_and_q_difference_operators}

The past decade has seen a wave of results on integrable (in other words, exactly solvable) stochastic systems in $(1+1)$ dimension beyond the free fermion (determinantal/Pfaffian) case, such as the partially asymmetric simple exclusion process (ASEP), directed random polymers, Macdonald measures and processes, the $q$-deformed totally asymmetric simple exclusion processes ($q$-TASEPs) and related systems, and $U_q(\widehat{\mathfrak{sl}_2})$ stochastic vertex models. Asymptotic analysis of these systems is based on the presence of concise exact formulas for averages of certain observables, often expressing the latter as multiple contour integrals. The asymptotic behavior places the above systems into the Kardar--Parisi--Zhang (KPZ) universality class (named after a stochastic PDE introduced thirty years ago \cite{KPZ1986}) --- a vaguely defined family of stochastic systems whose long-time and large-scale fluctuations are described by the GUE Tracy--Widom distribution \cite{tracy_widom1994level_airy} or one of its relatives. We refer to the survey \cite{CorwinKPZ} for details on the KPZ universality.

So far two principal mechanisms for getting explicit formulas triggering asymptotic analysis have been utilized: via some form of the Bethe ansatz, or via Macdonald $q$-difference operators. The goal of this paper is to describe new surprising structural connections between the two mechanisms, and also between some of the integrable stochastic systems solvable by these mechanisms. Let us begin by discussing the Bethe ansatz and the $q$-difference operator approaches in more detail.

\medskip

The coordinate formulation of the \emph{Bethe ansatz} goes back to \cite{Bethe1931} and postulates that one should look for eigenfunctions of a quantum integrable many-body system in the form of superposition of those for noninteracting bodies. This idea was used in the pioneering work of Tracy and Widom \cite{TW_ASEP1}, \cite{tracy2008fredholm}, \cite{TW_ASEP2} to diagonalize the Markov generator of the ASEP on the line. Later these methods (combined with Markov duality) were extended to other interacting particle systems in $(1+1)$ dimension such as the $q$-TASEPs\footnote{There is one continuous time (\emph{Poisson}) and two discrete time (\emph{geometric} and \emph{Bernoulli}) versions of the $q$-TASEP \cite{BorodinCorwin2013discrete}. The discussion in \Cref{sub:bethe_ansatz_and_q_difference_operators} applies to all of them.} and related models \cite{BorodinCorwinSasamoto2012}, \cite{BorodinCorwin2013discrete}, \cite{Povolotsky2013}, \cite{Corwin2014qmunu}, and further to the stochastic higher spin six vertex model \cite{CorwinPetrov2015}. Suitable degenerations of the latter system lead to ASEP and all $q$-TASEPs. The algebraic Bethe ansatz (based on the Yang--Baxter equation) can be used to solve an even more general stochastic higher spin six vertex model with inhomogeneities in both time and space direction, and produces $\ell$-fold contour integral formulas for the $q$-moments $\EE q^{\ell \HeightFunction(N+1,T)}$, $\ell\ge1$, of the height function $\HeightFunction$ in this model \cite{BorodinPetrov2016inhom}. This formula for the $q$-moments degenerates, in appropriate limits, to most known formulas of this type. We recall a particular case of the stochastic higher spin six vertex model in \Cref{sub:intro_coupling} below, and the fully general model and contour integral formulas in \Cref{sec:stochastic_higher_spin_six_vertex_model}.

\medskip

The most general instance of the \emph{$q$-difference operator} approach developed quite recently in \cite{BorodinCorwin2011Macdonald}, \cite{BCGS2013} applies to Macdonald measures on partitions $\la=(\la_1\ge \ldots\ge\la_N\ge0)$, $\la_i\in\Z$, and Macdonald processes on sequences of partitions. This approach is based on eigenrelations for Macdonald polynomials $P_\la(a_1,\ldots,a_N\md q,t)$ \cite[Ch. VI]{Macdonald1995} which are indexed by partitions $\la$, are symmetric in the $a_i$'s, and additionally depend on two Macdonald parameters $0\le q,t<1$. These eigenrelations imply that unnormalized probability weights under a Macdonald measure on partitions $\la$ (given by products of two Macdonald polynomials indexed by the same $\la$ and taken at $a_1,\ldots,a_N$ and, say, $b_1,\ldots,b_M$) are eigenfunctions of relatively simple $q$-difference operators in the variables $\{a_i\}_{i=1}^{N}$ viewed as parameters in this measure. The simplest such operator, $\MacdonaldOp_N$, has order one and is given in \eqref{Macdonald_operator}. The action of $\MacdonaldOp_N$ (and other higher order operators also diagonalized in the $P_\la$'s) can be written in a contour integral form. This produces contour integral expressions for averages of the corresponding eigenvalues with respect to Macdonald measures or processes.

Often an application of this method to concrete stochastic systems requires an independent argument to match an observable of interest to an observable of a Macdonald measure or process. Such a matching can be nontrivial: for example, it involves geometric Robinson--Schensted--Knuth correspondences for some models of random polymers \cite{Oconnell2009_Toda}, \cite{Seppalainen2012}, \cite{COSZ2011}, \cite{OSZ2012}, or $q$-analogues of these correspondences for the $q$-TASEPs \cite{OConnellPei2012}, \cite{BorodinPetrov2013NN}, \cite{MatveevPetrov2014}, \cite{pei2016qRSK} (for the continuous time $q$-TASEP a simpler construction was discovered earlier \cite[Section 3.3]{BorodinCorwin2011Macdonald}). Under these matchings for the $q$-TASEPs, the action of the operator $\MacdonaldOp_N$ with $t=0$ (denoted by $\WhittakerOp_N$ and called the \emph{first $q$-Whittaker operator}) produces $k$-fold nested contour integral formulas for the $q$-moments $\EE q^{k(\x_N+N)}$, $k\ge1$, where $\x_N$ is the location of the $N$-th particle in a $q$-TASEP started from the step initial configuration $\x_i(0)=-i$, $i=1,2,\ldots$. In \Cref{prop:qWhit_moments} we recall formulas pertaining to Macdonald measures with $t=0$ (called the \emph{$q$-Whittaker measures}), and in \Cref{prop:qTASEP_moments} discuss identification with $q$-TASEPs.

\medskip

These two ways of getting explicit formulas appear to be different. On the other hand, they both can be applied to the $q$-TASEPs, yielding the same formulas for the $q$-moments $\EE q^{k(\x_N+N)}$. Moreover, the essential structure of contour integral formulas for the $q$-moments of the $q$-TASEPs carries along Bethe ansatz lines all the way up to the inhomogeneous stochastic higher spin six vertex model.

The \textbf{first main result} of the present paper is an extension of the $q$-difference operator method from $q$-Whittaker measures~/~$q$-TASEPs to the inhomogeneous stochastic higher spin six vertex model. In more detail, we consider measures on partitions $\prob(\la)=\FProb_\la$ associated with arrow configurations in the stochastic higher spin six vertex model at a given horizontal slice. The probability weights $\FProb_\la$ can also be viewed as symmetric rational functions in spectral parameters of the vertex model (there is one spectral parameter per horizontal slice). We define certain conjugations $\HSOp_N$ of the $q$-Whittaker operators $\WhittakerOp_N$. The operators $\HSOp_N$ act on space inhomogeneities (varying from one vertical slice to another) of the stochastic higher spin six vertex model, and their action can be expressed in a contour integral form. In \Cref{lemma:action_is_nice} we show that the operators $\HSOp_N$ act nicely on specific linear combinations of unnormalized multiples $\FDenominators{M}_\la$ of the probabilities $\FProb_\la$. In contrast with the Macdonald case, this is no longer an eigenrelation and, moreover, the action on individual unnormalized weights is far from being nice. Nevertheless, as we show in \Cref{thm:action}, this still yields contour integral formulas for observables of the inhomogeneous stochastic higher spin six vertex model.

This result brings stochastic systems solvable by Bethe ansatz closer to the algebraic framework associated with (the $q$-Whittaker part of) the Macdonald hierarchy of symmetric functions and corresponding integrable stochastic systems.\footnote{Recently in \cite{deGierWheeler2016} a common generalization of the symmetric functions $\FProb_\la$ and the Macdonald symmetric polynomials was suggested, providing a hope for a more complete unification of the Bethe ansatz and the $q$-difference operator approaches.} As a byproduct, we also get a new independent proof of the contour integral formulas for the $q$-moments $\EE q^{\ell \HeightFunction(N+1,T)}$ which were first written down in \cite{BorodinPetrov2016inhom}.

% subsection bethe_ansatz_and_q_difference_operators (end)

\subsection{$q$-TASEP~/~vertex model coupling along time-like paths} % (fold)
\label{sub:intro_coupling}

To formulate the next result let us recall the definition of the \emph{homogeneous stochastic higher spin six vertex model with the step Bernoulli boundary condition} from \cite{CorwinPetrov2015}.\footnote{The term ``step Bernoulli'' follows the recent work \cite{AmolBorodin2016Phase} where it is connected with the step Bernoulli (also sometimes called half stationary) initial configuration for the ASEP \cite{TW_ASEP4}.} We assume that the main parameter $q\in(0,1)$ is fixed throughout the paper.

Fix parameters $0<\al<1$, $\be>0$, and denote $\tilde\be:=\be/(1+\be)$ and $\tilde\al:=\al/(1+\be)$. Consider a probability distribution on the space of infinite up-right paths in the quadrant $\Z_{\ge2}\times\Z_{\ge1}$, such that paths can share a vertical edge but not a horizontal edge (see \Cref{fig:intro_vertex_weights}, left). A new path can begin at each horizontal edge on the left boundary independently with probability $\tilde \be$. The distribution of the whole path collection is defined inductively using the stochastic vertex weights in \Cref{fig:intro_vertex_weights}, right, by determining the configuration in the first horizontal slice, then in the second horizontal slice given the configuration in the first one, etc. For each $N\ge1$ and $T\ge0$ let $\HeightFunctionBernoulli(N+1,T)$ be the number of vertical arrows in the $T$-th horizontal slice (i.e., crossing the horizontal line at height $T+\frac12$) strictly to the right of the location $N$. We have $\HeightFunctionBernoulli(N+1,0)\equiv0$.

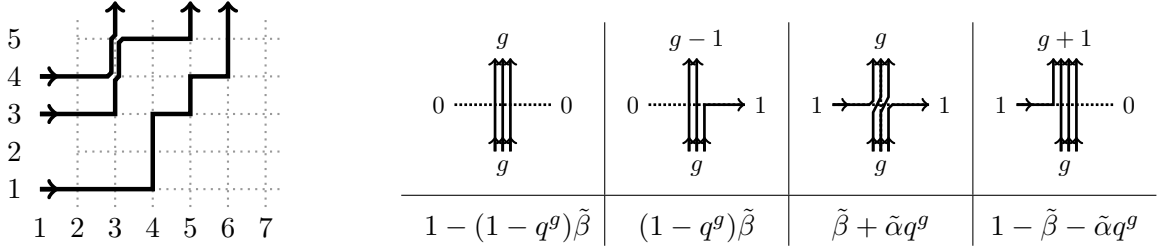
\begin{figure}[htbp]
	\raisebox{-40pt}{\begin{tikzpicture}
		[scale=.5,thick]
		\def\d{.1}
		\foreach \yyyy in {2,...,7}
		{
		\draw[dotted, opacity=.4] (\yyyy-2,5.5)--++(0,-5);
		\node[below] at (\yyyy-2,.5) {$\yyyy$};
		}
		\node[below] at (-1,.5) {$1$};
		\foreach \yyyy in {1,2,3,4,5}
		{
		\draw[dotted, opacity=.4] (0,\yyyy)--++(5.5,0);
		\node[left] at (-1.2,\yyyy) {$\yyyy$};
		}
		\draw[->, line width=1.7pt] (-1,1)--++(.5,0);
		\draw[->, line width=1.7pt] (-1,3)--++(.5,0);
		\draw[->, line width=1.7pt] (-1,4)--++(.5,0);
		\draw[->, line width=1.7pt]
		(-1,4)--++(2-2*\d,0)--++(\d,\d)--++(0,1-\d)--++(\d,\d)--++(0,1-\d);
		\draw[->, line width=1.7pt]
		(-1,3)--++(2,0)--++(0,1-\d)--++(\d,\d)--++(0,1-\d)--++(\d,\d)--++(2-2*\d,0)--++(0,1);
		\draw[->, line width=1.7pt]
		(-1,1)--++(3,0)--++(0,2)--++(1,0)--++(0,1)--++(1,0)
		--++(0,2);
	\end{tikzpicture}}\qquad\qquad
		\begin{tabular}{c|c|c|c}
		\vertexoo{.85}
		&\vertexol{.85}
		&\vertexll{.85}
		&\vertexlo{.85}
		\\
		\hline\rule{0pt}{15pt}
		$1-(1-q^g)\tilde \be$&
		$(1-q^g)\tilde \be$&
		$\tilde \be+\tilde\al q^g$&
		$1-\tilde \be-\tilde\al q^g$
	\end{tabular}
	\caption{A configuration and vertex weights in the homogeneous stochastic higher spin six vertex model with the step Bernoulli boundary condition. For this configuration we have, for example, $\HeightFunctionBernoulli(N+1,T)=2$, where $(N,T)=(2,3)$.}
	\label{fig:intro_vertex_weights}
\end{figure}

Next, let us recall the (homogeneous) \emph{geometric and Bernoulli $q$-TASEPs} introduced and solved by Bethe ansatz in \cite{BorodinCorwin2013discrete}. A connection of them with $q$-Whittaker processes was fully developed later in \cite{MatveevPetrov2014}. The geometric and Bernoulli $q$-TASEPs are discrete time Markov chains on the space of $L$-particle configurations $\vec x=\{x_1>x_2>\ldots>x_L\}$ in $\Z$ ($L\ge1$ is arbitrary and fixed) in which at most one particle per site is allowed, and particles jump only to the right.

\begin{figure}[htbp]
	\begin{adjustbox}{max width=.61\textwidth}
	\begin{tikzpicture}
		[scale=1,very thick]
		\def\pt{.17}
		\def\ee{.1}
		\def\h{.45}
		\draw[->] (-2,0) -- (7.5,0);
		\foreach \ii in {
		(-3*\h,0),
		(-2*\h,0),
		(-1*\h,0),
		(0*\h,0),
		(2*\h,0),
		(3*\h,0),(7*\h,0),(5*\h,0),(6*\h,0) ,(8*\h,0),(10*\h,0),(9*\h,0),(12*\h,0),(13*\h,0),(14*\h,0),(15*\h,0)}
		{
			\draw \ii circle(\pt);
		}
		\foreach \ii in {(\h,0),(4*\h,0),(11*\h,0)}
		{
			\draw[fill] \ii circle(\pt);
		}
		\node at (1*\h,-5*\ee) {$x_{3}(\ttime)$};
		\node at (4*\h,-5*\ee) {$x_{2}(\ttime)$};
		\node at (11*\h,-5*\ee) {$x_{1}(\ttime)$};
    \draw[->, very thick] (11*\h,.3)
    to [in=180, out=90] (11.5*\h,.6)
    to [in=90, out=0] (12*\h,.3);
    \draw[->, very thick] (11*\h,.3)
    to [in=180, out=90] (12*\h,.7)
    to [in=90, out=0] (13*\h,.3);
    \draw[->, very thick] (11*\h,.3)
    to [in=180, out=90] (12.5*\h,.8)
    to [in=90, out=0] (14*\h,.3);
    \draw[->, very thick] (11*\h,.3)
    to [in=180, out=90] (13*\h,.9)
    to [in=90, out=0] (15*\h,.3);
    \draw[->, very thick] (11*\h,.3)
    to [in=180, out=90] (13.5*\h,1)
    to [in=150, out=0] (16*\h,.8);
    \draw[->, very thick] (4*\h,.3)
    to [in=180, out=90] (4.5*\h,.6)
    to [in=90, out=0] (5*\h,.3);
    \draw[->, very thick] (4*\h,.3)
    to [in=180, out=90] (5*\h,.7)
    to [in=90, out=0] (6*\h,.3);
    \draw[->, very thick] (4*\h,.3)
    to [in=180, out=90] (5.5*\h,.8)
    to [in=90, out=0] (7*\h,.3);
    \draw[->, very thick] (4*\h,.3)
    to [in=180, out=90] (6*\h,.9)
    to [in=90, out=0] (8*\h,.3);
    \draw[->, very thick] (4*\h,.3)
    to [in=180, out=90] (6.5*\h,1)
    to [in=90, out=0] (9*\h,.3);
    \draw[->, very thick] (4*\h,.3)
    to [in=180, out=90] (7*\h,1.1)
    to [in=90, out=0] (10*\h,.3);
    \draw[->, very thick] (1*\h,.3)
    to [in=180, out=90] (1.5*\h,.6)
    to [in=90, out=0] (2*\h,.3);
    \draw[->, very thick] (1*\h,.3)
    to [in=180, out=90] (2*\h,.7)
    to [in=90, out=0] (3*\h,.3);
	\end{tikzpicture}
	\end{adjustbox}
	\caption{Geometric $q$-TASEP, with possibilities of jumps indicated.}
	\label{fig:intro_geometric}
\end{figure}
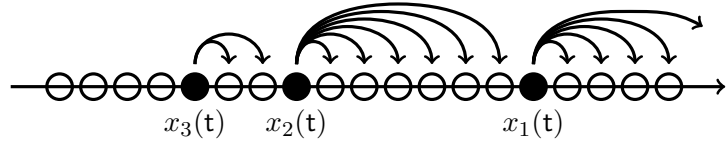
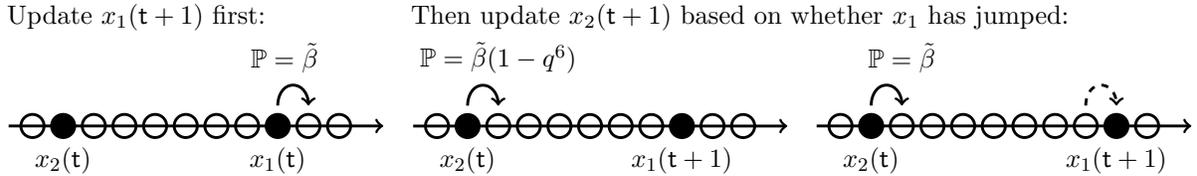
\begin{figure}[htbp]
	\begin{center}
		\begin{tabular}{lll}
		{\small\rm{}Update $x_1(\ttime+1)$ first:}&
		\multicolumn{2}{l}{\small\rm{}{Then update $x_2(\ttime+1)$ based on
		whether $x_1$ has jumped:}}\\
		\begin{adjustbox}{max width=.31\textwidth}
		\begin{tikzpicture}
			[scale=1,very thick]
			\def\pt{.17}
			\def\ee{.1}
			\def\h{.45}
			\draw[->] (1,0) -- (6.5,0);
			\foreach \ii in {(3*\h,0),(7*\h,0),(5*\h,0),(6*\h,0) ,(8*\h,0),(10*\h,0),(9*\h,0),(12*\h,0),(13*\h,0)}
			{
				\draw \ii circle(\pt);
			}
			\foreach \ii in {(4*\h,0),(11*\h,0)}
			{
				\draw[fill] \ii circle(\pt);
			}
			\node at (4*\h,-5*\ee) {$x_{2}(\ttime)$};
			\node at (11*\h,-5*\ee) {$x_{1}(\ttime)$};
		    \draw[->, very thick] (11*\h,.3) to [in=180, out=90] (11.5*\h,.6)
		    to [in=90, out=0] (12*\h,.3)
		    node [xshift=-10,yshift=20] {$\prob=\tilde \be$};
		\end{tikzpicture}
		\end{adjustbox}
		&
		\begin{adjustbox}{max width=.31\textwidth}
		\begin{tikzpicture}
			[scale=1,very thick]
			\def\pt{.17}
			\def\ee{.1}
			\def\h{.45}
			\draw[->] (1,0) -- (6.5,0);
			\foreach \ii in {(3*\h,0),(7*\h,0),(5*\h,0),(6*\h,0) ,(8*\h,0),(10*\h,0),(9*\h,0),(12*\h,0),(13*\h,0)}
			{
				\draw \ii circle(\pt);
			}
			\foreach \ii in {(4*\h,0),(11*\h,0)}
			{
				\draw[fill] \ii circle(\pt);
			}
			\node at (4*\h,-5*\ee) {$x_{2}(\ttime)$};
			\node at (11*\h,-5*\ee) {$x_{1}(\ttime+1)$};
		    \draw[->, very thick] (4*\h,.3) to [in=180, out=90] (4.5*\h,.6)
		    to [in=90, out=0] (5*\h,.3)
		    node [xshift=0,yshift=20] {$\prob=
		    \tilde\be(1-q^6)$};
		\end{tikzpicture}
		\end{adjustbox}
		&
		\begin{adjustbox}{max width=.31\textwidth}
		\begin{tikzpicture}
			[scale=1,very thick]
			\def\pt{.17}
			\def\ee{.1}
			\def\h{.45}
			\draw[->] (1,0) -- (6.5,0);
			\foreach \ii in {(3*\h,0),(7*\h,0),(5*\h,0),(6*\h,0) ,(8*\h,0),(10*\h,0),(9*\h,0),(11*\h,0),(13*\h,0)}
			{
				\draw \ii circle(\pt);
			}
			\foreach \ii in {(4*\h,0),(12*\h,0)}
			{
				\draw[fill] \ii circle(\pt);
			}
			\node at (4*\h,-5*\ee) {$x_{2}(\ttime)$};
			\node at (12*\h,-5*\ee) {$x_{1}(\ttime+1)$};
		    \draw[->, very thick, dashed] (11*\h,.3) to [in=180, out=90] (11.5*\h,.6)
		    to [in=90, out=0] (12*\h,.3);
		    \draw[->, very thick] (4*\h,.3) to [in=180, out=90] (4.5*\h,.6)
		    to [in=90, out=0] (5*\h,.3)
		    node [xshift=0,yshift=20] {$\prob=\tilde\be$};
		\end{tikzpicture}
		\end{adjustbox}
		\end{tabular}
	\end{center}
	\caption{Bernoulli $q$-TASEP.}
	\label{fig:intro_Bernoulliasd}
\end{figure}

Under the geometric $q$-TASEP with parameter $0<\al<1$, at each increment $\ttime{}\to\ttime{}+1$ of the discrete time each particle $x_i$ independently jumps to the right by a random distance having the probability distribution (see \Cref{fig:intro_geometric})
\begin{equation}\label{geometric_jump_distribution_intro}
	\prob\bigl(x_i(\ttime{}+1)=x_i(\ttime{})+j\md\vec x(\ttime)\bigr)=
	\frac{\al^j(\al;q)_{m-j}(q;q)_m}{(q;q)_j(q;q)_{m-j}},
	\qquad m=x_{i-1}(\ttime{})-x_i(\ttime{})-1,\quad
	0\le j\le m.
\end{equation}
Here and below $(z;q)_{k}=(1-z)(1-zq)\ldots(1-zq^{k-1})$, $k\in\Z_{\ge0}$, denotes the $q$-Pochhammer symbol. Since $0<q<1$, it makes sense for $k=+\infty$, too, and we take $m=+\infty$ for the jump of $x_1$. This defines a one-step Markov transition matrix on $L$-particle configurations which we denote by $\mathbf{G}^\circ_\al$. When $q=0$, the jumping distributions under the geometric $q$-TASEP reduce to the usual geometric distributions with parameter $\al$ truncated so that the particles' order is preserved.

Under the Bernoulli $q$-TASEP with parameter $\be>0$, at each increment $\ttime{}\to\ttime{}+1$ of the discrete time each particle can jump by at most one, and the jumps are not independent (the interaction propagates from right to left). Namely, the particle $x_1$ jumps to the right by one with probability $\tilde \be$ or stays put with the complementary probability $1-\tilde \be$. For each $i=2,\ldots,L$, if $x_{i-1}$ jumps (i.e., $x_{i-1}(\ttime{}+1)=x_{i-1}(\ttime{})+1$), then the jumping law of $x_i$ is the same as for $x_1$. If, on the other hand, $x_{i-1}$ stays, then the probability that $x_i$ jumps to the right by one is decreased to $\tilde \be(1-q^{x_{i-1}(\ttime{})-x_i(\ttime{})-1})$. See \Cref{fig:intro_Bernoulliasd}. Denote the corresponding one-step Markov transition matrix by $\mathbf{B}^\circ_\be$.

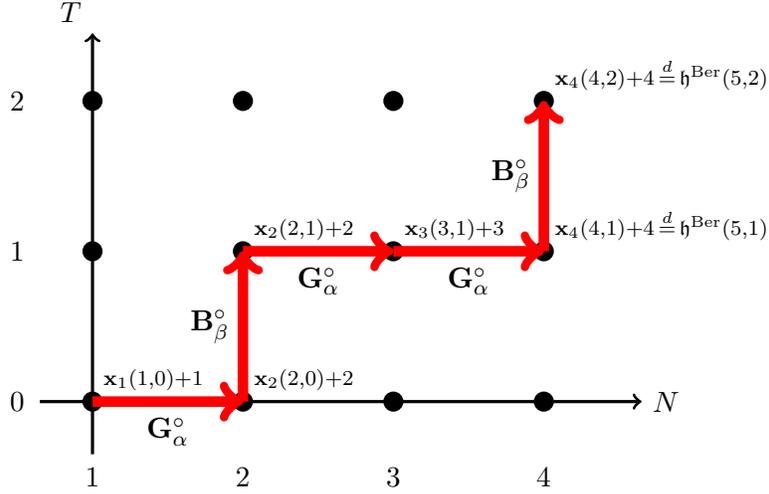
\begin{figure}[htbp]
	\begin{tikzpicture}
		[scale=2,very thick]
		\draw[->] (-.35,0)--++(4,0) node [right] {$N$};
		\draw[->] (0,-.35)--++(0,2.8) node [above left] {$T$};
		\foreach \nn in {1,...,4}
		{\node at (\nn-1,-.5) {$\nn$};}
		\foreach \tt in {0,...,2}
		{\node at (-.5,\tt) {$\tt$};}
		\foreach \nn in {0,...,3}
		{\foreach \tt in {0,...,2}
		{\draw[fill] (\nn,\tt) circle (1.6pt);}}
		\draw[line width=4pt,color=red,->] (0,0)--(1,0);
		\draw[line width=4pt,color=red,->] (1,0)--(1,1);
		\draw[line width=4pt,color=red,->] (1,1)--(2,1);
		\draw[line width=4pt,color=red,->] (2,1)--(3,1);
		\draw[line width=4pt,color=red,->] (3,1)--(3,2);
		\node[above right] at (0,0) {$\scriptstyle\x_1(1,0)+1$};
		\node[above right] at (1,0) {$\scriptstyle\x_2(2,0)+2$};
		\node[above right] at (1,1) {$\scriptstyle\x_2(2,1)+2$};
		\node[above right] at (2,1) {$\scriptstyle\x_3(3,1)+3$};
		\node[above right] at (3,1) {$\scriptstyle\x_4(4,1)+4\,\stackrel{\scriptscriptstyle{d}}{=}\,\HeightFunctionBernoulli(5,1)$};
		\node[above right] at (3,2) {$\scriptstyle\x_4(4,2)+4\,\stackrel{\scriptscriptstyle d}{=}\,\HeightFunctionBernoulli(5,2)$};
		\node at (.5,-.2) {$\mathbf{G}_\al^\circ$};
		\node at (.78,.5) {$\mathbf{B}_\be^\circ$};
		\node at (1.5,1-.2) {$\mathbf{G}_\al^\circ$};
		\node at (2.5,1-.2) {$\mathbf{G}_\al^\circ$};
		\node at (2.78,1.5) {$\mathbf{B}_\be^\circ$};
	\end{tikzpicture}
	\caption{Coupling of the height function $\HeightFunctionBernoulli(N+1,T)$ in the stochastic higher spin six vertex model with the step Bernoulli boundary condition and a mixed geometric/Bernoulli $q$-TASEP $\vec\x(N,T)$ along a path $\mathcal{P}$.}
	\label{fig:capling_intro}
\end{figure}

Fix an up-right path $\mathcal{P}=\{N_\ttime,T_\ttime\}_{\ttime\ge0}$ starting from $(N_0,T_0)=(1,0)$ (this is the red path in \Cref{fig:capling_intro}). With all the above notation, our \textbf{second main result} is the following:

\begin{introtheorem}\label{thm:intro_capling}
	Let parameters $0<\al<1$, $\be>0$, and a path $\mathcal{P}$ be fixed. Then the joint distribution of values of the height function $\HeightFunctionBernoulli(N_\ttime+1,T_\ttime)$ along $\mathcal{P}$ is the same as that of the particle locations $\x_{N_\ttime}(N_\ttime,T_\ttime)+N_\ttime$ in a mixed geometric/Bernoulli $q$-TASEP started from the step initial configuration $\x_{i}(1,0)=-i$, $i\ge1$. Here a horizontal or a vertical part of $\mathcal{P}$ (=~an increment of $N$ or $T$) corresponds to a geometric move $\mathbf{G}_\al^\circ$ or a Bernoulli move $\mathbf{B}_\be^\circ$ in $\vec \x$, respectively. See \Cref{fig:capling_intro} for an illustration.
\end{introtheorem}

Theorem \ref{thm:intro_capling} is a homogeneous specialization of a more general \Cref{thm:capling}. The latter deals with the inhomogeneous stochastic higher spin six vertex model (described in \Cref{sec:stochastic_higher_spin_six_vertex_model}) and $q$-TASEPs with particle-dependent jumping probabilities (\Cref{def:geom,def:Bernoulli}).

Passing from joint distributions along $\mathcal{P}$ to marginal distributions corresponding to an arbitrary point $(N,T)$ we see that Theorem \ref{thm:intro_capling} implies an equality in distribution\footnote{The matrices $\mathbf{B}^\circ_\be$ and $\mathbf{G}^\circ_\al$ commute (which can be viewed as a consequence of the Yang--Baxter equation, see \Cref{prop:commute} for details), so the order of Bernoulli and geometric moves in the definition of a single-time configuration $\vec\x(N,T)$ is irrelevant.} 
\begin{equation}
	\HeightFunctionBernoulli(N+1,T)\stackrel{\scriptscriptstyle d}{=}\x_N(N,T)+N.
	\label{equality_in_distribution}
\end{equation}
Using \cite{MatveevPetrov2014} or \cite{BorodinCorwin2013discrete}, $\x_N(N,T)+N$ can be identified with the last component $\la_N$ under a suitable $q$-Whittaker measure on partitions. By comparing $q$-moments of both sides of \eqref{equality_in_distribution}, one can say that this equality alternatively follows from $q$-difference operator or contour integral considerations discussed in \Cref{sub:bethe_ansatz_and_q_difference_operators}. On the other hand, we prove Theorem \ref{thm:intro_capling} by explicitly constructing a \emph{coupling} between the height function and the mixed geometric/Bernoulli $q$-TASEP along an arbitrary path $\mathcal{P}$. That is, we show that the collection of values of $\HeightFunctionBernoulli{}$ along $\mathcal{P}$ can be interpreted as a function of the trajectory of the mixed $q$-TASEP (depending on $\mathcal{P}$). In this way Theorem \ref{thm:intro_capling} provides an independent probabilistic reason behind \eqref{equality_in_distribution} (and also behind contour integral formulas for $\HeightFunctionBernoulli(N+1,T)$).

\medskip

Joint distributions of the random variables $\x_{N_\ttime}(N_\ttime,T_\ttime)$ as in Theorem \ref{thm:intro_capling} correspond to so-called \emph{time-like paths}. This should be contrasted with \emph{space-like paths} in ($q$-)TASEP-like interacting particle systems in one space dimension which lead to joint distributions of $\{x_{n_i}(t_i)\}_{i=1}^{k}$, where $x_{n_1}(0)\ge \ldots\ge x_{n_k}(0)$ and $t_1\ge \ldots\ge t_k$.\footnote{The terms ``time-like'' and ``space-like'' come from a well-known growth model reformulation of the TASEP and related interacting particle systems, e.g., see \cite{derrida1991dynamics}, \cite{Ferrari_Airy_Survey}.} For $q=0$, joint distributions along space-like paths in the mixed geometric/Bernoulli $q$-TASEP $\{\vec\x(N,T)\}$ described above have a determinantal structure thanks to a connection with Schur processes \cite{BorFerr2008DF} (for $0<q<1$ similar considerations would lead to a connection with $q$-Whittaker processes, but we will not discuss this here). Describing and analyzing joint distributions along time-like paths even in the case $q=0$ is substantially more involved, cf. \cite{johansson2015two}.

Our results produce formulas for certain time-like joint distributions in $q$-TASEPs. Namely, combining contour integral formulas for joint $q$-moments of $\HeightFunctionBernoulli(N_i+1,T)$ for fixed $T$ and arbitrary $\{N_i\}$ obtained in \cite{BorodinPetrov2016inhom} (we recall this result in \Cref{thm:obaservables_Thm9_8}) with Theorem \ref{thm:intro_capling}, one gets joint $q$-moments of the random variables $\{\x_{N_i}(N_i,T)\}$ in a contour integral form. These random variables can be thought of as coming from special time-like paths in the geometric $q$-TASEP started from a random configuration corresponding to $T$ Bernoulli $q$-TASEP moves. Such formulas for joint $q$-moments along general time-like paths presently seem out of reach.

\medskip

We refer to \Cref{sub:more_general_boundary_conditions} for further discussion of extensions of Theorem \ref{thm:intro_capling} to other boundary conditions and to the classical six vertex case (in which there is at most one arrow per edge in both vertical and horizontal directions).

% subsection _q_tasep_vertex_model_coupling (end)

\subsection{Matching with Schur measures and asymptotics of $q$-TASEP} % (fold)
\label{sub:matching_with_schur_measures_and_asymptotics_of_}

Recently Borodin \cite{borodin2016stochastic_MM} established another connection between the inhomogeneous stochastic higher spin six vertex model and Macdonald measures with general parameters $0\le q,t<1$ (as opposed to our results corresponding to $t=0$). Namely, averages of certain single-point observables of the higher spin six vertex model are equal to averages of some other observables with respect to a Macdonald measure (with matching parameters). This identification follows from a direct comparison of contour integral formulas for the two expectations, and algebraic structures behind this fact (as well as relations to our results) remain to be uncovered.

At least in two special cases the Macdonald measure corresponding to a vertex model reduces to a Schur measure (a particular $q=t$ case of the Macdonald measure): for the stochastic six vertex model, and for the higher spin six vertex model as described in \Cref{sub:intro_coupling}, but with special parameter $\al=q$. The identification of expectations in these cases allows to obtain GUE Tracy--Widom asymptotics in vertex models (and related systems) by means of the determinantal structure associated with Schur processes, see \cite[Section 6]{borodin2016stochastic_MM}, \cite[Appendix B]{AmolBorodin2016Phase}, \cite{BO2016_ASEP}.

%Since the identification of expectations only involves single-point observables of the vertex model, the knowledge of multi-point asymptotics of these determinantal processes (leading to the Airy$_2$ point process in the same regime) does not allow to say anything about multi-point asymptotics in vertex models.

Combining results of \cite{borodin2016stochastic_MM} with our Theorem \ref{thm:intro_capling} linking $q$-TASEPs and the stochastic higher spin six vertex model, we see that the mixed geometric/Bernoulli $q$-TASEP with an arbitrary Bernoulli parameter $\be$ and a special parameter $\al=q$ is related (in the sense of matching expectations) to a certain Schur measure. This leads to GUE Tracy--Widom asymptotics for this $q$-TASEP which we formulate and discuss in \Cref{thm:special_q_TASEP_asymptotics}.

\subsection{Outline}
In \Cref{sec:_q_whittaker_measures_and_q_difference_operators} we recall necessary definitions and properties of $q$-Whittaker measures and associated $q$-difference operators. In \Cref{sec:stochastic_higher_spin_six_vertex_model} we recall the inhomogeneous stochastic higher spin six vertex model and contour integral formulas for it. In \Cref{sec:action_of_q_whittaker_difference_operators_on_the_higher_spin_six_vertex_model} we utilize $q$-Whittaker difference operators to provide an alternative proof of these formulas, and establish our results mentioned in \Cref{sub:bethe_ansatz_and_q_difference_operators}. In \Cref{sec:capling_with_q_tasep} we construct a coupling along time-like paths of the higher spin six vertex model having the step Bernoulli boundary condition with a $q$-TASEP started from the step initial configuration, and prove Theorem \ref{thm:intro_capling} (as well as its more general version, \Cref{thm:capling}). The arguments in \Cref{sec:action_of_q_whittaker_difference_operators_on_the_higher_spin_six_vertex_model,sec:capling_with_q_tasep} are fairly independent of each other. In \Cref{sec:asymptotics_of_q_tasep_via_schur_measures} we obtain an asymptotic result for discrete time $q$-TASEPs with special parameters.

\subsection{Acknowledgments} 
We are very grateful to Alexei Borodin for insightful discussions, 
to Alexey Bufetov and Axel Saenz for helpful remarks, 
and to Pavel Etingof for a crucial remark recognizing 
$q$-Whittaker difference operators in the higher spin six vertex model. 
D.O.{} was partially supported by NSF grant DMS-1600653.

% section introduction (end)

\section{$q$-Whittaker measures and difference operators} % (fold)
\label{sec:_q_whittaker_measures_and_q_difference_operators}

In this section we recall the necessary facts about $q$-Whittaker symmetric polynomials and the associated ($q$-)difference operators and probability distributions. We mostly follow \cite[Ch. VI]{Macdonald1995} and \cite[Ch. 2 and 3]{BorodinCorwin2011Macdonald}.

\subsection{Macdonald polynomials} % (fold)
\label{sub:macdonald_polynomials}

Let $\Part_N$ be the set of partitions $\la=(\la_1\ge \ldots\ge\la_N\ge0)$, $\la_i\in\Z$, of length $N\ge0$ ($\Part_0$ consists of the single empty partition), and $\Part=\bigcup_{N=0}^{\infty}\Part_N$ be the set of all partitions. By agreement, $\Part_{N}\subset\Part_{N+1}$, where we append partitions of length $N$ by an additional zero.

The \emph{Macdonald polynomials} \cite{macdonald1988new}, \cite[Ch. VI]{Macdonald1995} form a distinguished family of multivariate symmetric polynomials, and have numerous algebraic properties.
They depend on two parameters $q,t$ which we regard as numbers belonging to $[0,1)$. We adopt the following definition of the Macdonald polynomials.\footnote{There is also an equivalent more abstract definition involving orthogonality with respect to a dot product on the algebra of symmetric functions \cite[Ch. VI.2]{Macdonald1995}.} Let 
\begin{equation}\label{Macdonald_operator}
	\MacdonaldOp_N:=
	\sum_{r=1}^{N}
	\Bigg(\prod_{\substack{1\le i\le N,\,i\ne r}}
	\frac{a_i-ta_r}{a_i-a_r}\Bigg)\ShiftOp_{q;a_r}
\end{equation}
(where $\ShiftOp_{q;x}f(x)=f(qx)$ denotes the $q$-shift in the variable $x$) be the \emph{first Macdonald difference operator} acting on functions in $a_1,\ldots,a_N$. This operator preserves the space of symmetric polynomials in the $a_i$'s, and its eigenfunctions in this space are the Macdonald polynomials (indexed by all $\la\in\Part_N$):
\begin{equation}\label{Macdonald_eigenrelation}
	\MacdonaldOp_N P_\la(a_1,\ldots,a_N\md q,t)=
	(q^{\la_1}t^{N-1}+q^{\la_2}t^{N-2}+
	\ldots
	+q^{\la_{N}})P_\la(a_1,\ldots,a_N\md q,t).
\end{equation}
For generic values of $q,t$ the above eigenvalues are pairwise distinct, so \eqref{Macdonald_eigenrelation} determines the $P_\la$'s up to a constant factor. This factor is fixed by requiring that 
\begin{equation*}
	P_\la(a_1,\ldots,a_N\md q,t)=a_1^{\la_1}a_2^{\la_2}\ldots
	a_N^{\la_N}+\textnormal{lexicographically lower terms},
\end{equation*}
where the lower terms incorporate all dependence on $q,t$.

The Macdonald polynomials are \emph{stable} in the sense that 
\begin{equation*}
	P_{\la\cup 0}(a_1,\ldots,a_N,0\md q,t)
	=P_\la(a_1,\ldots,a_N\md q,t),
\end{equation*}
where $\la\in\Part_N$, and $\la\cup 0$ means appending $\la$ by zero. Thus, they can be regarded as elements $P_\la$ of the algebra of symmetric functions $\mathit{Sym}$ \cite[Ch. I.2]{Macdonald1995} in the variables $\{a_i\}_{i=1}^{\infty}$ (more precisely, the coefficients of these symmetric functions are from $\mathbb{Q}[q,t]$). Let also $Q_\la:=b_\la(q,t)P_\la$ be certain multiples of the Macdonald polynomials, where the numerical constant $b_\la$ is given by \cite[VI.(6.19)]{Macdonald1995}. These $Q_\la$'s are employed in Cauchy summation identities like \eqref{Cauchy_identity} below. Finally, let $P_{\la/\mu}$ and $Q_{\la/\mu}$ denote the corresponding \emph{skew symmetric functions} \cite[Ch. VI.7]{Macdonald1995}.

The family of Macdonald symmetric polynomials includes as subfamilies the \emph{Schur} ($q=t$), \emph{Hall--Littlewood} ($q=0$), and \emph{Jack} ($q=t^{\al}\to1$) symmetric polynomials. Moreover, setting $t=0$ in the Macdonald polynomials leads to the \emph{$q$-Whittaker} symmetric polynomials, which we discuss next.

% subsection macdonald_polynomials (end)

\subsection{$q$-Whittaker polynomials} % (fold) 
\label{sub:_q_whittaker_functions}

Besides being $t=0$ degenerations of the Macdonald polynomials, the $q$-Whittaker polynomials also arise as $q$-deformations of $\mathfrak{gl}_{n}$ Whittaker functions, cf. \cite{GerasimovLebedevOblezin2011}. Most definitions and properties of the $q$-Whittaker polynomials and measures we discuss here and in \Cref{sub:_q_whittaker_measures} below have parallels in the Macdonald case, but for later use it is convenient to list concrete formulas in the $q$-Whittaker situation. We will denote the $q$-Whittaker polynomials simply by $P_\la(a_1,\ldots,a_N)$, omitting the dependence on $q$.

In the $q$-Whittaker case the operator \eqref{Macdonald_operator} and the eigenrelation \eqref{Macdonald_eigenrelation} turn into
\begin{equation}\label{Whittaker_difference_operator}
	\WhittakerOp_N:=\sum_{r=1}^{N}
	\Bigg(\prod_{\substack{1\le i\le N,\,i\ne r}}\frac{a_i}{a_i-a_r}\Bigg)\ShiftOp_{q;a_r}
\end{equation}
and
\begin{equation}\label{Whittaker_difference_operator_eigenrelation}
	\WhittakerOp_N P_\la(a_1,\ldots,a_N)=
	q^{\la_N}P_\la(a_1,\ldots,a_N).
\end{equation}

The functions $\{Q_{(n)}\}_{n=1}^{\infty}$ (where $(n)$ are partitions with only one nonzero part) are algebraically independent and generate the algebra of symmetric functions $\mathit{Sym}$. A \emph{specialization} $\rho$ is an assignment of arbitrary values $Q_{(n)}(\rho)\in\R$ to the functions $Q_{(n)}$, $n=1,2,\ldots$. This is a more general operation than simply specializing the variables $a_i$. A specialization $\rho$ is called (\emph{$q$-Whittaker}) \emph{nonnegative} if $P_{\la/\mu}(\rho)\ge0$ for any $\la,\mu\in\Part$. The classification of $q$-Whittaker nonnegative specializations is only conjectural (see \cite[Section II.9]{Kerov-book}; in fact, Kerov's conjecture also covers the Macdonald case). This conjecture is established in the Jack and the Schur cases, cf. \cite{Kerov1998}, \cite{Kerov-book} and references therein. Nevertheless, this conjectural classification provides a rich supply of $q$-Whittaker nonnegative specializations defined via generating series:
\begin{equation}\label{Pi_definition}
	\sum_{n=0}^{\infty}Q_{(n)}(\rho)\,u^{n}=e^{\gamma u}
	\prod_{i\ge 1}\frac{(1+\be_i u)}{(\al_iu;q)_{\infty}}
	=:\Pi_{\WhittakerMeas}(u;\rho),
\end{equation}
where $\al_i,\be_i\ge0$, $\gamma\ge0$, and $\sum_{i}(\al_i+\be_i)<\infty$. Here $u$ is a formal variable, and $\Pi_{\WhittakerMeas}(u;\rho)$ denotes the right-hand side. Identity \eqref{Pi_definition} holds numerically if $|\al_i u|<\infty$ for all $i$. If $\gamma=0$ and $\be_i=0$ for all $i$, and there are only finitely many $\al_i$'s, then the specialization \eqref{Pi_definition} reduces to plugging in the values $a_i=\al_i$ into the symmetric functions.

The $q$-Whittaker polynomials satisfy a Cauchy summation identity which we can write in the following form \cite[Ch. VI.2]{Macdonald1995}:
\begin{equation}\label{Cauchy_identity}
	\sum_{\la_1\ge\la_2\ge \ldots\ge \la_N\ge0}
	P_\la(a_1,\ldots,a_N)Q_\la(\rho)
	=
	\Pi_{\WhittakerMeas}(a_1;\rho) \ldots\Pi_{\WhittakerMeas}(a_N;\rho)
	:=
	\Pi_{\WhittakerMeas}(a_1,\ldots,a_N;\rho).
\end{equation}
Identity \eqref{Cauchy_identity} holds numerically if $|a_i\al_j|<1$ for all $i,j$ (here the $\al_j$'s come from the specialization $\rho$ as in \eqref{Pi_definition}).

% subsection _q_whittaker_functions (end)

\subsection{$q$-Whittaker measures} % (fold)
\label{sub:_q_whittaker_measures}

Let us take nonnegative $a_1,\ldots,a_N$ and a nonnegative specialization $\rho$ as in \eqref{Pi_definition} such that $|a_i\al_j|<1$ for all $i,j$. Cauchy identity \eqref{Cauchy_identity} allows to define a probability distribution on the set $\Part_N$:
\begin{equation}\label{qWhittaker_measure}
	\WhittakerMeas_{(a_1,\ldots,a_N);\rho}(\la):=
	\frac{P_\la(a_1,\ldots,a_N)Q_\la(\rho)}{\Pi_{\WhittakerMeas}(a_1,\ldots,a_N;\rho)},
	\qquad\la\in\Part_N,
\end{equation}
which is called the \emph{$q$-Whittaker measure}. The nonnegativity of the $a_i$'s and $\rho$ imply that the probability weights \eqref{qWhittaker_measure} are nonnegative for any $\la\in\Part_N$.

\begin{proposition}[{\cite[Proposition 3.1.5]{BorodinCorwin2011Macdonald}}]
	\label{prop:qWhit_moments}
	With the above notation we have
	\begin{multline}\label{qWhit_moments_formula}
		\EE_{\WhittakerMeas_{(a_1,\ldots,a_N);\rho}}
		\bigl[q^{k\la_{N}}\bigr]
		=\frac{(-1)^{k}q^{\frac{k(k-1)}2}}{(2\pi\i)^{k}}
		\oint\ldots\oint \prod_{1\le A<B\le k}\frac{z_A-z_B}{z_A-qz_B}
		\\\times\prod_{j=1}^{k}\biggl[
		\bigg(\prod_{m=1}^{N}\frac{a_m}{a_m-z_j}\biggr)
		\bigg(\prod_{i\ge1}
		(1-\al_iz_j)\frac{1+q\be_iz_j}{1+\be_iz_j}
		\biggr)e^{(q-1)\gamma z_j}
		\frac{dz_j}{z_j}\biggr],
	\end{multline}
	where $\EE_{\WhittakerMeas_{(a_1,\ldots,a_N);\rho}}$ stands for the expectation with respect to the $q$-Whittaker measure \eqref{qWhittaker_measure}. All integration contours are positively oriented and encircle $a_1,\ldots,a_N$, the $z_j$ contour contains $qz_{j+1},\ldots,qz_{k}$, and the contours encircle no other singularities. See \Cref{fig:qWhit_contours}.
\end{proposition}
Note that the integrand in \eqref{qWhit_moments_formula} contains factors of the form $\Pi_{\WhittakerMeas}(qz_j;\rho)/\Pi_{\WhittakerMeas}(z_j;\rho)$. The left-hand side of \eqref{qWhit_moments_formula} will be referred to as the ($k$th) \emph{$q$-moment} of the random variable $\la_N$.
\begin{figure}[htpb]
	\begin{tikzpicture}
		[scale=3]
		\def\pt{0.02}
		\def\q{.6}
		\def\ss{.56}
		\draw[->, thick] (-.8,0) -- (2,0);
	    \draw[->, thick] (0,-.9) -- (0,.9);
	    \draw[fill] (1.3,0) circle (\pt) node [below] {$a_i$};
	    \draw[fill] (1.1,0) circle (\pt);
	  	\draw[fill] (1.4,0) circle (\pt);
	  	\draw[fill] (0,0) circle (\pt) node [below left] {$0$};
	  	\draw (1.25,0) circle (.25) node [below,xshift=-7,yshift=20] {$z_3$};
	  	\draw[densely dashed] (1.25*\q,0) circle (.25*\q);
	  	\draw[densely dashed] (1.25*\q*\q,0) circle (.25*\q*\q);
	  	\draw (1.05,0) ellipse (.48 and .4) node [below,xshift=-36,yshift=34] {$z_2$};
	  	\draw (.92,0) ellipse (.65 and .5) node [below,xshift=-25,yshift=52] {$z_1$};
	\end{tikzpicture}
	\caption{Integration contours in \Cref{prop:qWhit_moments} for 
	$k=3$. The contours $qz_3$ and $q^{2}z_3$ are dashed. Note that this picture requires that $\min|a_i|>q\max|a_i|$, otherwise each contour $z_j$ would consist of several disjoint closed curves. We do not discuss these issues here.}
	\label{fig:qWhit_contours}
\end{figure}
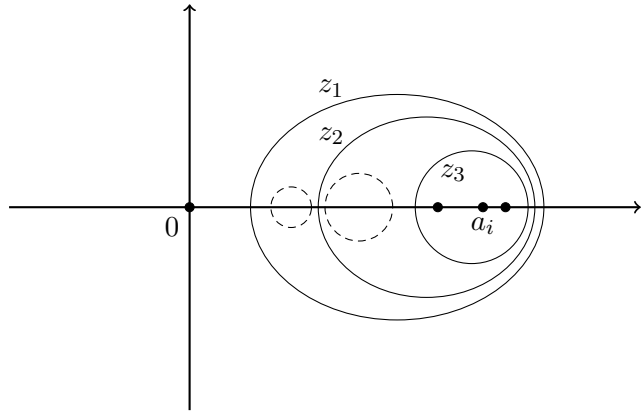
\begin{proof}[Proof of \Cref{prop:qWhit_moments}]
	Let us recall the proof for later use. Formula \eqref{qWhit_moments_formula} is based on two main ingredients. First, applying the operator $\WhittakerOp_N$ in the variables $a_i$ to the Cauchy identity \eqref{Cauchy_identity} $k$ times and dividing by $\Pi_{\WhittakerMeas}(a_1,\ldots,a_N;\rho)$, we see that
	\begin{equation}\label{qWhit_moments_formula_application_of_W}
		\EE_{\WhittakerMeas_{(a_1,\ldots,a_N);\rho}}
		\bigl[q^{k\la_{N}}\bigr]
		=
		\frac{(\WhittakerOp_N)^{k}\,\Pi_{\WhittakerMeas}(a_1,\ldots,a_N;\rho)}{\Pi_{\WhittakerMeas}(a_1,\ldots,a_N;\rho)}
	\end{equation}
	(this follows from the eigenrelation \eqref{Whittaker_difference_operator_eigenrelation}). The second ingredient is that the sum over $r$ in the definition of $\WhittakerOp_N$ \eqref{Whittaker_difference_operator} can be interpreted as a residue expansion, and so the action of $\WhittakerOp_N$ on functions of the product form $f(a_1)\ldots f(a_N)$ can be written as a contour integral
	\begin{equation}\label{qWhit_moments_formula_W_integral}
		\WhittakerOp_N \bigl(f(a_1)\ldots f(a_N)\bigr)=
		-\frac{1}{2\pi\i}\oint
		\bigg(\prod_{i=1}^{N}\frac{a_if(a_i)}{a_i-z}\biggr)
		\frac{f(qz)}{f(z)}\frac{dz}{z}.
	\end{equation}
	Here the $a_i$'s are assumed real, nonzero, and of the same sign, the integration contour encircles all the $a_i$'s, and $f$ is holomorphic and nonzero in a complex neighborhood of an interval of $\R$ containing $\{a_i,qa_i\}_{i=1}^{N}$. The observation that the right-hand side of \eqref{qWhit_moments_formula_W_integral} depends on the $a_i$'s in the same multiplicative way allows to apply \eqref{qWhit_moments_formula_W_integral} repeatedly and write the right-hand side of \eqref{qWhit_moments_formula_application_of_W} in a nested contour integral form.
\end{proof}

Along with the $q$-Whittaker measures one can define more general objects, the (\emph{ascending}) \emph{$q$-Whittaker processes}. These are probability distributions on sequences of partitions $\la^{(j)}\in\Part_j$ which depend on the same data as the $q$-Whittaker measures \eqref{qWhittaker_measure}, and are defined as 
\begin{equation}\label{qWhittaker_process}
	\WhittakerProc_{(a_1,\ldots,a_N);\rho}
	(\la^{(1)},\ldots,\la^{(N)}):=
	\frac{P_{\la^{(1)}}(a_1)P_{\la^{(2)}/\la^{(1)}}(a_2)
	\ldots
	P_{\la^{(N)}/\la^{(N-1)}}(a_N)Q_\la(\rho)}
	{\Pi_{\WhittakerMeas}(a_1,\ldots,a_N;\rho)}.
\end{equation}
The fact that these probabilities sum to one follows from \eqref{Cauchy_identity} and the branching rule for the $q$-Whittaker polynomials \cite[Ch. VI.(7.9')]{Macdonald1995}. The probability \eqref{qWhittaker_process} vanishes unless each pair of consecutive partitions $(\la^{(j)},\la^{(j+1)})$ \emph{interlaces}:
\begin{equation*}
	\la^{(j+1)}_{j+1}\le \la^{(j)}_{j}\le
	\la^{(j+1)}_{j}\le \ldots\le \la^{(j+1)}_{2}
	\le \la^{(j)}_{1}\le \la^{(j+1)}_{1}.
\end{equation*}
The marginal distribution of each single partition $\la^{(j)}$ under \eqref{qWhittaker_process} is given by the $q$-Whittaker measure $\WhittakerMeas_{(a_1,\ldots,a_j);\rho}$.

The joint $q$-moments of $\la^{(j)}_j$ under the distribution \eqref{qWhittaker_process} can be expressed as nested contour integrals similarly to \Cref{prop:qWhit_moments}. Namely, we have
\begin{proposition}[{\cite{BCGS2013}}]\label{prop:qWhit_process_moments}
	For any $k\ge1$ and $N\ge N_1\ge \ldots\ge N_k\ge1$ the expectation
	\begin{equation}\label{qWhit_process_moments_formula}
		\EE_{\WhittakerProc_{(a_1,\ldots,a_N);\rho}}
		\bigg[\prod_{j=1}^{k}q^{\la^{(N_j)}_{N_j}}\biggr]
	\end{equation}
	is given by the right-hand side of \eqref{qWhit_moments_formula} with the factor $\prod_{m=1}^{N}\frac{a_m}{a_m-z_j}$ replaced by 
	$\prod_{m=1}^{N_j}\frac{a_m}{a_m-z_j}$, $j=1,\ldots,N$.
\end{proposition}
\begin{proof}[Idea of proof]
	A key observation (first employed in \cite{BCGS2013}) is that with the help of \eqref{Whittaker_difference_operator_eigenrelation}, the expectation \eqref{qWhit_process_moments_formula} can be written as
	\begin{equation}\label{qWhit_process_moments_formula_proof}
		\EE_{\WhittakerProc_{(a_1,\ldots,a_N);\rho}}
		\bigg[\prod_{j=1}^{k}q^{\la^{(N_j)}_{N_j}}\biggr]=
		\frac{\WhittakerOp_{N_k}\ldots \WhittakerOp_{N_1}\Pi_{\WhittakerMeas}(a_1,\ldots,a_N;\rho)}{\Pi_{\WhittakerMeas}(a_1,\ldots,a_N;\rho)}.
	\end{equation}
	Note that the operators $\WhittakerOp_{N_i}$ and $\WhittakerOp_{N_j}$ do not commute for $N_i\ne N_j$, so the order of their application is important (and we first apply the operator in the largest number of variables). The right-hand side of \eqref{qWhit_process_moments_formula_proof} can be turned into a nested contour integral similarly to the proof of \Cref{prop:qWhit_moments}.
\end{proof}

We will not directly employ \Cref{prop:qWhit_process_moments} in connection with the stochastic higher spin six vertex model, but in \Cref{thm:action} below we encounter formulas similar to the right-hand side of \eqref{qWhit_process_moments_formula_proof} leading to contour integral formulas for averages of joint observables with respect to the higher spin six vertex model.

An alternative way to prove both \Cref{prop:qWhit_moments,prop:qWhit_process_moments} is by relating the $q$-Whittaker measures and processes to $q$-TASEPs, and utilizing a Markov duality approach for the latter. We outline these ideas and provide references in \Cref{sub:discrete_time_q_taseps} below.

% subsection _q_whittaker_measures (end)

\subsection{Remark. Higher order operators} % (fold)
\label{sub:higher_order_operators}

There are higher order difference operators $\MacdonaldOp_N^{(\ell)}$, $\ell\ge2$, commuting with the first Macdonald operator $\MacdonaldOp_N$ \eqref{Macdonald_operator}. When $t=0$, they turn into the operators $\WhittakerOp_{N}^{(\ell)}$ commuting with $\WhittakerOp_N$. The eigenvalues of $\MacdonaldOp_N^{(\ell)}$ and $\WhittakerOp_{N}^{(\ell)}$ on Macdonald and $q$-Whittaker polynomials are, respectively, $e_\ell(q^{\la_1}t^{N-1},\ldots,q^{\la_N})$ and $q^{\la_N+\la_{N-1}+\ldots+\la_{N-\ell+1}}$ (eigenrelations \eqref{Macdonald_eigenrelation} and \eqref{Whittaker_difference_operator_eigenrelation} correspond to $\ell=1$). Here $e_\ell$ is the $\ell$th elementary symmetric polynomial. The action of these operators can also be written as contour integrals, and so they can be utilized to compute averages of the corresponding observables (see Sections 2.2.3 and 3.1.3 in \cite{BorodinCorwin2011Macdonald} for formulas in the Macdonald and $q$-Whittaker cases, respectively). 

The higher order operators are useful in some probabilistic applications of Macdonald measures and processes (for example, see \cite{BorodinGorin2013beta}), but for our purposes the first order operators suffice. 

In \cite{borodin2016stochastic_MM} an expectation of $e_\ell(q^{\la_1}t^{N-1},\ldots,q^{\la_N})$ under a Macdonald measure with certain parameters was identified with an expectation coming from the inhomogeneous stochastic higher spin six vertex model (see Theorem 4.2 in that paper). It is not immediately clear how that result is related to ours, and we do not address this question here.

% subsection higher_order_operators (end)

% section _q_whittaker_measures_and_q_difference_operators (end)

\section{Stochastic higher spin six vertex model} % (fold)
\label{sec:stochastic_higher_spin_six_vertex_model}

Here we recall the necessary definitions and properties pertaining to the inhomogeneous stochastic higher spin six vertex model in the quadrant. This section mostly follows \cite{BorodinPetrov2016inhom}.

\subsection{Vertex weights and higher spin six vertex model} % (fold)
\label{sub:vertex_weights}

The essential ingredient in the definition of the higher spin six vertex model is the collection of vertex weights $\LLL(i_1,j_1;i_2,j_2)$, $i_1,i_2\in\Z_{\ge0}$, $j_1,j_2\in\{0,1\}$, assigned to each vertex $(N,T)$ of the quadrant $\Z_{\ge1}^{2}$ in the square grid. We interpret $i_1$ and $j_1$ as the numbers of arrows entering the vertex, respectively, from below and from the left, and $i_2$ and $j_2$ as the numbers of arrows leaving the vertex, respectively, upwards and to the right. We impose the arrow preservation property at each vertex: the weights $\LLL(i_1,j_1;i_2,j_2)$ must vanish unless $i_1+j_1=i_2+j_2$ (i.e., the number of outgoing arrows is the same as the number of incoming ones). We consider only stochastic vertex weights, i.e., those which satisfy
\begin{equation}\label{L_stochasticity_condition}
	\sum_{i_2,j_2\in\Z_{\ge0}\colon i_2+j_2=i_1+j_1}
	\LLL(i_1,j_1;i_2,j_2)=1.
\end{equation}
If, in addition, all $\LLL(i_1,j_1;i_2,j_2)\ge0$, then we interpret $\LLL(i_1,j_1;i_2,j_2)$ as a (conditional) probability that there are $i_2$ and $j_2$ arrows leaving the vertex given that there are $i_1$ and $j_1$ arrows entering the vertex.

\begin{figure}[htbp]
	\begin{tabular}{c|c|c|c|c}
	&\vertexoo{1}
	&\vertexol{1}
	&\vertexll{1}
	&\vertexlo{1}
	\\
	\hline\rule{0pt}{20pt}
	$\LLL_{u,a,\nu}$&
	$\dfrac{1-au q^{g}}{1-au}$&
	$\dfrac{-au(1-q^{g})}{1-au}$&
	$\dfrac{\nu q^{g}-au}{1-au}$&
	$\dfrac{1-\nu q^{g}}{1-au}$
	\phantom{\Bigg|}
	\end{tabular}
	\caption{Vertex weights $\LLL_{u,a,\nu}(i_1,j_1;i_2,j_2)$, where $g\in\Z_{\ge0}$ is arbitrary. Note that the weight automatically vanishes at the forbidden configuration $(0,0;-1,1)$.}
	\label{fig:vertex_weights_stoch}
\end{figure}

The concrete expressions for the vertex weights $\LLL=\LLL_{u,a,\nu}$ are given in the table in \Cref{fig:vertex_weights_stoch}. Their parameters $(u,a,\nu)=(u_T,a_N,\nu_N)$ in turn depend on the vertex location $(N,T)\in\Z_{\ge1}^{2}$, while the main parameter $q\in(0,1)$ is assumed fixed throughout the paper. The weights $\LLL_{u,a,\nu}$ always satisfy \eqref{L_stochasticity_condition}, and they are all nonnegative if, say, $u<0$, $a>0$, and $0<\nu<1$. The parameter $u$ is called the \emph{spectral parameter}.

We are now in a position to define the inhomogeneous stochastic higher spin six vertex model in the quadrant $\Z_{\ge1}^{2}$ with parameters
\begin{equation}\label{vertex_model_parameters}
	u_1,u_2,\ldots\in(-\infty,0),
	\qquad
	a_1,a_2,\ldots\in(0,+\infty),
	\qquad
	\nu_1,\nu_2,\ldots\in(0,1),
\end{equation}
such that the $a_i$'s and the $\nu_j$'s are uniformly bounded away from the endpoints of the intervals they belong to. These conditions are assumed to hold throughout the paper.

The higher spin six vertex model is a probability distribution on the set of infinite oriented up-right paths drawn in $\Z_{\ge 1}^2$, with all the paths starting from a left-to-right arrow entering at each of the points $\{(1,m)\colon m\in\Z_{\ge 1}\}$ on the left boundary (no path enters through the bottom boundary). No two paths share any horizontal piece, but common vertices and vertical pieces are allowed. The probability distribution on this set of paths can be constructed in a Markovian way. Assume that we have already defined the configuration inside the triangle $\{(N,T)\colon N+T\le n\}$, where $n\ge2$. For each vertex $(N,T)$ with $N+T=n$, we know the number of incoming arrows (from below and from the left) into this vertex. Sample, independently for each such vertex, the number of outgoing arrows according to the stochastic vertex weights $\LLL_{u_T,a_N,\nu_N}$. In this way the path configuration is now defined inside the larger triangle $\{(N,T)\colon N+T\le n+1\}$, and we can continue inductively. See \Cref{fig:HS}. Since we are using different parameters for the vertex weights at each vertex $(N,T)$, the model is inhomogeneous in both spatial directions.

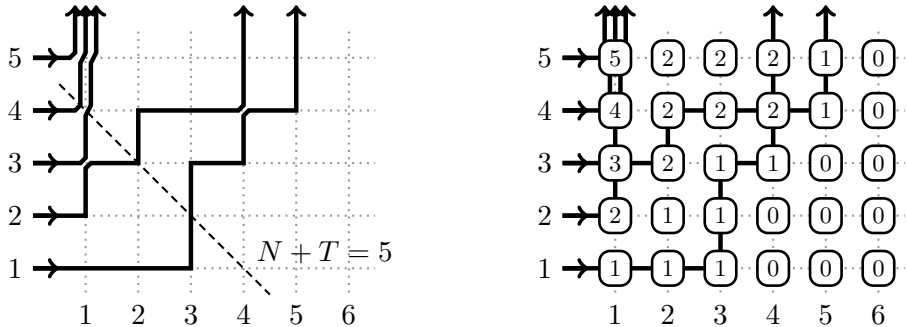
\begin{figure}[htb]
	\begin{tikzpicture}
		[scale=.7,thick]
		\def\d{.1}
		\foreach \yyyy in {1,...,6}
		{
		\draw[dotted, opacity=.4] (\yyyy-1,5.5)--++(0,-5);
		\node[below] at (\yyyy-1,.5) {$\yyyy$};
		}
		\foreach \yyyy in {1,2,3,4,5}
		{
		\draw[dotted, opacity=.4] (0,\yyyy)--++(5.5,0);
		\node[left] at (-1,\yyyy) {$\yyyy$};
		\draw[->, line width=1.7pt] (-1,\yyyy)--++(.5,0);
		}
		\draw[->, line width=1.7pt] (-1,5)--++(1-3*\d,0)--++(\d,\d)--++(0,1-\d);
		\draw[->, line width=1.7pt] (-1,4)--++(1-2*\d,0)--++(\d,\d)--++(0,1-2*\d)--++(\d,2*\d)--++(0,1-\d);
		\draw[->, line width=1.7pt] (-1,3)--++(1-\d,0)--++(\d,\d)--++(0,1-2*\d)--++(\d,2*\d)--++(0,1-2*\d)--++(\d,2*\d)--++(0,1-\d);
		\draw[->, line width=1.7pt] (-1,2)--++(1,0)--++(0,1-\d)--++(\d,\d)--++(1-\d,0)
		--++(0,1)--++(2-\d,0)--++(\d,\d)--++(0,2-\d);
		\draw[->, line width=1.7pt]
		(-1,1)--++(3,0)--++(0,2)--++(1,0)--++(0,1-\d)--++(\d,\d)--++(1-\d,0)
		--++(0,2);
		\draw[densely dashed] (-.5,4.5)--++(4,-4) node[above,anchor=west,yshift=16,xshift=-9] {$N+T=5$};
	\end{tikzpicture}
	\qquad \qquad
	\begin{tikzpicture}
		[scale=.7,thick]
		\tikzstyle{myblock} = [draw=black, fill=white, line width=1, minimum width=1em, minimum height=1em, rectangle, rounded corners, text centered]
		\def\d{.1}
		\foreach \yyyy in {1,...,6}
		{
		\draw[dotted, opacity=.4] (\yyyy-1,5.5)--++(0,-5);
		\node[below] at (\yyyy-1,.5) {$\yyyy$};
		}
		\foreach \yyyy in {1,2,3,4,5}
		{
		\draw[dotted, opacity=.4] (0,\yyyy)--++(5.5,0);
		\node[left] at (-1,\yyyy) {$\yyyy$};
		\draw[->, line width=1.7pt] (-1,\yyyy)--++(.5,0);
		}
		\draw[->, line width=1.7pt] (-1,5)--++(1-3*\d,0)--++(\d,\d)--++(0,1-\d);
		\draw[->, line width=1.7pt] (-1,4)--++(1-2*\d,0)--++(\d,\d)--++(0,1-2*\d)--++(\d,2*\d)--++(0,1-\d);
		\draw[->, line width=1.7pt] (-1,3)--++(1-\d,0)--++(\d,\d)--++(0,1-2*\d)--++(\d,2*\d)--++(0,1-2*\d)--++(\d,2*\d)--++(0,1-\d);
		\draw[->, line width=1.7pt] (-1,2)--++(1,0)--++(0,1-\d)--++(\d,\d)--++(1-\d,0)
		--++(0,1)--++(2-\d,0)--++(\d,\d)--++(0,2-\d);
		\draw[->, line width=1.7pt]
		(-1,1)--++(3,0)--++(0,2)--++(1,0)--++(0,1-\d)--++(\d,\d)--++(1-\d,0)
		--++(0,2);
		\node[myblock] at (0,1) {\footnotesize{}1};
		\node[myblock] at (1,1) {\footnotesize{}1};
		\node[myblock] at (2,1) {\footnotesize{}1};
		\node[myblock] at (3,1) {\footnotesize{}0};
		\node[myblock] at (4,1) {\footnotesize{}0};
		\node[myblock] at (5,1) {\footnotesize{}0};
		\node[myblock] at (0,2) {\footnotesize{}2};
		\node[myblock] at (1,2) {\footnotesize{}1};
		\node[myblock] at (2,2) {\footnotesize{}1};
		\node[myblock] at (3,2) {\footnotesize{}0};
		\node[myblock] at (4,2) {\footnotesize{}0};
		\node[myblock] at (5,2) {\footnotesize{}0};
		\node[myblock] at (0,3) {\footnotesize{}3};
		\node[myblock] at (1,3) {\footnotesize{}2};
		\node[myblock] at (2,3) {\footnotesize{}1};
		\node[myblock] at (3,3) {\footnotesize{}1};
		\node[myblock] at (4,3) {\footnotesize{}0};
		\node[myblock] at (5,3) {\footnotesize{}0};
		\node[myblock] at (0,4) {\footnotesize{}4};
		\node[myblock] at (1,4) {\footnotesize{}2};
		\node[myblock] at (2,4) {\footnotesize{}2};
		\node[myblock] at (3,4) {\footnotesize{}2};
		\node[myblock] at (4,4) {\footnotesize{}1};
		\node[myblock] at (5,4) {\footnotesize{}0};
		\node[myblock] at (0,5) {\footnotesize{}5};
		\node[myblock] at (1,5) {\footnotesize{}2};
		\node[myblock] at (2,5) {\footnotesize{}2};
		\node[myblock] at (3,5) {\footnotesize{}2};
		\node[myblock] at (4,5) {\footnotesize{}1};
		\node[myblock] at (5,5) {\footnotesize{}0};
	\end{tikzpicture}
	\caption{Stochastic higher spin six vertex model in $\Z_{\ge1}^{2}$. The paths crossing the horizontal line at $4+\frac12$ are encoded by the partition $\mu^{(4)}=(5,4,1,1)$. The corresponding height function \eqref{height_function_def} at each vertex is given on the right.}
	\label{fig:HS}
\end{figure}

\begin{lemma}\label{lemma:convergence_conditions_on_path_conf}
	For each $T=1,2,\ldots$, almost surely there are exactly $T$ paths crossing the horizontal line at $T+\frac12$. In other words, paths cannot stay horizontal forever. 
\end{lemma}
\begin{proof}
	The probability $\LLL_{u_T,a_N,\nu_N}(0,1;0,1)=(\nu_N-a_Nu_T)/(1-a_Nu_T)$ that a path goes right is uniformly bounded away from $1$ since the $a_i$'s and the $\nu_i$'s are uniformly bounded from the endpoints of their corresponding intervals. Thus, an infinite product (over $N$) of $\LLL_{u_T,a_N,\nu_N}(0,1;0,1)$ is zero, which means that a path almost surely cannot stay horizontal forever.
\end{proof}

\Cref{lemma:convergence_conditions_on_path_conf} implies that the configuration of the stochastic higher spin six vertex model can be encoded by a sequence of partitions $\{\mu^{(T)}\}_{T\ge1}$, where $\mu^{(T)}\in\Partt_T:=\{\mu\in\Part_T\colon \mu_T\ge1\}$. Here $\mu^{(T)}_1\ge \mu^{(T)}_2\ge \ldots\ge\mu^{(T)}_T\ge1$ are $T$ intersections of the up-right paths with the horizontal line at $T+\frac12$, cf. \Cref{fig:HS}. 
\begin{remark}
	The transition from $\mu^{(T)}\in\Partt_{T}$ to $\mu^{(T+1)}\in\Partt_{T+1}$ can be interpreted as one move of a Markov chain with left-to-right update, see \cite[Section 6.4.2]{BorodinPetrov2016inhom}.
\end{remark}

\begin{remark}\label{rmk:matching_with_S_XI_parameters}
	The parameters $(u,a,\nu)$ we use here are related to the ones in \cite{BorodinPetrov2016inhom} as $a=-\upxi\SP$, $\nu=\SP^{2}$, plus a change of sign of the parameters $u_i$. Our choice here turns out to be convenient in connection with difference operators and $q$-Whittaker measures. 
\end{remark}

Conditions \eqref{vertex_model_parameters} ensuring the nonnegativity of the vertex weights are sufficient but not necessary. Other values of parameters can also lead to nonnegative weights (cf. \cite[Proposition 2.3]{CorwinPetrov2015}, where $a$ is denoted by $\al$). Most notably, setting $\nu=1/q$ leads to the classical stochastic six vertex model which was introduced in \cite{GwaSpohn1992} and whose asymptotic behavior was studied in \cite{BCG6V} (see also \cite{AmolBorodin2016Phase}, \cite{Amol2016Stationary}, \cite{borodin2016stochastic_MM}). In this case the number of vertical arrows per one edge of the square grid is at most one. In the sequel the parameters are always assumed to satisfy conditions \eqref{vertex_model_parameters}.

% subsection vertex_weights (end)

\subsection{Step and step Bernoulli boundary conditions} % (fold)
\label{sub:step_and_bernoulli_boundary_conditions}

The boundary condition in the stochastic higher spin six vertex model described in \Cref{sub:vertex_weights} is that the paths start at every edge on the left boundary of the quadrant, and no arrows are incoming at the bottom boundary. This boundary condition will be referred to as the \emph{step boundary condition}.

We also consider the \emph{step Bernoulli boundary condition} obtained by specializing $\nu_1=0$. Under this condition still no arrows are incoming from below, and independently at each edge from $(1,T)$ to $(2,T)$ on the left boundary there is a right-pointing arrow (starting a path) with probability $-a_1u_T/(1-a_1u_T)\in(0,1)$ or no arrow with the complementary probability $1/(1-a_1u_T)$. The description of this specialization readily follows by considering the vertex weights with $\nu=0$.

The higher spin six vertex model with the step boundary condition in $\Z_{\ge1}^{2}$ thus gives rise to the model with the step Bernoulli boundary condition in $\Z_{\ge2}\times\Z_{\ge1}$. It is convenient not to shift the horizontal axis, and keep the latter model in the quadrant $\Z_{\ge2}\times\Z_{\ge1}$.

Another way to obtain the step Bernoulli boundary condition from the step boundary condition is to set $a_1=-\nu_1$, specialize the first $K$ parameters $u_i$ into a suitable geometric progression with ratio $q$, and take $K\to+\infty$. See \cite[Sections 6.6.2 and 10.2]{BorodinPetrov2016inhom}.

A generalization of the step Bernoulli boundary condition (for the stochastic six vertex model, i.e., with $\nu_j\equiv 1/q$) corresponding to setting the first several parameters $\nu_i$ to zero was introduced and studied recently in \cite{AmolBorodin2016Phase}. We briefly discuss this boundary condition in \Cref{sub:more_general_boundary_conditions} below.

% subsection step_and_bernoulli_boundary_conditions (end)

\subsection{$q$-moments of the height function} % (fold)
\label{sub:_q_moments_of_the_height_function}

The vertex weights described in \Cref{sub:vertex_weights} have a very special property which justifies their definition: they satisfy (a version of) the Yang--Baxter equation. We do not reproduce it here and refer to \cite{Mangazeev2014}, \cite{Borodin2014vertex}, \cite{CorwinPetrov2015}, \cite{BorodinPetrov2016_Hom_Lectures}, \cite{BorodinPetrov2016inhom} for details in the context of $U_q(\widehat{\mathfrak{sl}_2})$ vertex models, and to \cite{baxter2007exactly} for a general background. The Yang--Baxter equation is a key tool used in \cite{BorodinPetrov2016inhom} to compute averages of certain observables of the higher spin six vertex model in a contour integral form. Here we recall one such formula for the $q$-moments of the height function.

The \emph{height function} $\HeightFunction(N,T)$, $(N,T)\in\Z_{\ge1}^{2}$, of the stochastic higher spin six vertex model is defined as 
\begin{equation}\label{height_function_def}
	\begin{split}
		\HeightFunction(N,T)
		&:=
		\#
		\{
		\textnormal{up-right paths crossing
		the horizontal line at $T+\tfrac12$ at points $\ge N$}
		\}
		\\&\phantom{:}=\#\{i\colon \mu^{(T)}_i\ge N\},
	\end{split}
\end{equation}
where $\mu^{(T)}$ is defined after \Cref{lemma:convergence_conditions_on_path_conf}. In particular, $\HeightFunction(+\infty,T)=0$, and for the step boundary condition we have $\HeightFunction(1,T)=T$. For fixed $T$, $\HeightFunction(N,T)$ is a nonincreasing function of $N$. Since the configuration of up-right paths in the quadrant is random, $\{\HeightFunction(N,T)\}$ is a collection of random variables indexed by points of the quadrant. See \Cref{fig:HS}, right, for an illustration.

Assume that in addition to \eqref{vertex_model_parameters} our parameters satisfy
\begin{equation}\label{7.4_conditions}
	\sup_{i\ge1}\Bigl\{\frac{\nu_i}{a_i}\Bigr\}<\inf_{i\ge1}
	\Bigl\{\frac{\nu_i}{qa_i},\frac{1}{a_i}\Bigr\},
	\qquad
	\textnormal{$u_i\ne q u_j$ for any $i,j=1,\ldots,T$}.
\end{equation}
\begin{theorem}[{\cite[Theorem 9.8]{BorodinPetrov2016inhom}}]
	\label{thm:obaservables_Thm9_8}
	Under assumptions \eqref{vertex_model_parameters} and \eqref{7.4_conditions},  for any $\ell\ge1$ and $N_1\ge \ldots N_\ell\ge 0$ we have
	\begin{multline}\label{obaservables_Thm9_8}
		\EE_{\mathbf{HS}}^{\textnormal{step}}
		\bigg[\prod_{i=1}^{\ell}q^{\HeightFunction(N_i+1,T)}\biggr]=
		q^{\frac{\ell(\ell-1)}2}
		\oint\limits_{\cont[\mathbf{u}^{-1}\md1]}\frac{d w_1}{2\pi\i}
		\ldots
		\oint\limits_{\cont[\mathbf{u}^{-1}\md\ell]}
		\frac{d w_\ell}{2\pi\i}
		\prod_{1\le \aind<\bind\le \ell}
		\frac{w_\aind-w_\bind}{w_\aind-qw_\bind}
		\\\times
		\prod_{i=1}^{\ell}\bigg(
		w_i^{-1}
		\prod_{j=1}^{N_i}
		\frac{a_j-\nu_jw_i}{a_j-w_i}
		\prod_{j=1}^{T}\frac{1-qu_jw_i}{1-u_jw_i}
		\bigg),
	\end{multline}
	where $\EE_{\mathbf{HS}}^{\textnormal{step}}$ stands for the expectation with respect to the higher spin six vertex model with the step boundary condition. The integration contours encircle all points of $\mathbf{u}^{-1}:=\{u_1^{-1},\ldots,u_T^{-1}\}$ and $0$, but not $q^{\pm k}u_i$, $k\ne0$, $a_i$, or $a_i/\nu_i$. Moreover, parts of the contours are nested around zero in such a way that shrinking them to zero in the order $w_1,\ldots,w_\ell$ does not produce any residues at $w_\aind=q w_\bind$, $\aind<\bind$. See \Cref{fig:zero_thing_contours}.
\end{theorem}

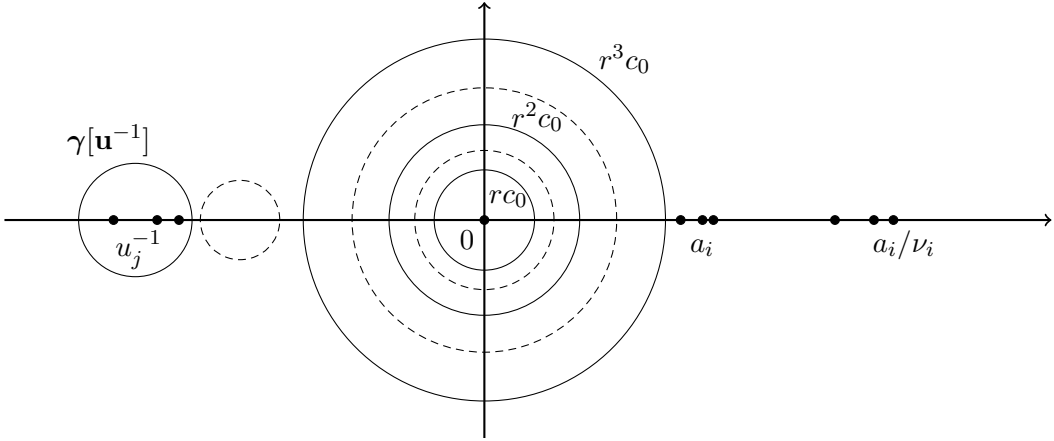
\begin{figure}[htbp]
	\begin{center}
	\begin{tikzpicture}
		[scale=2.9]
		\def\pt{0.02}
		\def\q{.7}
		\def\ss{.56}
		\draw[->, thick] (-2.2,0) -- (2.6,0);
	  	\draw[->, thick] (0,-1) -- (0,1);
	  	\draw[fill] (.9,0) circle (\pt);
	  	\draw[fill] (1,0) circle (\pt) node [below, xshift=7pt]
	  	{$a_i\phantom{/\nu_i}$};
	  	\draw[fill] (1.05,0) circle (\pt) ;
	  	\draw[fill] (.9/\ss,0) circle (\pt);
	  	\draw[fill] (1/\ss,0) circle (\pt);
	  	\draw[fill] (-1.5,0) circle (\pt) node [below, xshift=-7pt] {$u_j^{-1}$};
	  	\draw[fill] (-1.7,0) circle (\pt) ;
	  	\draw[fill] (-1.4,0) circle (\pt) ;
	  	\draw[fill] (1.05/\ss,0) circle (\pt) node [below, xshift=4pt]
	  	{$a_i/\nu_i$};
	  	\draw[fill] (0,0) circle (\pt) node [below left] {$0$};
	  	\draw (0,0) circle (.23) node [xshift=9,yshift=8] {$rc_0$};
	  	\def\rrrrr{1.9}
	  	\draw (0,0) circle (.23*\rrrrr) node [xshift=20,yshift=39] {$r^{2}c_0$};
	  	\draw[densely dashed] (0,0) circle (.24*\rrrrr*\q);
	  	\draw (0,0) circle (.23*\rrrrr*\rrrrr) node [xshift=53,yshift=61] {$r^{3}c_0$};
	  	\draw[densely dashed] (0,0) circle (.24*\rrrrr*\rrrrr*\q);
	  	\draw (-1.6,0) circle (.26) node [below,xshift=-10,yshift=40]
	  	{$\cont[\mathbf{u}^{-1}]$};
	  	\draw[densely dashed] (-1.6*\q,0) circle (.26*\q);
	\end{tikzpicture}
	\end{center}
  	\caption{The contours $\cont[\mathbf{u}^{-1}\md1]=\cont[\mathbf{u}^{-1}]\cup rc_0$, $\cont[\mathbf{u}^{-1}\md2]=\cont[\mathbf{u}^{-1}]\cup r^{2}c_0$, and $\cont[\mathbf{u}^{-1}\md3]=\cont[\mathbf{u}^{-1}]\cup r^{3}c_0$ for $\ell=3$ in \Cref{thm:obaservables_Thm9_8}. Here $c_0$ is a small circle around zero, $r>q^{-1}$, and $\cont[\mathbf{u}^{-1}]$ is a contour encircling all the $u_j^{-1}$'s. Contours $q\cont[\mathbf{u}^{-1}]$, $q\cont[\mathbf{u}^{-1}\md2]$, and $q\cont[\mathbf{u}^{-1}\md3]$ are shown dashed.}
  	\label{fig:zero_thing_contours}
\end{figure}

Expectations in the left-hand side of \eqref{obaservables_Thm9_8} will be referred to as the (\emph{joint}) \emph{$q$-moments of the height function} of the higher spin six vertex model. In \Cref{sec:action_of_q_whittaker_difference_operators_on_the_higher_spin_six_vertex_model,sec:capling_with_q_tasep} we give two new independent proofs of \Cref{thm:obaservables_Thm9_8} (the second one in a particular case of the step Bernoulli boundary condition and single-point $q$-moments).

% subsection _q_moments_of_the_height_function (end)

\subsection{Horizontal probability distributions} % (fold)
\label{sub:horizontal_probability_distributions}

The distribution of the random partition $\mu^{(T)}\in\Partt_T$ describing the stochastic higher spin six vertex model at the $T$th horizontal can be written down explicitly. It is also based on the Yang--Baxter equation, and can be viewed as an instance of turning algebraic Bethe ansatz expressions for eigenfunctions of higher spin transfer matrices into coordinate form. We refer to \cite[Section 4.5]{BorodinPetrov2016inhom} for further details. The probabilistic interpretation in the following proposition is based on \cite[Section 6.4]{BorodinPetrov2016inhom}:
\begin{proposition}\label{prop:probabilities_symmetrization_formula}
	Under assumptions \eqref{vertex_model_parameters}, for any $T\ge1$ the distribution of the random partition $\mu^{(T)}\in\Partt_T$ corresponding to the higher spin six vertex model with the step boundary condition is given by 
	\begin{equation*}
		\prob(\mu^{(T)}=\ka)=
		\FProb_\ka(u_1,\ldots,u_T\md\mathbf{a},\boldsymbol\nu),
		\qquad\ka\in\Partt_T
		,
	\end{equation*}
	where
	\begin{equation}\label{F_probabilistic_symmetrization_formula}
		\FProb_\ka(u_1,\ldots,u_T\md\mathbf{a},\boldsymbol\nu):=
		\prod_{r\ge1}\frac{(\nu_r;q)_{k_r}}{(q;q)_{k_r}}
		\sum_{\sigma\in S_T}\sigma\bigg(
		\prod_{1\le \aind<\bind\le T}\frac{u_{\aind}-qu_{\bind}}{u_{\aind}-u_{\bind}}
		\prod_{i=1}^{T}\frac{1-q}{1-a_{\ka_i}u_{i}}
		\prod_{j=1}^{\ka_i-1}\frac{\nu_j-a_ju_i}{1-a_ju_i}
		\bigg).
	\end{equation}
	Here the partition $\ka$ is written in the multiplicative notation $\ka=(1^{k_1}2^{k_2}\ldots)$ (i.e., it has $k_1\ge0$ parts equal to $1$, $k_2\ge0$ parts equal to $2$, etc.), and the permutation $\sigma$ acts by permuting $u_1,\ldots,u_T$ and not $\ka_1,\ldots,\ka_T$.
\end{proposition}
In particular, this implies that for the quantities given by \eqref{F_probabilistic_symmetrization_formula} we have
\begin{equation}\label{F_stoch_sum_to_one}
	\sum\nolimits_{\ka\in\Partt_T}
	\FProb_\ka(u_1,\ldots,u_T\md\mathbf{a},\boldsymbol\nu)
	=1
\end{equation}
(conditions \eqref{vertex_model_parameters} ensure the convergence of the series). This identity can be viewed as a specialization of a Cauchy-type summation formula involving symmetric rational functions $\FProb_\ka$. The functions \eqref{F_probabilistic_symmetrization_formula} and this Cauchy-type summation identity are central to the proof of \Cref{thm:obaservables_Thm9_8} given in \cite{BorodinPetrov2016inhom}. We will use \Cref{prop:probabilities_symmetrization_formula} in our computations in \Cref{sec:action_of_q_whittaker_difference_operators_on_the_higher_spin_six_vertex_model}.

\begin{remark}\label{rmk:HL_connection}
	To avoid confusion, for this remark only rename the parameter $q$ of the higher spin six vertex model by $Q$. 

	In a limit of the parameters $a_i,\nu_i$ which amounts to setting the $\SP_i$'s (see the end of \Cref{sub:vertex_weights}) to zero plus a straightforward rescaling, the functions \eqref{F_probabilistic_symmetrization_formula} become the Hall--Littlewood symmetric polynomials in the $u_i$'s. The latter are $q=0$ degenerations of the Macdonald polynomials depending on the second Macdonald parameter $t$. This parameter $t$ in the Hall--Littlewood polynomials coming out of \eqref{F_probabilistic_symmetrization_formula} should be taken equal to $Q$.
	% In \Cref{sec:action_of_q_whittaker_difference_operators_on_the_higher_spin_six_vertex_model} we will see that the action of the $Q$-Whittaker difference operators (corresponding to another degeneration of the Macdonald objects) in the parameters $a_i$ on combinations of probabilities \eqref{F_probabilistic_symmetrization_formula} is nice.
\end{remark}

% subsection horizontal_probability_distributions (end)

% section stochastic_higher_spin_six_vertex_model (end)

\section{Action of $q$-Whittaker difference operators} % (fold)
\label{sec:action_of_q_whittaker_difference_operators_on_the_higher_spin_six_vertex_model}

In this section we show that the action of (conjugations of) the $q$-Whittaker operators in the variables $a_i$ on certain linear combinations of the horizontal probabilities \eqref{F_probabilistic_symmetrization_formula} has a nice form (see \Cref{lemma:action_is_nice_1} and \Cref{lemma:action_is_nice}). The idea to apply difference operators to probability weights resembles the treatment of the $q$-Whittaker measures by difference operators (\Cref{prop:qWhit_moments,prop:qWhit_process_moments}), but for the higher spin six vertex model we need to combine probabilities in a specific way to get the results (cf. \Cref{rmk:bad_examples}).

These computations lead to our first new proof of the $q$-moment formula for the higher spin six vertex model (\Cref{thm:obaservables_Thm9_8}), see \Cref{sub:applications_to_q_moments_and_q_laplace_transforms}. This proof works for joint moments and for the step boundary condition, which is the most general situation considered in \cite{BorodinPetrov2016inhom}.

\subsection{Computation with difference operators} % (fold)
\label{sub:computations_of_the_action}

Let $\nua_i:=\nu_i/a_i$, and in the rest of the paper in addition to \eqref{vertex_model_parameters} assume that 
\begin{equation}\label{condition_on_parameters_for_Whittaker_link}
	\textnormal{$a_i\nua_j<1$\quad{}for all\quad{}$i,j$}.
\end{equation}
Define
\begin{equation}\label{Phi_N_conjugation_factor}
	\Pi_N:=
	\Pi_{\WhittakerMeas}(a_1,\ldots,a_N;\nua_1,\ldots,\nua_N)
	=
	\prod_{i,j=1}^{N}\frac1{(a_i\nua_j;q)_{\infty}},
\end{equation}
see \eqref{Pi_definition}, \eqref{Cauchy_identity}, where as the second specialization $\rho$ we take the specialization into finitely many $\al$-parameters $\nua_1,\ldots,\nua_N$. Let $\HSOp_N$ be the following conjugation of the $q$-Whittaker difference operator \eqref{Whittaker_difference_operator} acting in $a_1,\ldots,a_N$:
\begin{equation}\label{conjugated_operators_DN_HSOp_N}
	\HSOp_N:=\Pi_N^{-1}\WhittakerOp_N\Pi_N
	=\sum_{r=1}^{N}
	\Bigg(
	\prod_{j=1}^{N}(1-\nua_j a_r)
	\prod_{1\le i\le N,\,i\ne r}\frac{a_i}{a_i-a_r}
	\Biggr)
	\ShiftOp_{q;a_r}.
\end{equation}
Passing from the variables $(a_i,\nua_i)$ to $(a_i,\nu_i)$, we can rewrite \eqref{conjugated_operators_DN_HSOp_N} as 
\begin{equation}\label{conjugated_operators_DN_HSOp_N_a_nu}
	\HSOp_N=\sum_{r=1}^{N}(1-\nu_r)\Bigg(
	\prod_{1\le i\le N,\,i\ne r}
	\frac{a_i-\nu_ia_r}{a_i-a_r}\Biggr)
	\ShiftOp_{q;a_r}
	\ShiftOp_{q;\nu_r}.
\end{equation}
Note that while the operator $\HSOp_N$ written as \eqref{conjugated_operators_DN_HSOp_N_a_nu} acts on the $a_i$'s and the $\nu_i$'s, it does not affect their combinations $\nua_i=\nu_i/a_i$. Formula \eqref{conjugated_operators_DN_HSOp_N_a_nu} turns out to be very useful in computations of this subsection. In particular, it implies that for any $u$, 
\begin{equation}\label{simple_q_commutation}
	\HSOp_N\bigg(\prod_{i=1}^{N}(\nu_i-a_i u)\biggr)
	=q\bigg(\prod_{i=1}^{N}(\nu_i-a_i u)\biggr)\HSOp_N.
\end{equation}
\begin{remark}\label{rmk:Macdonald_connection}
	From \eqref{conjugated_operators_DN_HSOp_N_a_nu} see that for any function $G$ depending on $a_1,\ldots,a_N$ but not on $\nu_1,\ldots,\nu_N$,
	\begin{equation}\label{relation_with_Macdonald_operators}
		\bigl(\HSOp_N G\bigr)\big\vert_{\nu_1=\ldots=\nu_N=t}=(1-t)\MacdonaldOp_NG,
	\end{equation}
	where $\MacdonaldOp_N$ is the usual Macdonald difference operator \eqref{Macdonald_operator} in the $a_i$'s. Note that \eqref{relation_with_Macdonald_operators} does not extend to a relation between powers $(\HSOp_N)^{k}$ and $(\MacdonaldOp_N)^k$, $k\ge2$, because $\HSOp_NG$ already depends on $\nu_1,\ldots,\nu_N$ even if $G$ did not. However,
	\begin{equation*}
		\HSOp_N \biggl(\ldots\Bigl(\HSOp_N \bigl(\HSOp_NG\bigr)\big\vert_{\nu_i\equiv t}
		\Bigr)\Big\vert_{\nu_i\equiv t} \ldots \biggr)\bigg\vert_{\nu_i\equiv t}=
		(1-t)^{k}(\MacdonaldOp_N)^{k}G,
	\end{equation*}
	where $\HSOp_N$ on the left is applied $k$ times and, by agreement, does not affect $t$.
\end{remark}

\begin{lemma}\label{lemma:D_N_first_order_as_contour}
	Assume that the $a_i$'s are nonzero real numbers which are either all positive or all negative. Let the function $f(w)$, $w\in\C$, be holomorphic and nonzero in a complex neighborhood of an interval in $\R$ containing $\{a_i,qa_i\}_{i=1}^{N}$. Let $G(a_1,\ldots,a_N):=f(a_1)\ldots f(a_N)$. Then
	\begin{equation}\label{D_N_first_order_as_contour}
		\frac{\HSOp_N G(a_1,\ldots,a_N)}{G(a_1,\ldots,a_N)}
		=-\frac{1}{2\pi\i}\oint
		\bigg(\prod_{i=1}^{N}\frac{a_i-\nu_i z}{a_i-z}\biggr)
		\frac{f(qz)}{f(z)}\frac{dz}{z},
	\end{equation}
	where the integration contour is positively oriented and encircles $a_1,\ldots,a_N$ but no other singularities of the integrand.
\end{lemma}
\begin{proof}
	Follows from the contour integral formula \eqref{qWhit_moments_formula_W_integral} for the $q$-Whittaker difference operator, or alternatively by interpreting \eqref{conjugated_operators_DN_HSOp_N_a_nu} as a residue expansion.
\end{proof}

\begin{lemma}\label{lemma:action_on_1}
	We have
	\begin{equation}\label{action_on_1}
		\HSOp_N1=1-\nu_1 \ldots\nu_N,
	\end{equation}
	and, moreover, for any $\ell$ and $N_1\ge \ldots\ge N_\ell\ge 1$ we have
	\begin{equation}\label{action_on_1_multipoint}
		\HSOp_{N_\ell}\ldots\HSOp_{N_2}\HSOp_{N_1} 1=
		\prod_{j=1}^{\ell}(1-q^{\ell-j}\nu_1 \ldots\nu_{N_j}).
	\end{equation}
	Here by $1$ we mean the constant function.
\end{lemma}
Note that the order of application of the $\HSOp_{N_i}$'s in \eqref{action_on_1_multipoint} is important.
\begin{proof}[Proof of \Cref{lemma:action_on_1}]
	First, observe that \eqref{action_on_1_multipoint} follows from \eqref{action_on_1} with the help of \eqref{simple_q_commutation}. To establish \eqref{action_on_1}, set $f(z)\equiv 1$ in \eqref{lemma:D_N_first_order_as_contour}, and note that the only singularities outside the integration contour are $z=0$ and $z=\infty$. Taking minus residues at these points, we obtain \eqref{action_on_1}.
\end{proof}

\begin{remark}
	In light of \Cref{rmk:Macdonald_connection}, observe that \eqref{action_on_1} generalizes {$\MacdonaldOp_N1=(1-t^N)/(1-t)$} (the latter identity is a consequence of \eqref{Macdonald_eigenrelation} and the fact that $1=P_{(0,0,\ldots,0)}(\cdots\md q,t)$ is a Macdonald polynomial).
\end{remark}

We will employ the following normalizations of the horizontal probabilities $\FProb_\ka$ \eqref{F_probabilistic_symmetrization_formula}:
\begin{align}
	\Phi_M(u_1,\ldots,u_T)&:=\prod_{i=1}^{T}\prod_{j=1}^{M}(1-a_ju_i),
	\label{Phi_M}
	\\
	\label{F_tilde}
	\FDenominators{M}_\ka
	(u_1,\ldots,u_{T}\md \mathbf{a},\boldsymbol\nu)&:=
	\Phi_M(u_1,\ldots,u_T)
	\FProb_\ka(u_1,\ldots,u_{T}\md \mathbf{a},\boldsymbol\nu),
\end{align}
where $\ka\in\Partt_T$, and $M\ge1$ is arbitrary. Observe that multiplying by $\Phi_M$ clears all the denominators in $\FProb_\ka$ involving $a_1,\ldots,a_M$. Therefore, $\FDenominators{M}_\ka$ is a polynomial in $\{a_i,\nu_i\}$ if $M\ge\ka_1$. This normalization of the symmetric rational functions $\FProb_\ka$ is similar to \cite[(30)--(31)]{deGierWheeler2016}.

\begin{lemma}\label{lemma:action_on_all_bigger_than_N}
	Fix $N\ge1$. For any $T\ge1$, $\ka\in\Partt_T$ with $\ka_T>N$, and any $M\ge N$ we have
	\begin{equation}\label{action_on_far_F}
		\HSOp_N \FDenominators{M}_\ka
		(u_1,\ldots,u_{T}\md \mathbf{a},\boldsymbol\nu)=
		q^{T}(1-\nu_1 \ldots\nu_N)
		\FDenominators{M}_\ka
		(u_1,\ldots,u_{T}\md \mathbf{a},\boldsymbol\nu).
	\end{equation}
	Moreover, for any $\ell$ and $N_1\ge \ldots\ge N_\ell\ge1$, $N_1\le N$, we have
	\begin{equation}\label{action_on_far_F_multilevel}
		\HSOp_{N_\ell}\ldots \HSOp_{N_1}
		\FDenominators{M}_\ka
		(u_1,\ldots,u_{T}\md \mathbf{a},\boldsymbol\nu)=
		q^{\ell T}
		\FDenominators{M}_\ka
		(u_1,\ldots,u_{T}\md \mathbf{a},\boldsymbol\nu)
		\cdot
		(\HSOp_{N_\ell}\ldots \HSOp_{N_1}1).
	\end{equation}
\end{lemma}
\begin{proof}
	The first identity immediately follows from \eqref{F_probabilistic_symmetrization_formula} because the prefactor (the product over $r\ge1$) does not contain $\nu_1,\ldots,\nu_N$, and in the sum over $\sigma\in S_T$ all summands are multiplied by $q$ under the action of $\HSOp_N$. The second identity follows by a repeated application of the first one similarly to the proof of \Cref{lemma:action_on_1}.
\end{proof}
The next statement is a key computation leading to the main results of this section.
\begin{lemma}\label{lemma:action_is_nice_1}
	For any fixed
	$N\ge1$, $T\ge0$, and $M\ge N$ we have
	\begin{equation}\label{action_is_nice_1}
		\HSOp_N\sum_{\ka\in\Partt_T\colon \ka_1\le N}
		\FDenominators{M}_\ka(u_1,\ldots,u_T\md
		\mathbf{a},\boldsymbol\nu)=
		(1-q^{T}\nu_1 \ldots\nu_N)\sum_{\ka\in\Partt_T\colon \ka_1\le N}
		\FDenominators{M}_\ka(u_1,\ldots,u_T\md
		\mathbf{a},\boldsymbol\nu).
	\end{equation}
\end{lemma}
\begin{proof}
	Denote
	\begin{equation*}
		\SSS_{T}(u_1,\ldots,u_T):=
		\sum_{\ka\in\Partt_T\colon \ka_1\le N}
		\FProb_\ka(u_1,\ldots,u_T\md\mathbf{a},\boldsymbol\nu).
	\end{equation*}
	This can be interpreted as the probability $\prob\bigl(\HeightFunction(N+1,T)=0\bigr)$ in the higher spin six vertex model with the step boundary condition (by agreement, $\SSS_0\equiv 1$). We aim to express $\SSS_T(u_1,\ldots,u_T)$ through $\SSS_0,\SSS_1,\ldots,\SSS_{T-1}$, and then argue by induction. The induction base is the case of $\SSS_0$, which is \Cref{lemma:action_on_1}.

	First, from \eqref{F_stoch_sum_to_one} and the form of the vertex weights (cf. \Cref{fig:vertex_weights_stoch}) we see that
	\begin{equation*}
		\sum_{\ka\in\Partt_T\colon \ka_1\le N+1}
		\FProb_\ka(u_1,\ldots,u_T\md\mathbf{a},\boldsymbol\nu)
		\big\vert_{a_{N+1}=\nu_{N+1}=0}
		=1,
	\end{equation*}
	because for $a_{N+1}=\nu_{N+1}=0$ no up-right path can go past the $(N+1)$st vertical column. On the other hand, the above sum contains $\SSS_T(u_1,\ldots,u_T)$ because setting $a_{N+1}=\nu_{N+1}=0$ does not affect $\FProb_\ka$ with $\ka_1\le N$. We thus get
	\begin{equation}\label{action_is_nice_1_proof0}
		\SSS_T(u_1,\ldots,u_T)=1-
		\sum_{\ka\in\Partt_T\colon \ka_1= N+1}
		\FProb_\ka(u_1,\ldots,u_T\md\mathbf{a},\boldsymbol\nu)\big\vert_{a_{N+1}=\nu_{N+1}=0}.
	\end{equation}

	Let $\ka$ in the above sum have $\ka_1=\ldots=\ka_m=N+1$, $\ka_{m+1}<N+1$ for some $m=1,\ldots,T$. Then we can write
	\begin{multline}
		\FProb_\ka(u_1,\ldots,u_T\md\mathbf{a},\boldsymbol\nu)
		\big\vert_{a_{N+1}=\nu_{N+1}=0}
		=
		\frac{(1-q)^{m}}{(q;q)_m}
		\sum_{\sigma\in S_T}
		\bigg(\prod_{i=1}^{m}\prod_{j=1}^{N}
		\frac{\nu_j-a_ju_{\sigma(i)}}{1-a_ju_{\sigma(i)}}\biggr)
		\\\times
		\prod_{r=1}^{N}\frac{(\nu_r;q)_{k_r}}{(q;q)_{k_r}}
		\prod_{1\le \aind<\bind\le T}\frac{u_{\sigma(\aind)}-qu_{\sigma(\bind)}}{u_{\sigma(\aind)}-u_{\sigma(\bind)}}
		\bigg(
		\prod_{i=m+1}^{T}\frac{1-q}{1-a_{\ka_i}u_{\sigma(i)}}
		\prod_{j=1}^{\ka_i-1}\frac{\nu_j-a_ju_{\sigma(i)}}{1-a_ju_{\sigma(i)}}
		\biggr)
		,
		\label{action_is_nice_1_proof}
	\end{multline}
	where $\ka=(1^{k_1}2^{k_2}\ldots)$ in the multiplicative notation. We claim that the sum of the above expressions over all such $\ka$ (with $\ka_m=N+1$ and $\ka_{m+1}<N$) is equal to
	\begin{equation}
		\sum_{I\subseteq\{1,\ldots,T\}\colon|I|=m}
		\,\prod_{i\in I}
		\bigg(\prod_{j=1}^{N}
		\frac{\nu_j-a_ju_i}{1-a_ju_i}
		\prod_{\bind\notin I}
		\frac{u_{i}-q u_{\bind}}{u_{i}-u_{\bind}}
		\bigg)
		\SSS_{T-m}(u_p\colon p\notin I).
		\label{action_is_nice_1_proof2}
	\end{equation}
	Indeed, this is seen by splitting the product over $\aind<\bind$ into three products $\bind\le m$, $\aind\le m$ and $\bind>m$, and $\aind>m$, and using the symmetrization identity \cite[Ch.\;III.1, formula\;(1.4)]{Macdonald1995}:
	\begin{equation*}
		\sum_{\omega\in S_\ell}\prod_{1\le \aind<\bind\le \ell}
		\frac{u_{\omega(\aind)}-q u_{\omega(\bind)}}
		{u_{\omega(\aind)}-u_{\omega(\bind)}}=\frac{(q;q)_{\ell}}{(1-q)^{\ell}},
		\qquad \ell\in\Z_{\ge1}.
	\end{equation*}
	Then $\SSS_{T-m}$ in \eqref{action_is_nice_1_proof2} incorporates the summation over $\ka_{m+1},\ldots,\ka_N$, while the sum over $I$ is the same as the sum over the set $\{\sigma(1),\ldots,\sigma(m)\}$ in \eqref{action_is_nice_1_proof} (here and below $|I|$ means the number of elements in a set $I$). 

	To shorten the notation, denote
	\begin{equation*}
		\CCC_I(u_1,\ldots,u_T):=\prod_{i\in I}
		\bigg(\prod_{j=1}^{N}
		\frac{\nu_j-a_ju_i}{1-a_ju_i}
		\prod_{\bind\notin I}
		\frac{u_{i}-q u_{\bind}}{u_{i}-u_{\bind}}
		\bigg).
	\end{equation*}
	The above computation with \eqref{action_is_nice_1_proof}, \eqref{action_is_nice_1_proof2} shows that \eqref{action_is_nice_1_proof0} can be written as
	\begin{equation}\label{action_is_nice_1_proof3}
		\SSS_T(u_1,\ldots,u_T)=1-
		\sum_{I\subseteq\{1,\ldots,T\}\colon|I|\ge 1}
		\CCC_I(u_1,\ldots,u_T)
		\SSS_{T-|I|}(u_p\colon p\notin I).
	\end{equation}
	We aim to apply the operator $\HSOp_N$ to the above identity multiplied by $\Phi_M(u_1,\ldots,u_T)$. To do that, first apply \Cref{lemma:D_N_first_order_as_contour} to write (the contour encircles the $a_i$'s but no other singularities)
	\begin{equation}
		\begin{split}
			\frac{\HSOp_N\Phi_M(u_1,\ldots,u_T)}
			{\Phi_M(u_1,\ldots,u_T)}&=
			-\frac{1}{2\pi\i}\oint
			\prod_{i=1}^{N}\frac{a_i-\nu_i z}{a_i-z}
			\bigg(\prod_{j=1}^{T}
			\frac{1-qu_jz}{1-u_jz}\bigg)\frac{dz}{z}
			\\&=
			1
			-q^{T}\nu_1 \ldots\nu_N-
			(1-q)
			\sum_{j=1}^{T}
			\CCC_{\{j\}}(u_1,\ldots,u_T).
		\end{split}
		\label{action_is_nice_1_proof4}
	\end{equation}
	Here we evaluated the integral by taking minus residues at the poles $z=0$, $z=\infty$, and $z=u_j^{-1}$, $j=1,\ldots,T$, outside the integration contour. The latter residues lead to the $\CCC_{\{j\}}$'s.
	
	Denote
	\begin{equation*}
	\begin{split}
		\SSST_T(u_1,\ldots,u_T):={}&\SSS_T(u_1,\ldots,u_T)
		\Phi_M(u_1,\ldots,u_T),\\
		\CCCT_I(u_1,\ldots,u_T):={}&\CCC_I(u_1,\ldots,u_T)
		\Phi_M(u_i\colon i\in I)
		.
	\end{split}
	\end{equation*}
	Using the induction assumption, we have from \eqref{action_is_nice_1_proof3}, \eqref{action_is_nice_1_proof4}:
	\begin{multline}
		\HSOp_N
		\SSST_T=
		(1
		-q^{T}\nu_1 \ldots\nu_N)\Phi_M-
		(1-q)
		\sum_{j=1}^{T}
		\Phi_M(\widehat{u_j})
		\CCCT_{\{j\}}
		\\-
		\sum_{I\subseteq\{1,\ldots,T\}\colon|I|\ge 1}
		(q^{|I|}-q^{T}\nu_1 \ldots\nu_N)
		\CCCT_{I}
		\SSST_{T-|I|}(u_p\colon p\notin I),
		\label{action_is_nice_1_proof5}
	\end{multline}
	where $(\widehat{u_j})=(u_1,\ldots,u_{j-1},u_{j+1},\ldots,u_T)$, and here and below in the proof all functions depending on the $u_i$'s have $(u_1,\ldots,u_T)$ as their arguments unless explicitly indicated. In the part of the last sum in the right-hand side of \eqref{action_is_nice_1_proof5} corresponding to $|I|=1$ express $\SSST_{T-1}$ using \eqref{action_is_nice_1_proof3}:
	\begin{equation}\label{action_is_nice_1_proof6}
		\sum_{j=1}^{T}
		\CCCT_{\{j\}}
		\SSST_{T-1}(\widehat{u_j})
		=
		\sum_{j=1}^{T}
		\CCCT_{\{j\}}
		\bigg(
		\Phi_M(\widehat{u_j})-
		\sum_{I_j\subseteq\{1,\ldots,T\}\setminus\{j\}
		\colon |I_j|\ge 1}
		\CCCT_{I_j}(\widehat{u_j})\SSST_{T-1-|I_j|}
		(u_p\colon p\notin I_j\cup\{j\})
		\bigg).
	\end{equation}
	Observe that
	\begin{equation*}
		\CCCT_{\{j\}}(u_1,\ldots,u_T)
		\CCCT_{I_j}(\widehat{u_j})=
		\CCCT_{I_j\cup \{j\}}(u_1,\ldots,u_T)
		\prod_{\bind\in I_j}\frac{u_j-q u_\bind}{u_j-u_\bind},
	\end{equation*}
	and so a part of \eqref{action_is_nice_1_proof6} can be simplified as
	\begin{multline}\label{action_is_nice_1_proof7}
		\sum_{j=1}^{T}\,
		\sum_{I_j\subseteq\{1,\ldots,T\}\setminus\{j\}
		\colon |I_j|\ge 1}
		\,\prod_{\bind\in I_j}\frac{u_j-q u_\bind}{u_j-u_\bind}\,
		\CCCT_{I_j\cup \{j\}}(u_1,\ldots,u_T)
		\SSST_{T-|I_j\cup\{j\}|}
		(u_p\colon p\notin I_j\cup\{j\})
		\\=
		\sum_{I\subseteq\{1,\ldots,T\}\colon|I|\ge 2}\frac{1-q^{|I|}}{1-q}
		\CCCT_{I}(u_1,\ldots,u_T)\SSST_{T-|I|}(u_p\colon p\notin I).
	\end{multline}
	Indeed, in the left-hand side of \eqref{action_is_nice_1_proof7} there are $|I_j|+1$ choices of the index $j$ to form $I_j\cup \{j\}$, and summing over these choices with the help of the identity 
	\begin{equation*}
		\sum_{\ell=1}^{A}\prod_{i\ne \ell}
		\frac{u_i-qu_\ell}{u_i-u_\ell}
		=
		\frac{1-q^{A}}{1-q},\qquad A\in\Z_{\ge1}
	\end{equation*}
	(this is the same as $\MacdonaldOp_N1=(1-t^N)/(1-t)$), we get the right-hand side of \eqref{action_is_nice_1_proof7}.

	Applying \eqref{action_is_nice_1_proof7} to \eqref{action_is_nice_1_proof5}, we conclude that 
	\begin{multline}\label{action_is_nice_1_proof8}
		\HSOp_N
		\SSST_T=
		(1
		-q^{T}\nu_1 \ldots\nu_N)\Phi_M
		-(1
		-q^{T}\nu_1 \ldots\nu_N)
		\sum_{j=1}^{T}
		\Phi_M(\widehat{u_j})
		\CCCT_{\{j\}}
		\\-
		\sum_{I\subseteq\{1,\ldots,T\}\colon|I|\ge 2}
		\bigg(q^{|I|}-q^{T}\nu_1 \ldots\nu_N
		-(q-q^{T}\nu_1 \ldots\nu_N)\frac{1-q^{|I|}}{1-q}
		\bigg)
		\CCCT_{I}
		\SSST_{T-|I|}(u_p\colon p\notin I).
	\end{multline}
	Noting that
	\begin{equation*}
		q^{|I|}-q^{T}\nu_1 \ldots\nu_N
		-(q-q^{T}\nu_1 \ldots\nu_N)\frac{1-q^{|I|}}{1-q}=
		(1-q^{T}\nu_1 \ldots\nu_N)\frac{q^{|I|}-q}{1-q},
	\end{equation*}
	we finally get an overall factor $1-q^{T}\nu_1 \ldots\nu_N$ in $\HSOp_N\SSST_T$. It remains to show that the right-hand side of \eqref{action_is_nice_1_proof8} divided by this factor is equal to $\SSST_T$. This follows by doing the previous computation \eqref{action_is_nice_1_proof6}--\eqref{action_is_nice_1_proof7} backwards, because $-(q^{|I|}-q)/(1-q)=-1+(1-q^{|I|})/(1-q)$. This completes the proof.
\end{proof}

\begin{remark}\label{rmk:bad_examples}
	In \Cref{lemma:action_is_nice_1} it is crucial that we take a specific linear combination of the $\FDenominators{M}_\ka$'s, as for individual such functions the statement might fail. Indeed, for example,
	\begin{equation}\label{bad_examples}
		% \frac{\HSOp_1\FDenominators2_{(2)}(u_1\md\mathbf{a},\boldsymbol\nu)}
		% {\FDenominators2_{(2)}(u_1\md\mathbf{a},\boldsymbol\nu)}=q(1-\nu_1),\qquad
		\frac{\HSOp_2\FDenominators2_{(2)}(u_1)}
		{\FDenominators2_{(2)}(u_1)}=
		1-q\nu_1\nu_2+(1-q)(1-\nu_1)\frac{a_2}{a_1-a_2},
		\qquad 
		\frac{\HSOp_2\bigl(
		\FDenominators2_{(1)}(u_1)+
		\FDenominators2_{(2)}(u_1)\bigr)}
		{\FDenominators2_{(1)}(u_1)+
		\FDenominators2_{(2)}(u_1)}=1-q\nu_1\nu_2
	\end{equation}
	(we omitted dependence on $\mathbf{a},\boldsymbol\nu$ for shorter notation). The second equality in \eqref{bad_examples} follows from \Cref{lemma:action_is_nice_1}, while in the first equality this lemma fails.
\end{remark}

\begin{proposition}\label{lemma:action_is_nice}
	For any $N\ge1$, $T\ge0$, $\ell\ge1$, $N\ge N_1\ge \ldots\ge N_\ell\ge1$, and $M\ge N_1$ we have
	\begin{multline}\label{computation_of_prtial_symm_multilevel}
		\HSOp_{N_\ell}\ldots \HSOp_{N_1}
		\sum_{\ka\in\Partt_T\colon \ka_1\le N}
		\FDenominators{M}_\ka
		(u_1,\ldots,u_T\md\mathbf{a},\boldsymbol\nu)\\=
		\sum_{\ka\in\Partt_T\colon \ka_1\le N}
		\,\prod_{j=1}^{\ell}\Big(q^{\HeightFunction_{\ka}(N_j+1)}-
		q^{T+\ell-j}\nu_1 \ldots\nu_{N_j}\Big)
		\FDenominators{M}_\ka
		(u_1,\ldots,u_T\md\mathbf{a},\boldsymbol\nu),
	\end{multline}
	where $\HeightFunction_{\ka}(N_j+1)$ is the number of components of $\ka$ which are $\ge N_j+1$.
\end{proposition}
\begin{proof}
	Consider the action of $\HSOp_{N_1}$ on the sum over $\ka$ with $\ka_1\le N$ as in the left-hand side of \eqref{computation_of_prtial_symm_multilevel}. Separate this sum into $T+1$ sums as follows:
	\begin{equation*}
		\sum_{\ka_1\le N}\FDenominators{M}_\ka(u_1,\ldots,u_T\md\mathbf{a},\boldsymbol\nu)
		=\sum_{m=0}^{T}\;\sum_{{N\ge \ka_1\ge \ldots\ge \ka_m>N_1,\,
		\ka_{m+1}\le N_1}}\FDenominators{M}_\ka(u_1,\ldots,u_T\md\mathbf{a},\boldsymbol\nu).
	\end{equation*}
	Write each $\FDenominators{M}_\ka$ in the sum above as a sum over permutations $\sigma\in S_T$ using \eqref{F_probabilistic_symmetrization_formula}. Then every term in this sum over $\sigma$ contains $m$ multiplicative factors of the form $\prod_{i=1}^{N_1}(\nu_i-a_i u_{\sigma(j)})$, $j=1,\ldots,m$. Interchanging the operator $\HSOp_{N_1}$ with them as in \eqref{simple_q_commutation} introduces the factor $q^{m}=q^{\HeightFunction_\ka(N_1+1)}$. After this interchange, the rest of the sum over $\sigma$ can be written as a linear combination with coefficients independent of $a_1,\ldots,a_{N_1},\nu_1,\ldots,\nu_{N_1}$ of quantities of the form
	\begin{equation*}
		\FDenominators{M}_{(\ka_{m+1},\ka_{m+2},\ldots,\ka_T)}(u_{\sigma(m+1)},\ldots,u_{\sigma(T)}\md\mathbf{a},\boldsymbol\nu).
	\end{equation*}
	The action of $\HSOp_{N_1}$ on the sum of the above quantities over all $N_1\ge\ka_{m+1}\ge\ldots\ge\ka_T$ introduces the factor $(1-q^{T-m}\nu_1 \ldots\nu_{N_1})$ by \Cref{lemma:action_is_nice_1}. Since this factor is independent of the $u_j$'s or $\ka_1,\ldots,\ka_m$, we conclude that the desired identity \eqref{computation_of_prtial_symm_multilevel} holds for $\ell=1$.

	The general case follows by a repeated application of \eqref{computation_of_prtial_symm_multilevel} with $\ell=1$ since
	\begin{equation*}
		\HSOp_{N_\ell}\ldots\HSOp_{N_2}
		\bigl(q^{\HeightFunction_\ka(N_1+1)}-
		q^{T}\nu_1 \ldots\nu_{N_1}\bigr)=
		\bigl(q^{\HeightFunction_\ka(N_1+1)}-
		q^{T+\ell-1}\nu_1 \ldots\nu_{N_1}\bigr)\HSOp_{N_\ell}
		\ldots\HSOp_{N_2}.
	\end{equation*}
	This completes the proof.
\end{proof}

% subsection computations_of_the_action (end)

\subsection{Application to higher spin six vertex model} % (fold)
\label{sub:applications_to_q_moments_and_q_laplace_transforms}

The next theorem summarizes the computations of \Cref{sub:computations_of_the_action}.
\begin{theorem}\label{thm:action}
	{\rm\bf{1.\/}\/}
	Under assumptions \eqref{vertex_model_parameters}, \eqref{condition_on_parameters_for_Whittaker_link}, for any $T\ge0$, $\ell\ge1$, $N_1\ge \ldots\ge N_\ell\ge1$, and arbitrary $M\ge N_1$ we have
	\begin{equation}\label{thm_action_1}
		\EE_{\mathbf{HS}}^{\textnormal{step}}
		\bigg[\prod_{j=1}^{\ell}\Big(q^{\HeightFunction(N_j+1,T)}-
		q^{T+\ell-j}\nu_1 \ldots\nu_{N_j}\Big)\bigg]=
		\frac{\HSOp_{N_\ell}\ldots \HSOp_{N_1}\Phi_M}{\Phi_M},
	\end{equation}
	where the operators $\HSOp_{N_j}$ \eqref{conjugated_operators_DN_HSOp_N_a_nu} are conjugate to the first $q$-Whittaker difference operators in $N_j$ variables, $\Phi_M$ is defined by \eqref{Phi_M}, $\EE_{\mathbf{HS}}^{\textnormal{step}}$ means the expectation with respect to the higher spin six vertex model with the step boundary condition, and $\HeightFunction$ is the height function \eqref{height_function_def}.
	
	\smallskip

	\noindent{\rm\bf{2.\/}\/} If, in addition, the parameters of the higher spin six vertex model satisfy
	\begin{equation}\label{7.4_prime_condition}
		\min\limits_{1\le i\le N}a_i>q\max\limits_{1\le i\le N}a_i,
	\end{equation}
	then \eqref{thm_action_1} can be written in a nested contour integral form:
	\begin{equation}\label{thm_action_2}
		\begin{split}
			&\EE_{\mathbf{HS}}^{\textnormal{step}}
			\bigg[\prod_{j=1}^{\ell}\Big(q^{\HeightFunction(N_j+1,T)}-
			q^{T+\ell-j}\nu_1 \ldots\nu_{N_j}\Big)\bigg]
			\\&\hspace{40pt}=
			\frac{(-1)^{\ell}q^{\ell(\ell-1)/2}}{(2\pi\i)^{\ell}}
			\oint\frac{dz_1}{z_1}\ldots \oint
			\frac{dz_\ell}{z_\ell}
			\prod_{1\le \aind<\bind\le \ell}\frac{z_{\aind}-z_{\bind}}{z_{\aind}-qz_{\bind}}
			\prod_{j=1}^{\ell}\bigg(
			\prod_{i=1}^{N_j}\frac{a_i-\nu_i z_j}{a_i-z_j}
			\prod_{r=1}^{T}\frac{1-qu_rz_j}{1-u_rz_j}\bigg),
		\end{split}
	\end{equation}
	where the integration contours are as in \Cref{fig:qWhit_contours}: positively oriented, encircle $a_1,\ldots,a_N$, the $z_j$ contour contains $qz_{j+1},\ldots,qz_{\ell}$, and the contours encircle no other singularities.
\end{theorem}
\begin{proof}
	Taking $N\to+\infty$ in \Cref{lemma:action_is_nice} lifts the restriction that $\ka_1\le N$ in both sides of \eqref{computation_of_prtial_symm_multilevel}. Using \eqref{F_stoch_sum_to_one}, we see that the sum in the left-hand side of \eqref{computation_of_prtial_symm_multilevel} is equal to $\Phi_M$. Dividing both sides by $\Phi_M$, we obtain \eqref{thm_action_1}.

  The contour integral formula \eqref{thm_action_2} follows from \eqref{thm_action_1} by repeatedly applying \Cref{lemma:D_N_first_order_as_contour} (this is similar to the $q$-Whittaker case, cf.{} the proof of \Cref{prop:qWhit_process_moments}).
\end{proof}
\begin{proof}[Proof of \Cref{thm:obaservables_Thm9_8} via difference operators]
	We start from the nested contour integral formulas \eqref{thm_action_2}. Assume that the parameters of the higher spin six vertex model satisfy \eqref{vertex_model_parameters}, \eqref{7.4_conditions}, \eqref{condition_on_parameters_for_Whittaker_link}, and \eqref{7.4_prime_condition}. The latter two conditions can be dropped after establishing identity \eqref{obaservables_Thm9_8} (similarly to analytic continuation arguments in \cite[Section 7.2]{BorodinPetrov2016inhom}). 

	First, note that if $N_\ell=0$ in \eqref{obaservables_Thm9_8}, then the only $w_\ell$ pole in the right-hand side outside the integration contour is infinity, and the minus residue there is equal to $q^{T}$. This agrees with the fact that $\HeightFunction(1,T)$ is identically equal to $T$. Thus, we can always drop any number of zeros among $N_1\ge \ldots\ge N_\ell$, and consider only the case $N_\ell\ge1$. 

	Observe that the right-hand side of \eqref{obaservables_Thm9_8} can be written as
	\begin{equation}\label{obaservables_Thm9_8_proof1}
		q^{\frac{\ell(\ell-1)}2}
		\oint\limits_{\cont[\mathbf{u}^{-1}\md1]}\frac{d w_1}{2\pi\i}
		\ldots
		\oint\limits_{\cont[\mathbf{u}^{-1}\md\ell]}
		\frac{d w_\ell}{2\pi\i}
		\prod_{1\le \aind<\bind\le \ell}
		\frac{w_\aind-w_\bind}{w_\aind-qw_\bind}
		\prod_{i=1}^{\ell}
		\frac{g_{N_i}(w_i)}{w_i},
	\end{equation}
	where
	\begin{equation*}
		g_{N_i}(w):=\prod_{j=1}^{N_i}
		\frac{a_j-\nu_jw}{a_j-w}
		\prod_{j=1}^{T}\frac{1-qu_jw}{1-u_jw}
	\end{equation*}
	is a rational function
	holomorphic in a neighborhood of infinity, and $g_{N_i}(\infty)=q^{T}\nu_1 \ldots\nu_{N_i}$. Let us evaluate the integral in \eqref{obaservables_Thm9_8_proof1} by taking minus residues at poles outside the integration contours, starting from $w_\ell$ down to $w_1$. In doing so, let us separate the pole at infinity from all other poles. This leads to a linear combination of $2^{\ell}$ summands containing integrals like \eqref{thm_action_2}. Indeed, the non-infinite poles of \eqref{obaservables_Thm9_8_proof1} are $w_\ell=a_i$, $w_{\ell-1}=a_i$ or $qa_i$, etc., and the integration contours in \eqref{thm_action_2} pick exactly these residues.

	In more detail, these $2^{\ell}$ integrals are indexed by subsets $I\subseteq\{1,\ldots,\ell\}$ such that the $w_i$'s for $i\notin I$ correspond to minus residue at infinity. Let $|I|=k$, $k=0,\ldots,\ell$, $I=\{i_1<\ldots<i_k\}$, and $\{1,\ldots,\ell\}\setminus I=\{p_1<\ldots<p_{\ell-k}\}$. Also denote $\|I\|:=i_1+\ldots+i_k$, and note that $\ell(\ell-1)/2-\|I\|=p_1+\ldots+p_{\ell-k}$. Clearly, taking the minus residue at each $w_{p_{j}}=\infty$ introduces the factor $q^{-(p_j-1)}g_{N_{p_j}}(\infty)$ coming from the product over $\aind<\bind$ in \eqref{obaservables_Thm9_8_proof1} (recall that the residues are taken in the order $w_\ell,\ldots,w_1$). Thus, \eqref{obaservables_Thm9_8_proof1} becomes (after renaming $w_{i_j}=z_j$): 
	\begin{multline*}
		\sum_{k=0}^{\ell}\,\sum_{I=\{i_1<\ldots<i_k\}\subseteq\{1,\ldots,\ell\}}
		q^{\|I\|+\ell-k}\bigg(\prod_{r\notin I}q^{T}\nu_1 \ldots\nu_{N_r}\bigg)
		\\\times
		\frac{(-1)^{k}}{(2\pi\i)^{k}}
		\oint dz_1\ldots \oint
		dz_k
		\prod_{1\le \aind<\bind\le k}\frac{z_{\aind}-z_{\bind}}{z_{\aind}-qz_{\bind}}
		\prod_{j=1}^{k}\frac{g_{N_{i_j}}(z_j)}{z_j},
	\end{multline*}
	where the contours are now $q$-nested as in \Cref{thm:action} (given in \Cref{fig:qWhit_contours}). With the help of \eqref{thm_action_2}, each integral above equals to an expectation, and the desired statement would follow from the formal identity
	\begin{equation}\label{formal_identity}
		\sum_{k=0}^{\ell}
		q^{\frac{\ell(\ell+1)}2-\frac{k(k+1)}2}
		\sum_{I=\{i_1<\ldots<i_k\}\subseteq\{1,\ldots,\ell\}}
		\bigg(\prod_{r\notin I}b_r\bigg)
		(X_{i_1}-q^{k-1+i_1}b_{i_1})
		\ldots
		(X_{i_k}-q^{i_k}b_{i_k})
		=X_1 \ldots X_\ell
	\end{equation}
	in indeterminates $X_1,\ldots,X_\ell, b_1,\ldots,b_\ell$, and $q$. Indeed, $X_i$ should be matched to $q^{\HeightFunction(N_i+1,T)+i}$, and $b_i$ --- to $q^{T}\nu_1 \ldots\nu_{N_i}$.

	Identity \eqref{formal_identity} can be proven by induction (base $\ell=0$ or $1$ is readily verified). Since it is linear in $X_{\ell}$, it suffices to establish it for two values of $X_\ell$, say, $q^{\ell}b_\ell$ and $\infty$. In the first case the summation in \eqref{formal_identity} reduces to subsets not containing $\ell$, and the identity becomes
	\begin{equation*}
		\sum_{k=0}^{\ell-1}
		q^{\frac{\ell(\ell+1)}2-\frac{k(k+1)}2}
		\sum_{I\subseteq\{1,\ldots,\ell-1\},\,|I|=k}
		b_{\ell}\bigg(\prod_{r\notin I}b_r\bigg)
		(X_{i_1}-q^{k-1+i_1}b_{i_1})
		\ldots
		(X_{i_k}-q^{i_k}b_{i_k})
		=X_1 \ldots X_{\ell-1}q^{\ell}b_\ell,
	\end{equation*}
	which is the same as \eqref{formal_identity} with $\ell$ replaced by $\ell-1$, and so holds by the induction assumption. Similarly, dividing \eqref{formal_identity} by $X_{\ell}$ and sending $X_\ell\to\infty$, we obtain a sum over subsets which must contain $\ell$: 
	\begin{equation*}
		\sum_{k=1}^{\ell}
		q^{\frac{\ell(\ell+1)}2-\frac{k(k+1)}2}
		\sum_{\substack{I\subseteq\{1,\ldots,\ell\}\\|I|=k,\,i_k=\ell}}
		\bigg(\prod_{r\notin I}b_r\bigg)
		(X_{i_1}-q^{k-1+i_1}b_{i_1})
		\ldots
		(X_{i_{k-1}}-q^{i_{k-1}}b_{i_{k-1}})
		=X_1 \ldots X_{\ell-1}.
	\end{equation*}
	This equality is also equivalent to \eqref{formal_identity} with $\ell-1$.

	This completes the proof of \Cref{thm:obaservables_Thm9_8}.
\end{proof}

\begin{remark}
	The proof of how formula \eqref{thm_action_2} for averages of the product-form observables implies \Cref{thm:obaservables_Thm9_8} is similar to the proof of \cite[Lemma 9.11]{BorodinPetrov2016inhom}. The latter involves separating residues at zero instead of at infinity, which leads to different product-form observables than the ones in the left-hand side of \eqref{thm_action_2}.
\end{remark}

% subsection applications_to_q_moments_and_q_laplace_transforms (end)

\subsection{Matching to $q$-Whittaker measures} % (fold)
\label{sub:matching_with_q_whittaker_measures}

The operators $\HSOp_N$ involved in the expression \eqref{thm_action_1} for averages of the product-form observables with respect to the stochastic higher spin six vertex model are conjugate to the first $q$-Whittaker operators $\WhittakerOp_N$ \eqref{Whittaker_difference_operator}, see \eqref{conjugated_operators_DN_HSOp_N}. This allows to match the left-hand side of \eqref{thm_action_1} to $q$-moments of a certain $q$-Whittaker measure. To this end, for any $T\ge0$ and $N\ge1$ define a specialization $\rho(N,T)$ having the only nonzero parameters
\begin{equation}\label{matching_Whittaker_parameters}
	(\al_1,\ldots,\al_N)=(\nua_1,\ldots,\nua_N)\quad
	\textnormal{and}\quad
	(\be_1,\ldots,\be_T)=(-u_1,\ldots,-u_T),
\end{equation}
cf. \eqref{Pi_definition} (recall that $\nua_i=\nu_i/a_i$).
\begin{theorem}\label{thm:matching}
	Under \eqref{vertex_model_parameters}, \eqref{condition_on_parameters_for_Whittaker_link}, for any $T\ge0$, $N\ge1$, and $\ell\ge 1$ we have
	\begin{equation}\label{matching_moment_formula}
		\EE_{\mathbf{HS}}^{\textnormal{step}}
		\bigg[\prod_{j=1}^{\ell}\Big(q^{\HeightFunction(N+1,T)}-
		q^{T+\ell-j}\nu_1 \ldots\nu_{N}\Big)\bigg]=
		\EE_{\WhittakerMeas_{(a_1,\ldots,a_N);\rho(N,T)}}
		\big[q^{\ell\,\la_N}\big], 
	\end{equation}
	where $\EE_{\WhittakerMeas_{(a_1,\ldots,a_N);\rho(N,T)}}$ is the expectation with respect to the $q$-Whittaker measure \eqref{qWhittaker_measure} corresponding to the specialization \eqref{matching_Whittaker_parameters}, and $\la_N$ is the last component of the random partition $\la\in\Part_N$ under this measure.
\end{theorem}
Conditions \eqref{vertex_model_parameters} ensure that the specialization \eqref{matching_Whittaker_parameters} is $q$-Whittaker nonnegative, and this together with \eqref{condition_on_parameters_for_Whittaker_link} implies that the $q$-Whittaker measure in the right-hand side of \eqref{matching_moment_formula} is well-defined and has nonnegative probability weights.
\begin{proof}[Proof of \Cref{thm:matching}]
	Setting $N_1=\ldots=N_\ell=N$ and $M=N$ in \eqref{thm_action_1} and using \eqref{conjugated_operators_DN_HSOp_N}, we see that the left-hand side of \eqref{matching_moment_formula} has the form
	\begin{equation}\label{matching_moment_formula_proof}
		\frac{(\WhittakerOp_N)^{\ell}(\Pi_N\Phi_N)}{\Pi_N\Phi_N},
		\quad\textnormal{where}\quad
		\Pi_N\Phi_N=
		\Pi_{\WhittakerMeas}\bigl(a_1,\ldots,a_N;\rho(N,T)\bigr),
	\end{equation}
	which together with \Cref{prop:qWhit_moments} implies the claim.
\end{proof}
\begin{remark}
	\Cref{thm:matching} can alternatively be deduced by comparing contour integral expressions for both sides of \eqref{matching_moment_formula}: for the left-hand side such an expression would follow from the results of \cite{BorodinPetrov2016inhom}, and for the right-hand side this is \Cref{prop:qWhit_moments}. However, let us emphasize that our proof of \Cref{thm:matching} does not use \cite{BorodinPetrov2016inhom} or contour integrals, and instead follows by considering difference operators.
\end{remark}
\begin{remark}\label{rmk:clash_for_multipoint}
	Averages of joint observables of the higher spin six vertex model depending on several locations $(N_i+1,T)$ as in \Cref{thm:action} cannot be directly interpreted as expectations with respect to a $q$-Whittaker measure or process because of the different conjugations arising in the right-hand side of \eqref{thm_action_1}:
	\begin{equation*}
		\frac{\HSOp_{N_\ell}\ldots \HSOp_{N_1}\Phi_M}{\Phi_M}
		=
		\frac{\Pi_{N_\ell}^{-1}
		\WhittakerOp_{N_\ell}\Pi_{N_\ell}\ldots \Pi_{N_2}^{-1}
		\WhittakerOp_{N_2}\Pi_{N_2}
		\Pi_{N_1}^{-1}\WhittakerOp_{N_1}\Pi_{N_1}\Phi_M}{\Phi_M}.
	\end{equation*}
	In \Cref{sec:capling_with_q_tasep} below we interpret joint averages as in \Cref{thm:action} (but for the step Bernoulli boundary conditions) via a certain $q$-TASEP.
\end{remark}

\begin{corollary}\label{cor:matching_gen_functions}
	For any $T\ge0$, $N\ge1$, and $\zeta\in\C\setminus\R_{<0}$ we have
	\begin{equation}\label{matching_gen_functions}
		\EE_{\mathbf{HS}}^{\textnormal{step}}
		\Bigl[
		\frac{(\zeta q^{T}\nu_1\cdots\nu_N;q)_{\infty}}
		{(\zeta q^{\HeightFunction(N+1,T)};q)_{\infty}}
		\Bigr]
		=\EE_{\WhittakerMeas_{(a_1,\ldots,a_N);\rho(N,T)}}
		\Bigl[
		\frac1{(\zeta q^{\la_{N}};q)_{\infty}}
		\Bigr].
	\end{equation}
\end{corollary}
\begin{proof}
	This is achieved by multiplying both sides of \eqref{matching_moment_formula} by $\zeta^{\ell}/(q;q)_{\ell}$ and summing over $\ell\ge0$ with the help of the $q$-binomial theorem \cite[(1.3.2)]{GasperRahman}. As the random variables $\prod_{j=1}^{\ell}(q^{\HeightFunction(N+1,T)}-q^{T+\ell-j}\nu_1 \ldots\nu_{N})$ and $q^{\ell\,\la_N}$ in both sides are bounded, both $q$-binomial series converge for sufficiently small $|\zeta|$, and so the summations and expectations can be interchanged. After establishing \eqref{matching_gen_functions} for small $|\zeta|$, we can analytically continue this identity to $\zeta\in\C\setminus\R_{<0}$.
\end{proof}

\begin{remark}\label{rmk:Fredholm_and_asymptotics}
	The right-hand side of \eqref{matching_gen_functions} can be viewed as a $q$-analogue of the Laplace transform of the random variable $\la_N$ \cite[Section 3.1.1]{BorodinCorwin2011Macdonald}, and can be written as a certain Fredholm determinant \cite[Section 3.2.3]{BorodinCorwin2011Macdonald}, \cite{BorodinCorwin2013discrete}. Such Fredholm determinants are instrumental for asymptotic analysis, cf. \cite{FerrariVeto2013}, \cite{barraquand2015phase}.\footnote{These two papers deal with the continuous time $q$-TASEP corresponding (as in \Cref{prop:qTASEP_moments} below) to the $q$-Whittaker measure \eqref{qWhittaker_measure} in which the specialization $\rho$ has the only nonzero parameter $\gamma$ (interpreted as continuous time).}
\end{remark}

Let $\bar\rho(N,T)$ be the specialization whose only nonzero parameters are
\begin{equation}\label{matching_Whittaker_parameters_Bernoulli}
	(\al_1,\ldots,\al_{N-1})=(\nua_2,\ldots,\nua_N)\quad
	\textnormal{and}\quad
	(\be_1,\ldots,\be_T)=(-u_1,\ldots,-u_T).
\end{equation}
\begin{corollary}\label{cor:matching_Bernoulli}
	For any $T\ge0$ and $N\ge1$ the distribution of the height function $\HeightFunction(N+1,T)$ of the stochastic higher spin six vertex model with the step Bernoulli boundary condition (defined in \Cref{sub:step_and_bernoulli_boundary_conditions}) is the same as the distribution of the last component $\la_N$ of a random partition under the $q$-Whittaker measure $\WhittakerMeas_{(a_1,\ldots,a_N);\bar\rho(N,T)}$. 
\end{corollary}
\begin{proof}
	Setting $\nu_1=0$ (and so $\nua_1=0$) in \Cref{cor:matching_gen_functions} turns the numerator in the left-hand side of \eqref{matching_gen_functions} into $1$, and so we have
	\begin{equation*}
		\EE_{\mathbf{HS}}^{\textnormal{step Bernoulli}}
		\Bigl[
		\frac{1}
		{(\zeta q^{\HeightFunction(N+1,T)};q)_{\infty}}
		\Bigr]
		=\EE_{\WhittakerMeas_{(a_1,\ldots,a_N);\bar\rho(N,T)}}
		\Bigl[
		\frac1{(\zeta q^{\la_{N}};q)_{\infty}}
		\Bigr].
	\end{equation*}
	Since the $q$-Laplace transform $\EE (\zeta q^{X};q)_{\infty}$ (viewed as a function of $\zeta\in\C\setminus\R_{>0}$) determines the distribution of a random variable $X\in\Z_{\ge0}$ (e.g., see \cite[Proposition 3.1.1]{BorodinCorwin2011Macdonald}), we get the desired equality of distributions.
\end{proof}

In the next section we present an independent probabilistic mechanism explaining the equality in distribution in \Cref{cor:matching_Bernoulli}.

% subsection matching_with_q_whittaker_measures (end)

% section action_of_q_whittaker_difference_operators_on_the_higher_spin_six_vertex_model (end)

% \end{comment}

\section{$q$-TASEP~/~vertex model coupling} % (fold)
\label{sec:capling_with_q_tasep}

In this section we mainly consider the stochastic higher spin six vertex model with the step Bernoulli boundary condition (as opposed to the more general step boundary condition, see \Cref{sub:step_and_bernoulli_boundary_conditions}). To indicate this, we will denote the corresponding height function (viewed as a collection of random variables indexed by points of the quadrant) by $\HeightFunctionBernoulli(N,T)$.

\subsection{Discrete time $q$-TASEPs} % (fold)
\label{sub:discrete_time_q_taseps}

Here we briefly recall the definitions of the Bernoulli and the geometric $q$-TASEPs from \cite{BorodinCorwin2013discrete}, and outline their connection to the $q$-Whittaker measures, as well as discuss their other properties. Statements formulated here (\Cref{prop:qTASEP_moments,prop:commute}) can be readily deduced from the existing literature, so we only give ideas of the proofs.

We need the \emph{$q$-deformation of the truncated geometric distribution}:
\begin{equation}\label{q_geom_distr}
	\pgeom_{m,\al}(j):=\al^{j}(\al;q)_{m-j}\frac{(q;q)_m}{(q;q)_j(q;q)_{m-j}},\qquad j=0,1,\ldots,m.
\end{equation}
Here $\al\in(0,1)$, and $m\ge0$ is arbitrary. This distribution also makes sense for $m=+\infty$;
\begin{equation}\label{q_geom_distr_infinity}
	\pgeom_{+\infty,\al}(j):=\al^{j}(\al;q)_{\infty}\frac{1}{(q;q)_{j}},
	\qquad j=0,1,2,\ldots.
\end{equation}
The $q$-binomial theorem \cite[(1.3.2)]{GasperRahman} implies that both \eqref{q_geom_distr} and \eqref{q_geom_distr_infinity} are probability distributions in $j$.

The geometric and Bernoulli $q$-TASEPs are discrete time Markov processes on the space of particle configurations in $\Z$ (in which at most one particle is allowed per site). The particles jump only to the right (but not necessarily by $1$) and preserve their order. For our purposes it suffices to consider processes with only finitely many particles (say, $L$). Denote them by $x_1> \ldots>x_L$. We will always assume that the processes start from the \emph{step initial configuration}: $x_i(0)=-i$, $i=1,\ldots,L$. 

\begin{definition}\label{def:geom}
	Let $a_1,\ldots,a_L>0$ and $\al_1,\al_2,\ldots>0$ be such that $a_i\al_j<1$ for all $i,j$. Under the \emph{geometric $q$-TASEP} with rates $a_i$ and parameters $\al_j$, at each increment $\ttime\to \ttime+1$ of the discrete time $\ttime\ge0$, each particle $x_i$ independently jumps to the right by $j\ge0$ with probability $\pgeom_{x_{i-1}(\ttime)-x_i(\ttime)-1,a_i\al_{\ttime+1}}(j)$, where, by agreement, $x_0\equiv+\infty$. Let us denote the corresponding Markov transition matrix (with rows and columns indexed by $L$-particle configurations in~$\Z$) by $\mathbf{G}_{\al_{\ttime+1}}$.
\end{definition}
\begin{definition}\label{def:Bernoulli}
	Let $a_1,\ldots,a_L>0$ and $\be_1,\be_2,\ldots>0$. The update at the time increment $\ttime\to\ttime+1$, $\ttime\ge0$, of the \emph{Bernoulli $q$-TASEP} with rates $a_i$ and parameters $\be_j$ looks as follows. First, the particle $x_1$ jumps to the right by one with the probability $a_1\be_{\ttime+1}/(1+a_1\be_{\ttime+1})$, or stays put with the complementary probability $1/(1+a_1\be_{\ttime+1})$. Sequentially for each $i=2,\ldots,L$, if the particle $x_{i-1}$ has jumped by one (i.e., $x_{i-1}(\ttime+1)=x_{i-1}(\ttime)+1$), then the particle $x_i$ jumps by one with the probability $a_i\be_{\ttime+1}/(1+a_i\be_{\ttime+1})$ or stays put with the complementary probability. Otherwise if $x_{i-1}(\ttime+1)=x_{i-1}(\ttime)$, then the particle $x_i$ jumps by one with the probability $(1-q^{x_{i-1}(\ttime)-x_i(\ttime)-1})a_i\be_{\ttime+1}/(1+a_i\be_{\ttime+1})$ or stays put with the complementary probability. Denote the corresponding Markov transition matrix by $\mathbf{B}_{\be_{\ttime+1}}$.
\end{definition}

Note that the update rule in the geometric $q$-TASEP is \emph{parallel}, while in the Bernoulli $q$-TASEP it is \emph{sequential}, with interaction propagating from right to left. Specializing $a_j\equiv 1$, we obtain the Markov transition matrices $\mathbf{G}^\circ_\al$ and $\mathbf{B}^\circ_\be$ from \Cref{sub:intro_coupling} in the Introduction.

\begin{remark}
	In suitable limits to continuous time (involving taking small parameters $\al_j$ or $\be_j$, respectively), both the geometric and the Bernoulli $q$-TASEPs give rise to the \emph{continuous time} (also called \emph{Poisson}) $q$-TASEP introduced in \cite{SasamotoWadati1998}, \cite{BorodinCorwin2011Macdonald} (see also \cite{BorodinCorwinSasamoto2012}). While the latter process should appear (similarly to \Cref{thm:capling} below) in connection with suitable degenerations of the higher spin six vertex model (e.g., when one rescales the vertical dimension of the quadrant from discrete to continuous), here we will not formulate the corresponding results.
\end{remark}
\begin{definition}\label{def:mixed_qTASEP}
	Let $L\ge1$, $N\ge1$, $T\ge0$, and the parameters 
	\begin{equation}\label{mixed_qTASEP_parameters}
		(a_1,\ldots,a_L),\qquad
		(\nua_2,\ldots,\nua_N)=\Bigl(\frac{\nu_2}{a_2},
		\ldots,\frac{\nu_N}{a_N}
		\Bigr),\qquad
		(u_1,\ldots,u_T)
	\end{equation}
	satisfy \eqref{vertex_model_parameters} and \eqref{condition_on_parameters_for_Whittaker_link}. Define the random $L$-particle configuration 
	\begin{equation}\label{mixed_qTASEP}
		\vec\x(N,T)=\bigl\{\x_1(N,T)>\x_2(N,T)> \ldots>\x_L(N,T)\bigr\}
	\end{equation}
	in $\Z$ as follows. Starting from the $L$-particle step initial configuration, perform $N-1$ moves\footnote{We use the word ``move'' to refer to transitions of discrete-time Markov chains during a single time increment, and reserve the word ``step'' for boundary conditions in vertex models and initial conditions in $q$-TASEPs.} $\mathbf{G}_{\nua_2},\ldots,\mathbf{G}_{\nua_N}$ of the geometric $q$-TASEP, and $T$ moves $\mathbf{B}_{-u_1},\ldots,\mathbf{B}_{-u_T}$ of the Bernoulli $q$-TASEP. For both processes, $a_1,\ldots,a_L$ are the particle rates as in \Cref{def:geom,def:Bernoulli}.
\end{definition}

One can think that \eqref{mixed_qTASEP} is a configuration of a mixed geometric/Bernoulli discrete time $q$-TASEP at time $\ttime=N+T-1$. As follows from the next proposition, the order of performing the geometric and Bernoulli moves in \Cref{def:mixed_qTASEP} does not affect the distribution of $\vec\x(N,T)$. Indeed, this is because the $q$-Whittaker measure depends on the parameters $\al_i$ and $\be_j$ of its specialization $\rho$ in a multiplicative way, cf. \eqref{Pi_definition}. See also \Cref{prop:commute} below for a stronger statement.

\begin{proposition}\label{prop:qTASEP_moments}
	The random variable $\x_L(N,T)+L$ has the same distribution as the last component $\la_L$ of the random partition under the $q$-Whittaker measure $\WhittakerMeas_{(a_1,\ldots,a_L);\bar\rho(N,T)}$, where the specialization $\bar\rho(N,T)$ is defined by \eqref{matching_Whittaker_parameters_Bernoulli}. 

	Consequently, for any $\ell\ge1$ the expectation $\EE_{\textnormal{$q$-TASEP}}[q^{\ell(\x_L(N,T)+L)}]$ under the distribution of \Cref{def:mixed_qTASEP} is given by an $\ell$-fold nested contour integral as in the right-hand side of \eqref{qWhit_moments_formula} with parameters corresponding to $\WhittakerMeas_{(a_1,\ldots,a_L);\bar\rho(N,T)}$.
\end{proposition}
\begin{proof}[Idea of proof]
	There are two ways to establish this statement. The first approach \cite{BorodinCorwin2013discrete} is via \emph{Markov duality}, which allows to write a closed system of evolution equations for the quantities $f_{N,T}(L_1,\ldots,L_\ell)=\EE_{\textnormal{$q$-TASEP}}\bigl[\prod_{j=1}^{\ell}q^{\x_{L_{j}}(N,T)+L_{j}}\bigr]$ using the very definition of the geometric or Bernoulli $q$-TASEP moves. Here by evolution equations we mean the ones expressing $f_{N+1,T}(L_1,\ldots,L_\ell)$ or $f_{N,T+1}(L_1,\ldots,L_\ell)$, respectively, as linear combinations of $f_{N,T}(L_1',\ldots,L_\ell')$ over some $L_i'$. The evolution equations can then be solved using a Bethe ansatz type approach in terms of nested contour integrals, which gives the second claim of the proposition. Comparing these expressions with the ones coming from the $q$-Whittaker measures then yields the first claim. Note that because $0<q<1$ and the random variables we deal with belong to $\Z_{\ge0}$, the equality of all $q$-moments implies the equality of the distributions.

	The second approach does not involve nested contour integrals, and first establishes the claim about the equality of distributions of $\x_L(N,T)+L$ and $\la_L$. This argument employs the \emph{column $q$-randomized Robinson--Schensted--Knuth (RSK) correspondences} of \cite{MatveevPetrov2014} (see also \cite[Section 3.3]{BorodinCorwin2011Macdonald}, \cite{OConnellPei2012}, \cite{BorodinPetrov2013NN} for simpler constructions working for the continuous time $q$-TASEP). In our context, the $q$-randomized RSKs provide Markov transition matrices (geometric or Bernoulli) mapping the $q$-Whittaker process $\WhittakerProc_{(a_1,\ldots,a_L);\bar\rho(N,T)}$ \eqref{qWhittaker_process} to $\WhittakerProc_{(a_1,\ldots,a_L);\bar\rho(N+1,T)}$ or $\WhittakerProc_{(a_1,\ldots,a_L);\bar\rho(N,T+1)}$, respectively. Under these Markov moves the evolution of the last components $(\la^{(1)}_{1},\ldots,\la^{(L)}_{L})$ of the random partitions is \emph{marginally Markovian}. Moreover, these geometric and Bernoulli moves at the level of $q$-Whittaker processes lead (via the identification $\la^{(i)}_i-i=x_i$) to the $q$-TASEP moves $\mathbf{G}_{\nua_{N+1}}$ and $\mathbf{B}_{-u_{T+1}}$, respectively. By recalling that the distribution of each partition $\la^{(j)}$ is a $q$-Whittaker measure, we get the first claim. \Cref{prop:qWhit_moments} then implies the second claim. 
\end{proof}

\begin{remark}
	The first approach to \Cref{prop:qTASEP_moments} via duality allows one to obtain nested contour integral formulas for $q$-moments of the processes with initial conditions other than the step initial condition. For example, see \cite[Theorem 2.11]{BorodinCorwinSasamoto2012} for the half stationary and \cite{BorodinCorwinPetrovSasamoto2013} for general initial conditions in the continuous time $q$-TASEP (the latter formulas are not as explicit as the ones for concrete initial conditions).
\end{remark}
\begin{proposition}\label{prop:commute}
	Fix particle rates $a_1,\ldots,a_L$. Then for any $\al_{1,2},\be_{1,2}>0$ such that $\al_{1,2} a_i<1$, $i=1,\ldots,L$, the Markov transition matrices $\mathbf{G}_{\al_{1}}$, $\mathbf{G}_{\al_{2}}$, $\mathbf{B}_{\be_{1}}$, and $\mathbf{B}_{\be_{1}}$ of \Cref{def:geom,def:Bernoulli} commute with each other.
\end{proposition}
The commutation of these matrices when applied to the step initial configuration (or, more generally, to a configuration obtained by making several geometric or Bernoulli $q$-TASEP moves from the step initial configuration) follows from \Cref{prop:qTASEP_moments}. \Cref{prop:commute} is stronger as it asserts that the matrices commute when applied to an arbitrary configuration.
\begin{proof}[Idea of proof of \Cref{prop:commute}]
	This can be deduced by spectral considerations similarly to \cite[Corollary 2.14]{CorwinPetrov2015}. Namely, one can show that the matrices $\mathbf{G}_{\al_{1,2}},\mathbf{B}_{\be_{1,2}}$ are diagonalized in the same basis. The commutation then follows from a spectral theory which was developed (in the homogeneous case $a_i\equiv \mathrm{const}$) in \cite{BorodinCorwinPetrovSasamoto2013}, \cite{BCPS2014}. In the inhomogeneous case the corresponding statements are contained in \cite[Section~7]{BorodinPetrov2016inhom}.

	Another proof not involving spectral considerations employs the commutation of transfer matrices of the inhomogeneous stochastic higher spin six vertex model, which ultimately follows from the Yang--Baxter equation. Let us explain how to turn those transfer matrices into $\mathbf{G}_\al,\mathbf{B}_\be$.

	Let $\mathbf{T}_v$ be the stochastic transfer matrix corresponding to one horizontal slice of the stochastic higher spin six vertex model with the step boundary condition. Viewed as a transition matrix in a discrete time Markov process, $\mathbf{T}_v$ corresponds to sending an $M$-arrow configuration into a random $(M+1)$-arrow one. Let us use the notation of \cite{BorodinPetrov2016inhom} for the parameters of $\mathbf{T}_v$ so that later we can specialize them into the parameters of $\mathbf{G}_\al$ or $\mathbf{B}_\be$. In more detail, $\mathbf{T}_v$ acts on arrow configurations in $\Z_{\ge1}$ and depends on $q$, the spectral parameter $v$, and $\{\upxi_i\}_{i\ge1}$, $\{\SP_i\}_{i\ge1}$ (the relation of these parameters to our usual ones is given in the end of \Cref{sub:vertex_weights}). The Yang--Baxter equation implies that $\mathbf{T}_v$ commutes with $\mathbf{T}_{v'}$ provided that $q,\upxi_i,\SP_i$, $i\ge1$, in both transfer matrices coincide. We refer to \cite{Borodin2014vertex}, \cite{CorwinPetrov2015}, \cite{BorodinPetrov2016inhom} for details.

	Turning $\mathbf{T}_v$ into $\mathbf{G}_\al$ requires the fusion procedure going back to \cite{KulishReshSkl1981yang}. Its stochastic interpretation \cite[Section 3]{CorwinPetrov2015}, \cite[Section 5]{BorodinPetrov2016inhom} implies that setting $\SP_i=v\upxi_i$, $i\ge1$, and defining $\mathbf{T}^{\scriptscriptstyle(J)}:=\mathbf{T}_{q^{J-1}v}\ldots\mathbf{T}_{qv}\mathbf{T}_{v}$ leads to the following Markov transition matrix. Acting on a configuration $\la=1^{\infty}2^{\ell_2}3^{\ell_3}\ldots$, $\mathbf{T}^{\scriptscriptstyle(J)}$ produces a random configuration $\mu=1^{\infty}2^{m_2}3^{m_3}\ldots$ with $m_i=\ell_i+H_{i-1}-H_i$, $i\ge2$, where $H_i\in\Z_{\ge0}$, $i\ge1$, are independent random variables having the distribution
	\begin{equation}\label{qHahn_distribution}
		\prob(H_i=j)=
		\varphi_{q,q^{J}\SP_i^{2},\SP_i^{2}}(j\md\ell_i),\qquad
		\varphi_{q,\eta,\zeta}(j\md \ell):=
		\eta^j\frac{(\zeta/\eta;q)_{j}(\eta;q)_{\ell-j}}{(\zeta;q)_{\ell}} \frac{(q;q)_{\ell}}{(q;q)_{j}(q;q)_{\ell-j}},
	\end{equation}
	and $\varphi_{q,\eta,\zeta}(j\md{+\infty})=\eta^j(\zeta/\eta;q)_{j}(\eta;q)_{\infty}/((q;q)_{j}(\zeta;q)_{\infty})$ for the case $i=1$ (when $\ell_1=+\infty$). The quantities $\varphi_{q,\eta,\zeta}(\cdot\md\ell)$ indeed define a probability distribution on $\{0,1,\ldots,\ell\}$ (the probability weights are nonnegative under certain restrictions on $q,\eta,\zeta$, cf. \cite[Section 6.6.1]{BorodinPetrov2016inhom}) called the \emph{$q$-deformed Beta-binomial distribution}. Notice that the probabilities in \eqref{qHahn_distribution} depend on $q^{J}$ in an analytic way, and thus we can regard $q^{J}$ in $\mathbf{T}^{\scriptscriptstyle(J)}$ as an independent parameter.

	Interpret each $\ell_i$, $i\ge1$ as a gap between consecutive particles in a configuration $x_1>x_2>\ldots$ in $\Z$ via $\ell_i=x_{i-1}-x_i-1$ (with the agreement that $\ell_1=x_0=+\infty$). The evolution of the particles $x_i$ under $\mathbf{T}^{\scriptscriptstyle(J)}$ can be described as a generalization of the geometric $q$-TASEP, when one replaces the probabilities $\pgeom_{\ell,\eta}(j)$ in \Cref{def:geom} by $\varphi_{q,\eta,\zeta}(j\md \ell)$. This generalization introduced in \cite{Povolotsky2013} is called the \emph{$q$-Hahn TASEP}. To get the geometric $q$-TASEP on the particles $x_i$, specialize the parameters further:
	\begin{equation*}
		\upxi_i^{2}=a_i,\quad i\ge1,\qquad
		q^{J}=\frac{\al}{v^{2}},\qquad v\searrow0,
	\end{equation*}
	In particular, this implies that $\SP_i\to0$, $i\ge1$. Thus, the distribution $\varphi_{q,q^{J}\SP_i^{2},\SP_i^{2}}(j\md\ell_i)$ becomes $\varphi_{q,a_i\al,0}(j\md\ell_i)=\pgeom_{\ell_i,a_i\al}(j)$, so we indeed get the geometric $q$-TASEP $\mathbf{G}_\al$ with particle rates $\{a_i\}$. 

	The commutation of transfer matrices implies that $\mathbf{T}^{\scriptscriptstyle(J_1)}\mathbf{T}^{\scriptscriptstyle(J_2)}=\mathbf{T}^{\scriptscriptstyle(J_2)}\mathbf{T}^{\scriptscriptstyle(J_1)}$, and this commutation also holds when applied to configurations of the form $1^{\infty}2^{\ell_2}3^{\ell_3}\ldots$ (simply take the limit $\ell_1\to+\infty$). This shows that the geometric $q$-TASEP transition matrices $\mathbf{G}_{\al_1}$, $\mathbf{G}_{\al_2}$ commute with each other.

	It remains to explain how to obtain the Bernoulli $q$-TASEP transition matrix $\mathbf{B}_\be$ from the transfer matrices. This can be achieved without the fusion. Consider the transfer matrix $\mathbf{T}_{-{\be}/{v}}$, where $\SP_i=v\upxi_i$ and $\upxi_i^{2}=a_i$, $i\ge1$. Because of this choice of $\upxi_i,\SP_i$, such $\mathbf{T}_{-{\be}/{v}}$ commutes with the above $v$-dependent transfer matrices (before the limit $v\searrow0$). Next, one can readily check that sending $v\searrow0$ and interpreting arrow configurations as gaps in the same way as above, we get the transition matrix $\mathbf{B}_\be$. This implies the remaining commutation statements.
\end{proof}

% subsection discrete_time_q_taseps (end)

\subsection{$q$-TASEP~/~vertex model coupling for vertex weights} % (fold)
\label{sub:capling}

\Cref{prop:qTASEP_moments} with $L=N$ together with \Cref{cor:matching_Bernoulli} implies that the random variables $\x_N(N,T)+N$ and $\HeightFunctionBernoulli(N+1,T)$ have the same distribution. Here and in \Cref{sub:capling_relating_joint_distributions} below we present an independent probabilistic explanation of this by constructing an explicit \emph{coupling} between these random variables. By this we mean constructing a probability space on which both  $\x_N(N,T)+N$ and $\HeightFunctionBernoulli(N+1,T)$ are defined such that $\x_N(N,T)+N=\HeightFunctionBernoulli(N+1,T)$. In fact, our coupling also extends to joint distributions of a certain type. We begin with matching local update rules in both models.

\begin{definition}
	\label{def:main_capling_construction}
	Fix the number of particles $L\ge1$, their rates $a_1,\ldots,a_L>0$, and $\al,\be>0$ such that $\al a_i<1$ for all $i$. Fix an $L$-particle configuration $\vec\xhs$. Let the random configuration $\vec\yhs$ be obtained from $\vec\xhs$ by applying the Bernoulli $q$-TASEP transition matrix $\mathbf{B}_{\be}$. Similarly, let the random configuration $\xhspvec$ be obtained from $\vec\xhs$ by applying $\mathbf{G}_{\al}$. We assume that the randomness in $\mathbf{G}_{\al}$ is independent from that in $\mathbf{B}_{\be}$. 

	Fix $m\in\{1,\ldots,L\}$. Let $\yhsp_m\in\Z$ be a random variable constructed given $\xhs_{m-1},\yhs_{m-1}$, and $\xhsp_m$ as follows:
	\begin{equation}\label{main_capling_construction}
		\begin{split}
			\prob(\yhsp_m=\xhsp_m)&:=
			\LLL_{-\be,a_m,\al a_m}
			(\xhs_{m-1}-\xhsp_{m}-1,\yhs_{m-1}-\xhs_{m-1};
			\yhs_{m-1}-\xhsp_{m}-1,0
			);
			\\
			\prob(\yhsp_m=\xhsp_m+1)&:=
			\LLL_{-\be,a_m,\al a_m}
			(\xhs_{m-1}-\xhsp_{m}-1,\yhs_{m-1}-\xhs_{m-1};
			\yhs_{m-1}-\xhsp_{m}-2,1
			).
		\end{split}
	\end{equation}
	In other words, using $i_1=\xhs_{m-1}-\xhsp_{m}-1$ and $j_1=\yhs_{m-1}-\xhs_{m-1}$ we sample $(i_2,j_2)$ according to the vertex weight $\LLL_{-\be,a_m,\al a_m}$ defined in \Cref{sub:vertex_weights}, and then interpret $j_2$ as $\yhsp_m-\xhsp_m$ (recall that $i_2+j_2=i_1+j_1$). We agree that when $m=1$, $\xhs_{m-1}=\yhs_{m-1}=+\infty$, while $\yhs_{m-1}-\xhs_{m-1}$ can be set to either~$0$ or~$1$ (the corresponding vertex weight does not depend on this choice). See \Cref{fig:vertex_sampling}.
\end{definition}

\begin{figure}[htbp]
	\begin{tikzpicture}
		[scale=1,ultra thick]
		\tikzstyle{myblock} = [draw=black, fill=white, line width=1, minimum width=1em, minimum height=1em, rectangle, rounded corners, text centered]
		\tikzstyle{myroundblock} = [draw=black, fill=white, line width=1, minimum width=1em, minimum height=1em, circle, rounded corners, text centered]
		\def\d{.1}
		\def\h{9}
		\def\v{1.3}
		\node[myblock] (x) at (0,0) {\begin{tikzpicture} [scale=1.5,thick] \draw[->] (-.8,0)--(1.5,0); \draw [fill,red] (1/5,0) circle(3pt);\foreach \pt in {4,1,-2} {\draw[fill] (\pt/5,0) circle(2pt);} \foreach \pt in {-3,...,6} {\draw (\pt/5,0) circle(2pt);}  \end{tikzpicture}};
		\node[myblock] (y) at (0,2*\v) {\begin{tikzpicture} [scale=1.5,thick] \draw[->] (-.8,0)--(1.5,0); \draw [fill,red] (1/5,0) circle(3pt); \foreach \pt in {5,1,-1} {\draw[fill] (\pt/5,0) circle(2pt);} \foreach \pt in {-3,...,6} {\draw (\pt/5,0) circle(2pt);}  \end{tikzpicture}};
		\node[myblock] (xp) at (\h,0) {\begin{tikzpicture} [scale=1.5,thick] \draw[->] (-.8,0)--(1.5,0); \draw [fill,red] (-1/5,0) circle(3pt); \foreach \pt in {6,3,-1 } {\draw[fill] (\pt/5,0) circle(2pt);} \foreach \pt in {-3,...,6} {\draw (\pt/5,0) circle(2pt);}  \end{tikzpicture}};
		\node [anchor=west] at (-1.8,-.57) {$\vec\xhs$};
		\node [anchor=west] at (-1.8,2*\v-.57) {$\vec\yhs$};
		\node [anchor=west] at (\h-1.8,-.57) {$\xhspvec$};
		\node [anchor=west] at (\h-1.8,2*\v-.57) {$\yhsp_m=-1$};
		\draw[line width=3pt,->] (x)--(y) node [midway,xshift=-13] {$\mathbf{B}_{\be}$};
		\draw[line width=3pt,->] (x)--(xp) node [midway,yshift=-10] {$\mathbf{G}_{\al}$};
		\node[myblock] (yp) at (\h,2*\v) {\begin{tikzpicture} [scale=1.5,thick] \draw[->] (-.8,0)--(1.5,0); \draw [fill,red] (0/5,0) circle(3pt); \foreach \pt in {0} {\draw[fill] (\pt/5,0) circle(2pt);} \foreach \pt in {-3,...,6} {\draw (\pt/5,0) circle(2pt);}  \end{tikzpicture}};
		\node[myblock] (mp) at (\h/2+.5,1.2*\v) {$\LLL_{-\be,a_m,\al a_m}$};
		\draw[->,densely dashed] (x)--([yshift=-4]mp.west) node [midway,yshift=10,xshift=-8] {$\xhs_{m-1}=0$};
		\draw[->,densely dashed] (xp)--([xshift=10]mp.south) node [midway,xshift=-20,yshift=-8] {$\xhsp_{m}=-2$};
		\draw[->,densely dashed] (y)--([yshift=4]mp.west) node [midway,yshift=11,xshift=10] {$\yhs_{m-1}=0$};
		\draw[line width=3pt,->] (mp)--(yp.west);
	\end{tikzpicture}
	\caption{Sampling of $\yhsp_{m}$ given $\xhs_{m-1},\yhs_{m-1}$, and $\xhsp_m$ in \Cref{def:main_capling_construction} (these four particles are highlighted). The transition displayed corresponds to the vertex $(1,0;0,1)$, cf. \Cref{fig:vertex_weights_stoch}.}
	\label{fig:vertex_sampling}
\end{figure}

\begin{definition}\label{def:main_capling_construction_2}
	In the setting of \Cref{def:main_capling_construction}, let us define two other random configurations. Let $\yhsgbvec$ be obtained from the configuration $\xhspvec$ by applying the Bernoulli $q$-TASEP move $\mathbf{B}_{\be}^{\GB{.12}}$, and let $\yhsbgvec$ be obtained from $\vec \yhs$ by applying the geometric $q$-TASEP move $\mathbf{G}_{\al}^{\BG{.12}}$. We denote these operations by $\mathbf{B}_{\be}^{\GB{.12}}$, $\mathbf{G}_{\al}^{\BG{.12}}$ to emphasize that they are independent from the ones in \Cref{def:main_capling_construction} (and from each other). See \Cref{fig:qTASEPs_sampling}. Thus, $\yhsgbvec$ depends on $\vec\xhs,\xhspvec$ but not on $\vec\yhs$, and similarly $\yhsbgvec$ depends on $\vec\xhs,\vec\yhs$ but not on $\xhspvec$. By \Cref{prop:commute}, $\yhsgbvec$ and $\yhsbgvec$ have the same distribution. 
\end{definition}

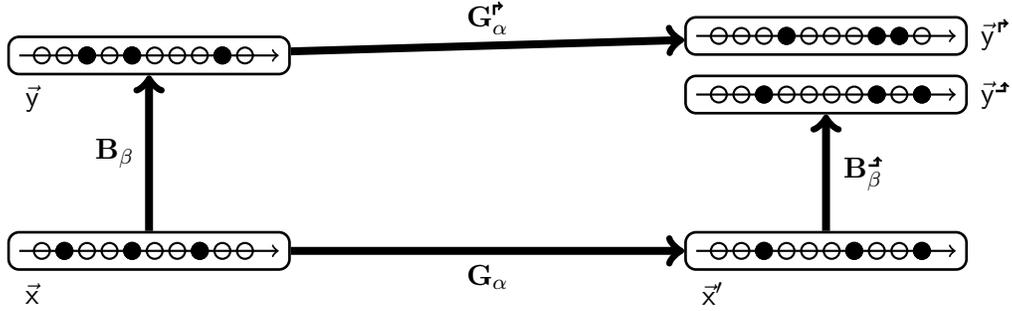
\begin{figure}[htbp]
	\begin{tikzpicture}
		[scale=1,ultra thick]
		\tikzstyle{myblock} = [draw=black, fill=white, line width=1, minimum width=1em, minimum height=1em, rectangle, rounded corners, text centered]
		\def\d{.1}
		\def\h{9}
		\def\v{1.3}
		\node[myblock] (x) at (0,0) {\begin{tikzpicture} [scale=1.5,thick] \draw[->] (-.8,0)--(1.5,0);\foreach \pt in {4,1,-2} {\draw[fill] (\pt/5,0) circle(2pt);} \foreach \pt in {-3,...,6} {\draw (\pt/5,0) circle(2pt);}  \end{tikzpicture}};
		\node[myblock] (y) at (0,2*\v) {\begin{tikzpicture} [scale=1.5,thick] \draw[->] (-.8,0)--(1.5,0); \foreach \pt in {5,1,-1} {\draw[fill] (\pt/5,0) circle(2pt);} \foreach \pt in {-3,...,6} {\draw (\pt/5,0) circle(2pt);}  \end{tikzpicture}};
		\node[myblock] (xp) at (\h,0) {\begin{tikzpicture} [scale=1.5,thick] \draw[->] (-.8,0)--(1.5,0); \foreach \pt in {6,3,-1} {\draw[fill] (\pt/5,0) circle(2pt);} \foreach \pt in {-3,...,6} {\draw (\pt/5,0) circle(2pt);}  \end{tikzpicture}};
		\node [anchor=west] at (-1.8,-.57) {$\vec\xhs$};
		\node [anchor=west] at (-1.8,2*\v-.57) {$\vec\yhs$};
		\node [anchor=west] at (\h-1.8,-.57) {$\xhspvec$};
		\draw[line width=3pt,->] (x)--(y) node [midway,xshift=-13] {$\mathbf{B}_{\be}$};
		\draw[line width=3pt,->] (x)--(xp) node [midway,yshift=-10] {$\mathbf{G}_{\al}$};
		\node[myblock] (ygb) at (\h,1.6*\v) {\begin{tikzpicture} [scale=1.5,thick] \draw[->] (-.8,0)--(1.5,0); \foreach \pt in {6,4,-1} {\draw[fill] (\pt/5,0) circle(2pt);} \foreach \pt in {-3,...,6} {\draw (\pt/5,0) circle(2pt);}  \end{tikzpicture}};
		\node[myblock] (ybg) at (\h,2.2*\v) {\begin{tikzpicture} [scale=1.5,thick] \draw[->] (-.8,0)--(1.5,0); \foreach \pt in {5,4,0} {\draw[fill] (\pt/5,0) circle(2pt);} \foreach \pt in {-3,...,6} {\draw (\pt/5,0) circle(2pt);}  \end{tikzpicture}};
		\draw[line width=3pt,->] (xp)--(ygb) node [midway,xshift=14] {$\mathbf{B}_{\be}^{\GB{.12}}$};
		\draw[line width=3pt,->] (y)--(ybg) node [midway,yshift=10] {$\mathbf{G}_{\al}^{\BG{.12}}$};
		\node [anchor=west] at (\h+1.9,1.6*\v) {$\yhsgbvec$};
		\node [anchor=west] at (\h+1.9,2.2*\v) {$\yhsbgvec$};
	\end{tikzpicture}
	\caption{Random configurations of \Cref{def:main_capling_construction_2} obtained from $\vec\xhs$ by applying the Bernoulli and the geometric $q$-TASEP moves in two different orders.}
	\label{fig:qTASEPs_sampling}
\end{figure}

\begin{proposition}\label{prop:main_capling_prop}
	Fix $L\ge1$ and $m\in\{1,\ldots,L\}$. Given the particle configuration $\vec\xhs$, the joint distribution of $(\yhs_{m-1},\yhsp_m)$ is the same as of $(\yhs_{m-1},\yhsbg_m)$.
\end{proposition}
\begin{proof}
	Let $I$ be the event that $\yhs_{m-1}=\xhs_{m-1}$, and denote $g:=\xhs_{m-1}-\xhs_m-1$. Assume that $m\ge 2$, the argument for $m=1$ is the same after substituting $g=+\infty$. We have for any $r\ge0$:
	\begin{equation*}
		\begin{split}			
		\prob\bigl(
		\yhsp_m=\xhs_m+r\md I,\vec\xhs
		\bigr)&=
		\prob\bigl(
		\yhsp_m=\xhs_m+r,\; \xhsp_m=\xhs_m+r\md I,\vec\xhs
		\bigr)\\&\hspace{120pt}+
		\prob\bigl(
		\yhsp_m=\xhs_m+r,\; \xhsp_m=\xhs_m+r-1\md I,\vec\xhs
		\bigr)
		\\&=
		\pgeom_{g, a_m\al}(r)
		\,\LLL_{-\be,a_m,\al a_m}(g-r,0;g-r,0)
		\\&\hspace{120pt}+
		\pgeom_{g, a_m\al}(r-1)
		\,\LLL_{-\be,a_m,\al a_m}(g-r+1,0;g-r,1).
		\end{split}
	\end{equation*}
	A straightforward calculation shows that this is equal to 
	\begin{equation*}
		\begin{split}
		&=
		\frac{a_m\be(1-q^{g})}{1+a_m\be}\,
		\pgeom_{g-1, a_m\al}(r-1)
		+
		\frac{1+a_m\be q^{g}}{1+a_m\be}\,
		\pgeom_{g, a_m\al}(r)
		\\&=
		\prob(\yhs_m=\xhs_m+1\md I,\vec\xhs)\,\pgeom_{g-1, a_m\al}(r-1)
		+
		\prob(\yhs_m=\xhs_m\md I,\vec\xhs)\,\pgeom_{g, a_m\al}(r)
		\\&=
		\prob\bigl(
		\yhsbg_m=\xhs_m+r\md I,\vec\xhs
		\bigr).
		\end{split}
	\end{equation*}
	Similarly, conditioning on the complement $I^{c}$ we have
	\begin{equation*}
	\begin{split}
		\prob\bigl(
		\yhsp_m=\xhs_m+r\md I^{c},\vec\xhs
		\bigr)
		&=
		\pgeom_{g, a_m\al}(r)
		\,\LLL_{-\be,a_m,\al a_m}(g-r,1;g-r+1,0)
		\\&\hspace{100pt}+
		\pgeom_{g, a_m\al}(r-1)
		\,\LLL_{-\be,a_m,\al a_m}(g-r+1,1;g-r+1,1)
		\\&=
		\frac{a_m\be}{1+a_m\be}\,
		\pgeom_{g, a_m\al}(r-1)+
		\frac{1}{1+a_m\be}\,
		\pgeom_{g+1, a_m\al}(r)
		\\&=
		\prob\bigl(
		\yhsbg_m=\xhs_m+r\md I^{c},\vec\xhs
		\bigr),
	\end{split}
	\end{equation*}
	which establishes the claim.
\end{proof}

\begin{proposition}\label{prop:main_capling_prop2}
	Fix $L\ge1$ and $m\in\{1,\ldots,L\}$. Given the particle configuration $\vec\xhs$, the joint distribution of $(\xhsp_{m},\yhsp_m)$ is the same as of $(\xhsp_{m},\yhsgb_m)$.
\end{proposition}
\begin{proof}
	Define
	\begin{equation*}
	 	p_A:=\prob\bigl(\yhsp_m=\xhsp_m+1\md \vec\xhs,\xhsp_m\bigr),\qquad
	 	p_B:=\prob\bigl(\yhsgb_m=\xhsp_m+1\md \vec\xhs,\xhsp_m\bigr).
	\end{equation*}
	We aim to show that $p_A=p_B$ by induction on $m$. Because $\yhsp_m-\xhsp_m, \yhsgb_m-\xhs_m\in\{0,1\}$ by the very	construction, this will imply the claim. 

	When $m=1$ we have:
	\begin{equation*}
		\prob\bigl(
		\xhsp_1-\xhs_1=r,\;
		\yhsp_1-\xhsp_1=1
		\md\vec\xhs\bigr)=
		\pgeom_{+\infty, a_1\al}(r)\frac{a_1\be}{1+a_1\be},
		\qquad r\in\Z_{\ge0},
	\end{equation*}
	that is, $\xhsp_1-\xhs_1\in\Z_{\ge0}$ and $\yhsp_1-\xhsp_1\in\{0,1\}$ are independent random variables which can be interpreted as jumps of the first particle under $\mathbf{G}_\al$ and $\mathbf{B}_\be$, respectively. Thus, their distribution is the same as for $(\xhsp_{1}-\xhs_1,\yhsgb_1-\xhsp_1)$. This establishes the case $m=1$.

	Let now $m\ge 2$. We need some notation. For any $k=1,\ldots,L-1$ and $j\ge0$, define events
	\begin{equation*}
		\begin{split}
			A_{0,j}^{k}&:=
			\bigl\{\yhs_k=\xhs_k,\;\yhsbg_k=\yhs_k+j\bigr\}
			,\qquad A_{1,j-1}^{k}:=
			\bigl\{\yhs_k=\xhs_k+1,\;\yhsbg_k=\yhs_k+j-1\bigr\};\\
			B_{j,0}^{k}&:=
			\bigl\{\xhsp_k=\xhs_k+j,\;\yhsgb_k=\xhsp_k\bigr\},\qquad 
			B_{j-1,0}^{k}:=
			\bigl\{\xhsp_k=\xhs_k+j-1,\;\yhsgb_k=\xhsp_k+1\bigr\}
			.
		\end{split}
	\end{equation*}
	If $j=0$, then by agreement the events $A_{1,j-1}^{k}$ and $B_{j-1,0}^{k}$ are empty. The commutation of the Bernoulli and the geometric moves (\Cref{prop:commute}) implies that 
	\begin{equation}\label{main_capling_prop2_proof1}
		\prob\bigl(A_{0,j}^{k}\bigr)+\prob\bigl(A_{1,j-1}^{k}\bigr)=
		\prob\bigl(B_{j,0}^{k}\bigr)+\prob\bigl(B_{j-1,1}^{k}\bigr).
	\end{equation}
	Next, denote 
	\begin{equation*}
		p_A^{j}:=\prob\bigl(\yhsp_m=\xhsp_m+1\md \vec\xhs,\xhsp_m,
		A_{0,j}^{m-1}\sqcup A_{1,j-1}^{m-1}\bigr),
		\qquad
		p_B^{j}:=
		\prob\bigl(\yhsgb_m=\xhsp_m+1\md \vec\xhs,\xhsp_m,
		B_{j,0}^{m-1}\sqcup B_{j-1,1}^{m-1}\bigr).
	\end{equation*}
	It suffices to show that $p_A^{j}=p_B^{j}$ because the additional conditioning is onto events of the same probability
	which partition the probability space as $j\ge0$ varies.
	
	First assume that $j=0$. Then by \Cref{def:main_capling_construction} we have
	\begin{equation*}
		p_A^{0}=\frac{\be a_m(1-q^{\xhs_{m-1}-\xhsp_m-1})}{1+\be a_m},
	\end{equation*}
	which is the same as $p_B^{0}$ because $\xhs_{m-1}=\xhsp_{m-1}$ when conditioned on $B^{m-1}_{0,0}$.

	Now let $j\ge1$. Then we can express $p_A^{j}$ and $p_B^{j}$ as follows:
	\begin{equation*}
		\begin{split}
			p_A^{j}\prob\bigl(A_{0,j}^{m-1}\sqcup A_{1,j-1}^{m-1}\bigr)&=
			\frac{a_m\be(1-q^{\xhs_{m-1}-\xhsp_m-1})}{1+a_m\be}
			\prob\bigl(A_{0,j}^{m-1}\bigr)
			+
			\frac{a_m\be+a_m\al q^{\xhs_{m-1}-\xhsp_m-1}}{1+a_m\be}
			\prob\bigl(A_{1,j-1}^{m-1}\bigr)
			;
			\\
			p_B^{j}\prob\bigl(B_{j,0}^{m-1}\sqcup B_{j-1,1}^{m-1}\bigr)&=
			\frac{a_m\be(1-q^{\xhs_{m-1}+j-\xhsp_m-1})}{1+a_m\be}
			\prob\bigl(B_{j,0}^{m-1}\bigr)
			+
			\frac{a_m\be}{1+a_m\be}
			\prob\bigl(B_{j-1,1}^{m-1}\bigr)
			.
		\end{split}
	\end{equation*}
	The coefficients in the first line come from the vertex weights according to \Cref{def:main_capling_construction}, and the ones in the second line correspond to the Bernoulli $q$-TASEP move $\mathbf{B}_{\be}^{\GB{.12}}$. We also used the fact that $\xhsp_{m-1}=\xhs_{m-1}+j$ when conditioned on $B_{j,0}^{m-1}$. 

	By \eqref{main_capling_prop2_proof1}, to prove $p_A^j=p_B^j$ it suffices to show that
	\begin{equation}\label{main_capling_prop2_proof2}
		\be \prob\bigl(A_{0,j}^{m-1}\bigr)-\al
		\prob\bigl(A_{1,j-1}^{m-1}\bigr)=\be q^{j}
		\prob\bigl(B_{j,0}^{m-1}\bigr)
	\end{equation}
	(because $m\ge2$, we could divide by $q^{\xhs_{m-1}-\xhsp_{m}-1}\ne 0$). 

	For $m=2$, the probabilities in \eqref{main_capling_prop2_proof2} can be readily written down, and it turns into
	\begin{equation*}
		\pgeom_{+\infty,a_1\al}(j)-a_1\al \pgeom_{+\infty,a_1\al}(j-1)
		=q^{j}\pgeom_{+\infty,a_1\al}(j).
	\end{equation*}
	From \eqref{q_geom_distr_infinity} we see that this holds.

	Finally, when $m\ge3$, let us show that identity \eqref{main_capling_prop2_proof2} holds when we condition the probabilities in the left-hand side on $A^{m-2}_{0,\ell}\sqcup A^{m-2}_{1,\ell-1}$, and in the right-hand side on $B^{m-2}_{\ell,0}\sqcup B^{m-2}_{\ell-1,1}$, where $\ell\in\Z_{\ge0}$ is fixed. Because the probabilities of conditions are the same by \eqref{main_capling_prop2_proof1}, this would imply \eqref{main_capling_prop2_proof2}. Denote $g:=\xhs_{m-2}-\xhs_{m-1}-1$. The conditioned version of \eqref{main_capling_prop2_proof2} is equivalent to
	\begin{equation}\label{main_capling_prop2_proof3}
	\begin{split}
		&\be \frac{1+\be a_{m-1} q^{g}}{1+\be a_{m-1}}\,
		\pgeom_{g,a_{m-1}\al}(j)
		\prob\bigl(A_{0,\ell}^{m-2}\bigr)
		+
		\be \frac{1}{1+\be a_{m-1}}\,\pgeom_{g+1,a_{m-1}\al}(j)
		\prob\bigl(A_{1,\ell-1}^{m-2}\bigr)
		\\
		&\hspace{10pt}-
		\al \frac{\be a_{m-1}(1-q^{g})}{1+\be a_{m-1}}
		\,\pgeom_{g-1,a_{m-1}\al}(j-1)
		\prob\bigl(A_{0,\ell}^{m-2}\bigr)
		-
		\frac{\be a_{m-1}}{1+\be a_{m-1}}\,
		\pgeom_{g,a_{m-1}\al}(j-1)
		\prob\bigl(A_{1,\ell-1}^{m-2}\bigr)
		\\
		&\hspace{30pt}=
		\be q^{j}
		\frac{1+\be a_{m-1}q^{g+\ell-j}}{1+\be a_{m-1}}
		\,\pgeom_{g,a_{m-1},\al}(j)
		\prob\bigl(B_{\ell,0}^{m-2}\bigr)
		+
		\be q^{j}
		\frac{1}{1+\be a_{m-1}}\,\pgeom_{g,a_{m-1},\al}(j)
		\prob\bigl(B_{\ell-1,1}^{m-2}\bigr).
	\end{split}
	\end{equation}
	(here we assumed that $\ell\ge1$, the case $\ell=0$ is similar and simpler). Rewrite the probabilities in the right-hand side of \eqref{main_capling_prop2_proof3} using \eqref{main_capling_prop2_proof1} and the induction hypothesis \eqref{main_capling_prop2_proof2} for $m-2$:
	\begin{multline*}
		\mathrm{RHS}\eqref{main_capling_prop2_proof3}=
		\bigl(\be \prob\bigl(A_{0,\ell}^{m-2}\bigr)
		-\al\prob\bigl(A_{1,\ell-1}^{m-2}\bigr)\bigr) 
		\frac{\be a_{m-1}q^{g}}{1+\be a_{m-1}}\,\pgeom_{g,a_{m-1},\al}(j)
		\\+
		\bigl(\prob\bigl(A_{0,\ell}^{m-2}\bigr)+
		\prob\bigl(A_{1,\ell-1}^{m-2}\bigr)\bigr)
		\be q^{j}
		\frac{1}{1+\be a_{m-1}}\,\pgeom_{g,a_{m-1},\al}(j).
	\end{multline*}
	Collecting the coefficients by $\prob\bigl(A_{0,\ell}^{m-2}\bigr)$ and $\prob\bigl(A_{1,\ell-1}^{m-2}\bigr)$ in both sides and simplifying, we arrive at the following identities:
	\begin{equation*}
		\begin{split}
			\pgeom_{g,a_{m-1}\al}(j)
			-
			\al {a_{m-1}(1-q^{g})}\pgeom_{g-1,a_{m-1}\al}(j-1)
			&=
			q^{j}
			\pgeom_{g,a_{m-1},\al}(j);\\
			\pgeom_{g+1,a_{m-1}\al}(j)
			-
			\al
			a_{m-1}\pgeom_{g,a_{m-1}\al}(j-1)
			&=
			-\al
			a_{m-1}q^{g}\pgeom_{g,a_{m-1},\al}(j)
			+
			q^{j}
			\pgeom_{g,a_{m-1},\al}(j).
		\end{split}
	\end{equation*}
	These identities can be readily verified using the definition \eqref{q_geom_distr}. This completes the proof of the proposition.
\end{proof}

\begin{remark}\label{rmk:square_joint_distribution_of_TASEPs_not_indep}
	One can construct a coupling between $\yhsgbvec$ and $\yhsbgvec$ (i.e., define both random configurations on the same probability space such that $\yhsgbvec=\yhsbgvec$). This coupling is not unique. The simplest such construction (going back to \cite{DiaconisFill1990}) would be to first sample, say, $\vec\yhs$ and $\yhsbgvec$, and then given $\yhsgbvec=\yhsbgvec$ sample $\xhspvec$ as a middle point in the sequence of Markov transitions $\vec\xhs\to\xhspvec\to\yhsgbvec$, cf. \Cref{fig:qTASEPs_sampling}. However, this coupling of $\yhsgbvec$ and $\yhsbgvec$ would eliminate the independence of $\vec\yhs$ and $\xhspvec$ (given $\vec\xhs$), a property which we employed in the proofs \Cref{prop:main_capling_prop,prop:main_capling_prop2}. It is not clear whether there exists a joint distribution (given $\vec\xhs$) of all three configurations $\vec\yhs,\xhspvec$, and $\yhsgbvec=\yhsbgvec$ under which $(\vec\yhs,\yhsbgvec)$ and $(\xhspvec,\yhsgbvec)$ have the same distributions as before, and, moreover, the corresponding extensions of \Cref{prop:main_capling_prop,prop:main_capling_prop2} hold. We will not discuss this question here.
\end{remark}

% subsection capling (end)

\subsection{$q$-TASEP~/~vertex model coupling along time-like paths} % (fold)
\label{sub:capling_relating_joint_distributions}

Let us return to the stochastic higher spin six vertex model with the step Bernoulli boundary condition defined in \Cref{sub:vertex_weights,sub:step_and_bernoulli_boundary_conditions}. Recall that it depends on parameters $a_i,\nu_i$, and $u_j$ satisfying \eqref{vertex_model_parameters}, where $\nu_1=0$. Also recall the notation $\nua_i=\nu_i/a_i$, $i\ge2$. The configuration of the higher spin six vertex model with the step Bernoulli boundary condition in $\Z_{\ge2}\times\Z_{\ge1}$ is completely described by the height function, which is a collection of random variables $\{\HeightFunctionBernoulli(N+1,T)\}$, where $N\in\Z_{\ge1}$, $T\in\Z_{\ge0}$. The step Bernoulli boundary condition translates into the following boundary values of the height function:
\begin{equation}\label{Bernoulli_height_function_boundary_values}
	\HeightFunctionBernoulli(N+1,0)=0,\quad N\ge1;
	\qquad \qquad
	\HeightFunctionBernoulli(2,T)=
	\mathsf{b}_1+\ldots+\mathsf{b}_T
	,\quad T\ge0
	,
\end{equation}
where $\mathsf{b}_i$ are independent Bernoulli random variables with $\prob(\mathsf{b}_i=1)=-a_1u_i/(1-a_1u_i)$.

\begin{definition}\label{def:up_right_mixed_TASEP_joint}
	Take any \emph{time-like path} $\mathcal{P}=\{(N_{\ttime},T_{\ttime})\}_{\ttime=0}^{\mathsf{M}}$ in the quadrant $\Z_{\ge1}\times\Z_{\ge0}$ defined as
	\begin{equation}\label{up_right_path}
		(N_0,T_0)=(1,0),\quad
		N_{\ttime+1}-N_{\ttime}\in\{0,1\},\quad
		T_{\ttime+1}-T_{\ttime}\in\{0,1\},\quad
		N_{\ttime+1}-N_{\ttime}+T_{\ttime+1}-T_{\ttime}=1.
	\end{equation}
	Define the mixed geometric/Bernoulli $q$-TASEP 
	$\vec\x(N_{\ttime},T_{\ttime})$, $\ttime=0,\ldots,\mathsf{M}$ (i.e., $\ttime$ plays the role of the discrete time) with particle rates $\{a_i\}$, which is started from the step initial configuration $\x_i(1,0)=-i$, $i\ge1$, and evolves as follows:
	\begin{enumerate}[\quad$\bullet$]
		\item If $T_{\ttime+1}=T_{\ttime}+1$, then the time increment $\ttime\to\ttime+1$ corresponds to the move $\mathbf{B}_{-u_{T_{\ttime+1}}}$;
		\item If $N_{\ttime+1}=N_{\ttime}+1$, then the time increment $\ttime\to\ttime+1$ corresponds to the move $\mathbf{G}_{\nua_{N_{\ttime+1}}}$.
	\end{enumerate}
	About the name ``time-like'' see the discussion in the end of \Cref{sub:intro_coupling}. We are not using the name ``up-right path'' for $\mathcal{P}$ to distinguish it from the up-right paths formed by arrows in configurations of the higher spin six vertex model (cf. \Cref{fig:HS}, left).
\end{definition}
Note that by \Cref{prop:commute} the marginal distribution of the configuration $\vec\x(N_{\mathsf{M}},T_{\mathsf{M}})$ depends on the path $\mathcal{P}$ only through its endpoint $(N_{\mathsf{M}},T_{\mathsf{M}})$. However, the joint distribution of the sequence of configurations $\{\vec\x(N_{\ttime},T_{\ttime})\}_{\ttime=0}^{\mathsf{M}}$ depends on the whole path $\mathcal{P}$.

\begin{theorem}\label{thm:capling}
	For any time-like path $\mathcal{P}$, the joint distribution of the random variables 
	\begin{equation}\label{higher_spin_up_right}
		\HeightFunctionBernoulli\bigl({\mathcal{P}}+\mathbf{e}_1\bigr):=\bigl\{\HeightFunctionBernoulli(N_{\ttime}+1,T_\ttime)\bigr\}_{\ttime=0}^{\mathsf{M}},\qquad \mathbf{e}_1:=(1,0),
	\end{equation}
	corresponding to the inhomogeneous stochastic higher spin six vertex model with the step Bernoulli boundary condition is the same as the joint distribution of the (shifted) particle locations\footnote{We use notation $\X$ instead of $\vec\x,\x$ because of the shifting of the particle coordinates.} 
	\begin{equation}\label{TASEP_up_right}
		\X({\mathcal{P}}):=\bigl\{\x_{N_{\ttime}}(N_{\ttime},T_{\ttime})+N_{\ttime}\bigr\}_{\ttime=0}^{\mathsf{M}}
	\end{equation}
	under the mixed geometric/Bernoulli $q$-TASEP described above.
\end{theorem}
By specializing $a_i\equiv 1$, $u_i\equiv -\be$, $i\ge1$, and $\nu_j\equiv \al$, $j\ge2$, \Cref{thm:capling} implies Theorem \ref{thm:intro_capling} from the Introduction.
\begin{proof}[Proof of \Cref{thm:capling}]
	Let $\bar{\mathcal{P}}:=\mathcal{P}\cup\{(N_{\mathsf{M}},T_{\mathsf{M}}+1),(N_{\mathsf{M}}+1,T_{\mathsf{M}})\}$, call it a \emph{forked time-like path} with endpoints $(N_{\mathsf{M}},T_{\mathsf{M}}+1)$ and $(N_{\mathsf{M}}+1,T_{\mathsf{M}})$. See \Cref{fig:2_way_splitting_path}, left. Define
	\begin{equation*}
	\begin{split}
		\HeightFunctionBernoulli\bigl({\bar{\mathcal{P}}}+\mathbf{e}_1\bigr)&:=\bigl\{\HeightFunctionBernoulli(N+1,T)\colon (N,T)\in\bar{\mathcal{P}}\bigr\};\\
		\X({\bar{\mathcal{P}}})&:=\X({\mathcal{P}})\cup\bigl\{\y_{N_{\mathsf{M}}}+N_{\mathsf{M}},
		\x_{N_{\mathsf{M}}+1}'+N_{\mathsf{M}}+1\bigr\},
	\end{split}
	\end{equation*}
	where the configuration $\vec\y$ is obtained from $\vec\x(N_{\mathsf{M}},T_{\mathsf{M}})$ via a Bernoulli move $\mathbf{B}_{-u_{T_{\mathsf{M}}+1}}$, the configuration $\vec\x'$ is obtained from $\vec\x(N_{\mathsf{M}},T_{\mathsf{M}})$ via a geometric move $\mathbf{G}_{\nua_{N_{\mathsf{M}}+1}}$, and the moves $\mathbf{B}_{-u_{T_{\mathsf{M}}+1}}$ and $\mathbf{G}_{\nua_{N_{\mathsf{M}}+1}}$ are independent given $\vec\x(N_{\mathsf{M}},T_{\mathsf{M}})$.

	\begin{figure}[htbp]
		\begin{tikzpicture}
			[scale=.7,very thick]
			\draw[->] (-.65,0)--++(7.5,0) node [right] {$N$};
			\draw[->] (0,-.65)--++(0,5.5) node [above left] {$T$};
			\foreach \nn in {1,...,7}
			{\node at (\nn-1,-1) {$\nn$};}
			\foreach \tt in {0,...,4}
			{\node at (-1,\tt) {$\tt$};}
			\foreach \nn in {0,...,6}
			{\foreach \tt in {0,...,4}
			{\draw[fill] (\nn,\tt) circle (4pt);}}
			\draw[line width=3.5pt,color=red,->] (0,0)
			--++(2,0)--++(0,1)--++(1,0)--++(0,1);
			\draw[line width=3.5pt,color=red,->] (3,2)--++(0,1);
			\draw[line width=3.5pt,color=red,->] (3,2)--++(1,0);
		\end{tikzpicture}\hspace{30pt}
		\begin{tikzpicture}
			[scale=.7,very thick]
			\draw[->] (-.65,0)--++(7.5,0) node [right] {$N$};
			\draw[->] (0,-.65)--++(0,5.5) node [above left] {$T$};
			\foreach \nn in {1,...,7}
			{\node at (\nn-1,-1) {$\nn$};}
			\foreach \tt in {0,...,4}
			{\node at (-1,\tt) {$\tt$};}
			\foreach \nn in {0,...,6}
			{\foreach \tt in {0,...,4}
			{\draw[fill] (\nn,\tt) circle (4pt);}}
			\draw[densely dashed] (.5,.5)--++(5,0)--++(0,1)--++(-1,0)--++(0,1)--++(-1,0)--++(0,1)--++(-3,0)--cycle;
			\node at (3.55,3.85) {$Y$};
			\draw[densely dashed] (-.5,4.5)--(-.5,-.5)--++(7,0);
			\draw[densely dashed] (.5,4.5)--(.5,.5)--++(6,0);
			\foreach \pt in {(1,4),(4,3),(5,2),(6,1)}
			{\draw \pt circle (8pt);}
		\end{tikzpicture}
		\caption{Left: Forked time-like path $\bar{\mathcal{P}}$ with $(N_{\mathsf{M}},T_{\mathsf{M}})=(4,2)$. Right: the sets $Y$ and $Y^{\sim}$ (inner corners are circled).}
		\label{fig:2_way_splitting_path}
	\end{figure}
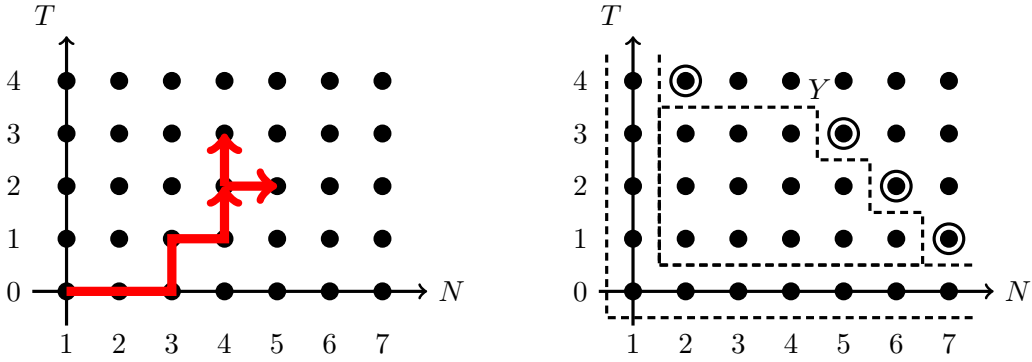
	
	Let $Y\subset\Z_{\ge2}\times\Z_{\ge1}$ be a finite subset such that with any $(i,j)\in Y$ it contains all $(i',j')$ with $2\le i'\le i$, $1\le j'\le j$. (In other words, $Y$ is a Young diagram.) Let also $Y^{\sim}:=Y\cup(\Z_{\ge1}\times\{0\})\cup(\{1\}\times\Z_{\ge0})$. See \Cref{fig:2_way_splitting_path}, right.

	By induction on adding points to $Y$, we aim to show that for any $Y$, the distributions of $\HeightFunctionBernoulli({\bar{\mathcal{P}}}+\mathbf{e}_1)$ and $\X({\bar{\mathcal{P}}})$ coincide for any forked time-like path $\bar{\mathcal{P}}$ contained inside $Y^{\sim}$. This statement is stronger than the theorem because one can restrict $\HeightFunctionBernoulli({\bar{\mathcal{P}}}+\mathbf{e}_1)$ and $\X({\bar{\mathcal{P}}})$ to $\HeightFunctionBernoulli({{\mathcal{P}}}+\mathbf{e}_1)$ and $\X({{\mathcal{P}}})$, respectively (where $\mathcal{P}$ is the original time-like path).

	The induction base is $Y=\varnothing$. Then in the set $Y^{\sim}$ there is only one forked time-like path $\{(1,0),(1,1),(2,0)\}$. The joint distribution of $\{\HeightFunctionBernoulli(2,0),\HeightFunctionBernoulli(2,1),\HeightFunctionBernoulli(3,0)\}$ is prescribed by the boundary condition \eqref{Bernoulli_height_function_boundary_values}, and is the same as for $\{\x_{1}(1,0)+1,\y_1+1,\x_{2}'+2\}$. Indeed, $\x_{1}(1,0)+1=0$ because $\vec\x(1,0)$ is the step initial configuration. Next, $\y_1+1=\mathsf{b}_1$ because of how the evolution of the first particle under the Bernoulli $q$-TASEP looks like. Finally, $\x_{2}'+2=0$ because under a single move of the geometric $q$-TASEP from the step initial configuration only the first particle can change its location.

	For the induction step, consider adding a new point, say, $A:=(N+1,T+1)$, to $Y$. We need to show that the distributions of $\HeightFunctionBernoulli({\bar{\mathcal{P}}}+\mathbf{e}_1)$ and $\X({\bar{\mathcal{P}}})$ coincide for any forked time-like path $\bar{\mathcal{P}}$ inside $(Y\cup \{A\})^{\sim}$. It suffices to consider only $\bar{\mathcal{P}}$ for which one of the endpoints is $A$. Fix such $\bar{\mathcal{P}}$. Let $B:=(N,T+1)$, $C:=(N,T)$, and $D:=(N+1,T)$. Since a point to $Y$ can be added only at an inner corner (see \Cref{fig:2_way_splitting_path}, left), we have $B,C,D\in Y^{\sim}$. Assume that $E:=(N+2,T)\in Y^{\sim}$ is the other endpoint of $\bar{\mathcal{P}}$ (if $E\notin Y^{\sim}$, there is nothing to prove). The case when the other endpoint is $(N,T+2)$ is analogous.

	\begin{figure}[htbp]
		\begin{tikzpicture}
			[scale=1.5,very thick]
			\draw[densely dashed, ultra thick] (-.25,2)--(0,2)--++(1,0)--++(0,-1)--++(2.25,0);
			\node[draw,circle] (A) at (1.5,1.5) {$A$};
			\node[draw,circle] (B) at (.5,1.5) {$B$};
			\node[draw,circle] (C) at (.5,.5) {$C$};
			\node[draw,circle] (D) at (1.5,.5) {$D$};
			\node[draw,circle] (E) at (2.5,.5) {$E$};
			\node at (.5,-1.25) {\phantom{$N$}};
			\node at (.5,-.25) {$N$};
			\node at (1.5,-.25) {$N+1$};
			\node at (2.5,-.25) {$N+2$};
			\node at (-.45,.5) {$T$};
			\node at (-.45,1.5) {$T+1$};
			\draw[red,line width=3pt,->] (C)--(D);
			\draw[red,line width=3pt,->] (D)--(A);
			\draw[red,line width=3pt,->] (D)--(E);
		\end{tikzpicture}\hspace{60pt}
		\begin{tikzpicture}
			[scale=1.5,very thick]
			\draw[densely dashed, ultra thick] (-.25,2)--(0,2)--++(1,0)--++(0,-1)--++(2.25,0);
			\node[draw,circle] (A) at (1.5,1.5) {$A$};
			\node[draw,circle] (B) at (.5,1.5) {$B$};
			\node[draw,circle] (C) at (.5,.5) {$C$};
			\node[draw,circle] (D) at (1.5,.5) {$D$};
			\node[draw,circle] (F) at (1.5,-.5) {$F$};
			\node[draw,circle] (G) at (2.5,-.5) {$G$};
			\node[draw,circle] (E) at (2.5,.5) {$E$};
			\node at (.5,-1.25) {$N$};
			\node at (1.5,-1.25) {$N+1$};
			\node at (2.5,-1.25) {$N+2$};
			\node at (-.45,.5) {$T$};
			\node at (-.45,-.5) {$T-1$};
			\node at (-.45,1.5) {$T+1$};
			\draw[red,line width=3pt,->] (F)--(D);
			\draw[red,line width=3pt,->] (D)--(A);
			\draw[red,line width=3pt,->] (D)--(E);
		\end{tikzpicture}
		\caption{Two possibilities for the forked time-like path $\bar{\mathcal{P}}$ with endpoints $A$ and $E$. The dashed line is the boundary of $Y$.}
		\label{fig:YD_proof}
	\end{figure}
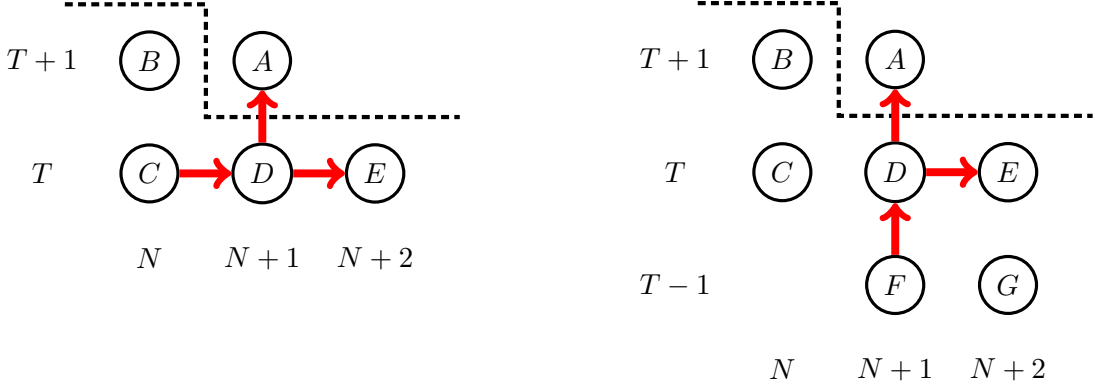

	There are two possibilities of how $\bar{\mathcal{P}}$ can approach $D$, see \Cref{fig:YD_proof}. Let $\bar{\mathcal{P}}$ pass through $C$ (the case when it passes through $F:=(N+1,T-1)$ is analogous). Because $\bar{\mathcal{P}}':=(\bar{\mathcal{P}}\cup\{B\})\setminus\{A,E\}$ is a forked time-like path inside $Y^{\sim}$, the induction hypothesis implies that the joint distribution of the random variables $\HeightFunctionBernoulli(\bar{\mathcal{P}}'+\mathbf{e}_1)$ is the same as for $\X(\bar{\mathcal{P}}')$. In other words, the values of the height function along $\bar{\mathcal{P}}'+\mathbf{e}_1$ can be sampled using $q$-TASEP moves. Let $\bar{\mathcal{P}}'':=(\bar{\mathcal{P}}'\cup\{A\})\setminus\{B\}=\bar{\mathcal{P}}\setminus\{E\}$. By \Cref{prop:main_capling_prop2}, the values of the height function along $\bar{\mathcal{P}}''+\mathbf{e}_1$ also can be sampled using the corresponding $q$-TASEP moves. Now, given $\vec\x(D)=\vec\x(N+1,T)$, sample $\vec\x(E)=\vec\x(N+2,T)$ independently of $\vec\x(A)$ using the geometric $q$-TASEP move $\mathbf{G}_{\nua_{N+2}}$, and put $\widehatHeightFunctionBernoulli(E+\mathbf{e}_1)=\widehatHeightFunctionBernoulli(N+3,T):=\x_{N+2}(E)+N+2$. By the Markov property, $\widehatHeightFunctionBernoulli(E+\mathbf{e}_1)$ is independent (given $\vec\x(D)$) of the $q$-TASEP configurations $\vec \x$ corresponding to all points of $\bar{\mathcal{P}}''\setminus\{A,D\}=\bar{\mathcal{P}}\setminus\{A,D,E\}$.

	It remains to show that the joint distribution of $\X(\bar{\mathcal{P}}'')\cup\{\widehatHeightFunctionBernoulli(E+\mathbf{e}_1)\}$ is the same as for $\HeightFunctionBernoulli(\bar{\mathcal{P}}''+\mathbf{e}_1)\cup\{\HeightFunctionBernoulli(E+\mathbf{e}_1)\}$, where $\HeightFunctionBernoulli(E+\mathbf{e}_1)$ comes from the higher spin six vertex model. If $T=0$, this follows similarly to the induction base. For $T\ge1$, $\HeightFunctionBernoulli(E+\mathbf{e}_1)$ is sampled knowing $\HeightFunctionBernoulli(D+\mathbf{e}_1)$, $\HeightFunctionBernoulli(F+\mathbf{e}_1)$, $\HeightFunctionBernoulli(G+\mathbf{e}_1)$ using the vertex weights. Consider any forked time-like path $\bar{\mathcal{P}}^{\vee}$ with endpoints $D$ and $G:=(N+2,T-1)$. It lies inside $Y^{\sim}$. By the induction hypothesis and \Cref{prop:main_capling_prop}, the joint distribution of the values of the height function along $\bigl((\bar{\mathcal{P}}^{\vee}\cup\{E\})\setminus\{G\}\bigr)+\mathbf{e}_1$ can be sampled using the corresponding $q$-TASEP moves. Restricting to $D$ and $E$, we see that the pair $\{\HeightFunctionBernoulli(D+\mathbf{e}_1),\HeightFunctionBernoulli(E+\mathbf{e}_1)\}$ can be sampled as follows. Take the configuration $\vec\x(D)$ having the marginal $q$-TASEP distribution (obtained by making $T$ Bernoulli and $N$ geometric moves in any order), make one more geometric move $\mathbf{G}_{\nua_{N+2}}$ to get $\vec\x(E)$, and set
	\begin{equation*}
		\HeightFunctionBernoulli(D+\mathbf{e}_1)=\x_{N+1}(D)+N+1
		,\qquad
		\HeightFunctionBernoulli(E+\mathbf{e}_1)=\x_{N+2}(E)+N+2
		.
	\end{equation*}
	But this is exactly how we defined $\widehatHeightFunctionBernoulli(E+\mathbf{e}_1)$ earlier. This completes the proof.
\end{proof}

\Cref{thm:capling} extends the matching between the single-point distributions of the inhomogeneous stochastic higher spin six vertex model the step Bernoulli boundary condition and a mixed geometric/Bernoulli $q$-TASEP (known from \Cref{cor:matching_Bernoulli} and \Cref{prop:qTASEP_moments}) to multi-point joint distributions along time-like paths.\footnote{Note also that the proof of \Cref{thm:capling} is independent from the contour integral, difference operator, or $q$-deformed Robinson--Schensted--Knuth considerations leading to \Cref{cor:matching_Bernoulli} and \Cref{prop:qTASEP_moments}. Thus, \Cref{thm:capling} presents yet another way of proving single-point $q$-moment formulas.} Moreover, this theorem provides a (partial) coupling of the two models in the sense that the random variables $\HeightFunctionBernoulli\bigl({\mathcal{P}}+\mathbf{e}_1\bigr)$ \eqref{higher_spin_up_right} corresponding to any fixed time-like path $\mathcal{P}$ can be regarded as functions on the space of sample paths of a mixed $q$-TASEP as in \eqref{TASEP_up_right} (for which the order of the geometric and Bernoulli moves depends on $\mathcal{P}$).

\begin{remark}
	It would be very interesting to see if the partial coupling of \Cref{thm:capling} can be extended to a full coupling between the random variables $\{\HeightFunctionBernoulli(N+1,T)\}_{N\ge1,\; T\ge0}$ and a $q$-TASEP-like particle system. Such an extension would probably require substantial new ideas (cf.~\Cref{rmk:square_joint_distribution_of_TASEPs_not_indep} presenting a clear obstacle to a straightforward approach), and we do not pursue it here.
\end{remark}

% subsection coupling_relating_joint_distributions (end)

\subsection{Extensions and remarks} % (fold)
\label{sub:more_general_boundary_conditions}

Let us conclude this section by making a number of remarks in connection with \Cref{thm:capling}. 

\subsubsection{Generalized step Bernoulli boundary condition}

First, consider the generalized step Bernoulli boundary condition for the stochastic higher spin six vertex model introduced in \cite{AmolBorodin2016Phase}. To obtain this boundary condition (of order $r\ge1$), specialize $\nu_2=\ldots=\nu_r=0$ in the model with the step Bernoulli boundary condition, and look at the arrow configuration in the shifted quadrant $\Z_{\ge r+1}\times\Z_{\ge1}$. The case $r=1$ (corresponding to no specialization of $\nu_2,\nu_3,\ldots$) is simply the step Bernoulli boundary condition.

Fix $r\ge1$, and denote the height function of the model with the generalized step Bernoulli boundary condition (of order $r$) by $\HeightFunctionGenBernoulli_r(N+r,T)$, $N\in\Z_{\ge1}$, $T\in\Z_{\ge0}$. Note that the distribution of the values of $\HeightFunctionGenBernoulli_r$ at the left boundary of $\Z_{\ge r+1}\times\Z_{\ge1}$ depends on $r$ free parameters $a_1,\ldots,a_r$ (these parameters correspond to the first $r$ columns).

For any time-like path $\mathcal{P}=\{(N_{\ttime},T_{\ttime})\}_{\ttime=0}^{\mathsf{M}}$ in the quadrant $\Z_{\ge1}\times\Z_{\ge0}$ (as in \Cref{def:up_right_mixed_TASEP_joint}), define the collection of random variables $\HeightFunctionGenBernoulli_r(\mathcal{P}+r\mathbf{e}_1):=\big\{\HeightFunctionGenBernoulli_r(N_\ttime+r,T_{\ttime})\big\}_{\ttime=0}^{\mathsf{M}}$, where $r\mathbf{e}_1=(r,0)$. Also let the $\mathcal{P}$-dependent mixed geometric/Bernoulli $q$-TASEP $\vec\x^{(r)}(N_\ttime{},T_\ttime{})$, $\ttime{}=0,1,\ldots,\mathsf{M}$, be as in \Cref{def:up_right_mixed_TASEP_joint}, but with geometric moves having shifted parameters $\nua_{r+1},\nua_{r+2},\ldots$. That is, the marginal distribution of $\vec\x^{(r)}(N,T)$ corresponds to taking $N+T-1$ moves of the $q$-TASEP (with particle rates $a_1,a_2,\ldots$), of which there are $T$ Bernoulli moves with parameters $-u_1,\ldots,-u_T$, and $N-1$ geometric moves with parameters $\nua_{r+1},\ldots,\nua_{r+N-1}$.

\begin{proposition}\label{prop:generalized_step_Bernoulli}
	With the above notation, for any time-like path $\mathcal{P}$ the joint distribution of $\HeightFunctionGenBernoulli_r(\mathcal{P}+r\mathbf{e}_1)$ is the same as the joint distribution of the shifted particle locations
	\begin{equation*}
		\X^{(r)}(\mathcal{P}):=
		\big\{
		\x^{(r)}_{N_\ttime+r-1}(N_\ttime,T_\ttime)+N_{\ttime}+r-1
		\big\}_{\ttime=0}^{\mathsf{M}}
	\end{equation*}
	in the mixed geometric/Bernoulli $q$-TASEP $\vec\x^{(r)}$ described above.
\end{proposition}
\begin{proof}
	Follows from \Cref{thm:capling} after setting $\nu_2=\ldots=\nu_r=0$ and shifting indices by $r-1$.
\end{proof}

\begin{remark}\label{rmk:BBP_type_limit_transition_to_reference}
	\Cref{prop:generalized_step_Bernoulli} can be reformulated to complement \Cref{thm:capling} as follows. Take a time-like path $\mathcal{P}$. By \Cref{thm:capling}, the values along $\mathcal{P}+\mathbf{e}_1$ of the height function $\HeightFunctionBernoulli$ with the step Bernoulli boundary condition can be identified with shifted particle locations in a certain mixed geometric/Bernoulli $q$-TASEP (depending on $\mathcal{P}$). Fix $r\ge2$. Looking at \emph{different particles in the same $q$-TASEP}, more precisely, using the $(k+r-1)$th particle whenever \Cref{thm:capling} used the $k$th particle gives the values (along $\mathcal{P}+r\mathbf{e}_1$) of the height function $\HeightFunctionGenBernoulli_r$ with the generalized step Bernoulli boundary condition of order $r$.
\end{remark}

\subsubsection{Other boundary conditions}
 
It would be very interesting to obtain analogues of \Cref{thm:capling} and \Cref{prop:generalized_step_Bernoulli} for the case of the step boundary condition (i.e., when a new arrow enters at each horizontal on the left boundary). Note that the step Bernoulli boundary condition gives rise to the step boundary condition in the limit as $a_1\to+\infty$. However, the statement of \Cref{thm:capling} does not survive such a limit transition because for sufficiently large $a_1$ the jumping distribution \eqref{q_geom_distr_infinity} of the first particle under the geometric $q$-TASEP does not make sense. It seems likely that to construct a coupling in the case of the step boundary condition in the stochastic higher spin six vertex model one must invent suitable modifications of the Bernoulli and geometric $q$-TASEPs. This will be addressed in a future work.

\subsubsection{Stochastic six vertex model}

Taking $\nu_i\equiv 1/q$ turns the stochastic higher spin six vertex model into the stochastic six vertex model, in which the number of arrows per vertical edge is bounded by one. The stochastic six vertex model was introduced in \cite{GwaSpohn1992} and studied recently in \cite{BCG6V}, \cite{AmolBorodin2016Phase}, \cite{Amol2016Stationary}, \cite{borodin2016stochastic_MM}. It would be interesting to see if either of the approaches of the present paper (via $q$-difference operators or via coupling with $q$-TASEP-like particle systems) is applicable to the stochastic six vertex model. The moment formulas for the stochastic six vertex model (see \cite[Corollary 10.1]{BorodinPetrov2016inhom} for the most general ones) do not seem to immediately match with moment formulas of a $q$-TASEP-like particle system similarly to \Cref{cor:matching_Bernoulli}.

On the other hand, a different identification of averages of observables between the stochastic higher spin six vertex model (with the step boundary condition) and Macdonald measures was recently observed in \cite{borodin2016stochastic_MM}. For the stochastic six vertex model it provides a matching of averages of observables with a Schur measure, and leads to another way of obtaining GUE Tracy--Widom asymptotics of the stochastic six vertex model (in addition to an earlier work \cite{BCG6V}). The results of \cite{borodin2016stochastic_MM} deal with single-point distributions, and the upcoming work \cite{BorodinBufetovPrep} will provide an identification of multi-point joint distributions between the stochastic six vertex model and Hall--Littlewood processes (which are $q=0$ degenerations of the Macdonald processes).

% subsection more_general_boundary_conditions (end)

% section capling_with_q_tasep (end)
 
\section{Asymptotics of $q$-TASEP via Schur measures} % (fold)
\label{sec:asymptotics_of_q_tasep_via_schur_measures}

Here we employ results of \cite{borodin2016stochastic_MM} together with \Cref{prop:qTASEP_moments} (which also follows from \Cref{thm:capling} proven independently of \Cref{prop:qTASEP_moments}) to derive asymptotics of a mixed geometric/Bernoulli $q$-TASEP with certain special parameters using Schur measures.

\subsection{Matching to Schur measures} % (fold)
\label{sub:matching_to_schur_measures}

Setting $q=0$ in the $q$-Whittaker measures (described in \Cref{sub:_q_whittaker_measures}) turns them into \emph{Schur measures}. Let us denote the Schur measure on $\Part_n$ with one specialization $(x_1,\ldots,x_n)$ (where $x_i>0$) and another specialization $\rho$ as in \eqref{Pi_definition} with $q=0$ by $\SchurMeas_{(x_1,\ldots,x_n);\rho}$. The advantage of Schur measures is that they have a determinantal structure allowing to express arbitrary multi-point correlations of the corresponding random partition as determinants of a certain kernel. The kernel itself has a double contour integral form first written down in \cite{okounkov2001infinite}. See also the surveys \cite{Soshnikov2000}, \cite{peres2006determinantal}, \cite{Borodin2009} for a general discussion of determinantal point processes.

Consider the stochastic higher spin six vertex model with the step Bernoulli boundary condition (depending on a parameter $a_1>0$), other parameters $u_1,u_2,\ldots<0$, $a_2,a_3,\ldots>0$, and $\nu_i\equiv q\in(0,1)$, $i\ge2$. That is, we are making a special choice of all the $\nu$ parameters in the model.

\begin{proposition}\label{prop:matching_to_schur_measures}
	For any $T\ge0$, $N\ge1$, and $\zeta\in\C\setminus\R_{<0}$ we have\footnote{In the right-hand side, by agreement, $\la_m=0$ for $m\le0$.}
	\begin{equation}\label{matching_to_schur_measures}
		\EE_{\mathbf{HS},\,\nu_i\equiv q}^{\textnormal{step Bernoulli}}
		\bigg[\prod_{i\ge0}
		\frac{1}{1+\zeta q^{\HeightFunction(N+1,T)+i}}
		\bigg]
		=
		\EE_{\SchurMeas_{(-u_1^{-1},\ldots,-u_T^{-1});\rho_N}}
		\bigg[
		\prod_{j\ge0}
		\frac{1+\zeta q^{\la_{T-j}+j}}{1+\zeta q^j}
		\bigg],
	\end{equation}
	where in the Schur measure on the right the specialization $\rho_N$ has no nonzero alpha parameters, and its beta parameters have the form $(a_1^{-1},a_1^{-1}q,a_1^{-1}q^2,\ldots; {a_2^{-1},\ldots,a_N^{-1}})$ (see also \eqref{Schur_specialization_matching} below).
\end{proposition}
\begin{proof}
	This is obtained from the results of \cite[Section 4]{borodin2016stochastic_MM} (which correspond to the step boundary condition in the vertex model) by letting $\nu_1\to 0$ and taking into account the matching of the parameters as in \Cref{rmk:matching_with_S_XI_parameters}. Note also that in the case of the stochastic six vertex model with the (generalized) step Bernoulli boundary condition a similar identification with Schur measures was written down in \cite[Appendix B]{AmolBorodin2016Phase}.
\end{proof}
\begin{remark}
	When $\nu_i\equiv \nu\in(0,1)$ but $\nu\ne q$, an identification similar to \eqref{matching_to_schur_measures} holds, but with a certain Macdonald measure in the right-hand side, cf. \cite[Section 4]{borodin2016stochastic_MM}. For the purpose of asymptotic analysis we focus only on the Schur case.
\end{remark}

% subsection matching_to_schur_measures (end)

\subsection{$q$-TASEP with special parameters} % (fold)
\label{sub:_q_tasep_with_special_parameters}

As follows from \Cref{prop:qTASEP_moments} or \Cref{thm:capling}, under the stochastic higher spin six vertex model described above the random variable $\HeightFunction(N+1,T)$ has the same distribution as $\x_N(N,T)+N$. Here $\vec\x(N,T)$ is the configuration of the mixed geometric/Bernoulli $q$-TASEP (started from the step initial configuration) after $T$ Bernoulli moves with \emph{arbitrary parameters} $\mathbf{B}_{-u_1},\ldots,\mathbf{B}_{-u_T}$ and $N-1$ geometric moves with \emph{special parameters} $\mathbf{G}_{q/a_2},\ldots,\mathbf{G}_{q/a_N}$ (by \Cref{prop:commute}, the order of these moves does not matter).

Thus, \Cref{prop:qTASEP_moments,prop:matching_to_schur_measures} together imply a matching of averages of observables between the $q$-Whittaker measure 
\begin{equation}\label{Whittaker_specialization_matching}
	\WhittakerMeas_{(a_1,\ldots,a_N);\bar\rho(N,T)},
	\qquad
	\Pi_{\WhittakerMeas}(a_1,\ldots,a_N;\bar\rho(N,T))=
	\prod_{k=1}^N\biggl(
	\prod_{i=1}^{T}(1-a_ku_i)
	\prod_{j=2}^{N}\frac{1}{(qa_ka_j^{-1};q)_{\infty}}
	\biggr)
\end{equation}
and the Schur measure
\begin{equation}\label{Schur_specialization_matching}
	\SchurMeas_{(-u_1^{-1},\ldots,-u_T^{-1});\rho_N},
	\qquad
	\Pi_{\SchurMeas}(-u_1^{-1},\ldots,-u_T^{-1};\rho_N)=
	\prod_{k=1}^T\biggl(
	(a_1^{-1}u_k^{-1};q)_{\infty}\prod_{j=2}^N(1-a_j^{-1}u_k^{-1})
	\biggr).
\end{equation}
Here we gave normalization constants for both measures (see \eqref{qWhittaker_measure} for the $q$-Whittaker case, and the Schur measure is given by a similar product of two Schur symmetric functions). The matching of averages of observables between \eqref{Whittaker_specialization_matching} and \eqref{Schur_specialization_matching} is based on a comparison of contour integral formulas, and it would be very interesting to find an independent argument implying this matching.

Put $u_j\equiv u<0$, $j\ge1$, and $a_i\equiv 1$, $i\ge2$ (setting $a_i\equiv a$ for some $a>0$ is essentially equivalent to having $a=1$). Recall that the rate of the first particle is still $a_1>1$ which can differ from $1$. Then the jumping distribution \eqref{q_geom_distr} of the second, third, etc.{} particle in the geometric $q$-TASEP simplifies to
\begin{equation}\label{q_geom_distr_special}
	\pgeom_{m,q}(j)=\frac{q^{j}(q;q)_m}{(q;q)_j},\qquad j=0,1,\ldots,m.
\end{equation}

% subsection _q_tasep_with_special_parameters (end)

\subsection{Asymptotic behavior} % (fold)
\label{sub:asymptotics}

Let the parameters of the mixed geometric/Bernoulli $q$-TASEP $\vec\x(N,T)$ be as above. Set $T=\lfloor \tau M \rfloor$, $N=\lfloor \eta M \rfloor$ for some $\tau,\eta>0$, where $M\to+\infty$. Let $\la^{(M)}$ be a random partition distributed according to the Schur measure as in the right-hand side of \eqref{matching_to_schur_measures}, with $(-u_1^{-1},\ldots,-u_T^{-1})=(-u^{-1},\ldots,-u^{-1})$ ($\lfloor \tau M \rfloor$ times), and such that the specialization $\rho_N=\rho_{\lfloor\eta M\rfloor}$ has beta parameters $(a_1^{-1},a_1^{-1}q,a_1^{-1}q^2,\ldots;1,\ldots,1)$ (with $1$ appearing $\lfloor \eta M \rfloor-1$ times). Let $\ell(\la^{(M)})$ denote the length of this partition (i.e., the number of nonzero parts).

\begin{proposition}\label{prop:asymptotically_equivalent}
	With the above notation the sequences of random variables 
	\begin{equation*}
		\x_{\lfloor \eta M \rfloor}(\lfloor \eta M \rfloor,\lfloor \tau M \rfloor)+\eta M
		\qquad\textnormal{and}\qquad
		\tau M-\ell(\la^{(M)})
	\end{equation*}
	are asymptotically equivalent as $M\to+\infty$ in the sense of \textnormal{\cite[Definition 5.2]{borodin2016stochastic_MM}}.
\end{proposition}
\begin{proof}
	This is \cite[Corollary 5.11]{borodin2016stochastic_MM} applied to our concrete situation.
\end{proof}

\begin{theorem}\label{thm:special_q_TASEP_asymptotics}
	Let $a_1\ge1$. With the above notation for any $\eta,\tau>0$ we have the convergence in probability
	\begin{equation*}
		\lim_{M\to+\infty}\frac{\x_{\lfloor \eta M \rfloor}
    (\lfloor \eta M \rfloor,\lfloor \tau M \rfloor)}{M}=
		\mathcal{X}(\eta,\tau),
	\end{equation*}
	where
	\begin{equation*}
		\mathcal{X}(\eta,\tau):=\begin{cases}
			\dfrac{-u(\tau-\eta)-2\sqrt{-u\eta\tau}}{1-u},
			&\qquad \tau/\eta>-u^{-1};\\
			-\eta,&\qquad
			0<\tau/\eta\le-u^{-1}.
		\end{cases}
	\end{equation*}
	In the case $\tau/\eta>-u^{-1}$ we also have
	\begin{equation*}
		\lim_{M\to+\infty}\prob
		\Bigl(\frac{\x_{\lfloor \eta M \rfloor}(\lfloor \eta M \rfloor,\lfloor \tau M \rfloor)-M \mathcal{X}(\eta,\tau)}{\sigma_{\eta,\tau,u} M^{1/3}}\ge-r\Bigr)=F_{GUE}(r),\qquad r\in\R,
	\end{equation*}
	where $F_{GUE}$ is the GUE Tracy--Widom distribution, and $\sigma_{\eta,\tau,u}$ is given in \eqref{Schur_variance} below.
\end{theorem}
In the language of stochastic higher spin six vertex model a version of \Cref{thm:special_q_TASEP_asymptotics} for the step boundary condition was proven in \cite[Theorem 6.3]{borodin2016stochastic_MM}. Our limit shape and parameters of fluctuations are the same up to having $a=\sqrt{q}$ in \cite{borodin2016stochastic_MM} and the shifting by $\eta M$ coming from the $q$-TASEP particle coordinate. For completeness we give the necessary details below.
\begin{proof}[Proof of \Cref{thm:special_q_TASEP_asymptotics}]
	By \Cref{prop:asymptotically_equivalent}, it suffices to consider the asymptotic behavior of $\ell(\la^{(M)})$. The random configuration $\{\la^{(M)}_i-i\}_{i=1}^{\infty}\subset\Z$ is a determinantal point process whose correlation kernel has a double contour integral form \cite{okounkov2001infinite}:
	\begin{equation}\label{Schur_kernel}
		K(i,j)=\frac1{(2\pi\i)^2}\oint\oint
		\frac{(-a_1^{-1}v;q)_{\infty}}{(-a_1^{-1}w;q)_{\infty}}
		\Big(\frac{1+v}{1+w}\Big)^{\lfloor\eta M\rfloor-1}
		\Big(\frac{u+v^{-1}}{u+w^{-1}}\Big)^{\lfloor \tau M\rfloor}
		\frac{dv dw}{v^{i+1}w^{-j}(v-w)},\quad
		i,j\in\Z.
	\end{equation}
	Here the integration contours are positively oriented, the $w$ contour is contained inside $v$, and both contours encircle $0$ and $-u^{-1}>0$ and leave outside $-1$ and $-a_1q^{-\Z_{\ge0}}$.

	Let us comment on how to choose integration contours in \eqref{Schur_kernel} (by essentially applying the discussion of \cite[Section 5]{BorodinGorinSPB12} to our concrete situation). Formula for the kernel in \cite[Theorem~2]{okounkov2001infinite} can be understood as a generating function identity
	\begin{equation}\label{Schur_kernel_genfunc}
		\sum_{i,j\in\Z}K(i,j)v^iw^{-j}=
		\frac{\Pi_{\SchurMeas}(w^{-1};-\mathbf{u}^{-1})}{\Pi_{\SchurMeas}(v^{-1};-\mathbf{u}^{-1})}
		\frac{\Pi_{\SchurMeas}(v;\rho_{\lfloor\eta M\rfloor})}{\Pi_{\SchurMeas}(w;\rho_{\lfloor\eta M\rfloor})}
		\sum_{k=1}^{\infty}\Bigl(\frac wv\Bigr)^k.
	\end{equation}
	Here the specialization $\rho_{\lfloor\eta M\rfloor}$ is described above, and $-\mathbf{u}^{-1}=(-u^{-1},\ldots,-u^{-1})$ ($\lfloor\tau M\rfloor$ times). The formal meaning of \eqref{Schur_kernel_genfunc} comes from understanding each of the four products $\Pi_{\SchurMeas}$ in the right-hand side as a generating function of suitably specialized complete homogeneous symmetric functions. Then the coefficient by $v^iw^{-j}$ in the right-hand side is the kernel $K(i,j)$. The double contour integral \eqref{Schur_kernel} simply picks the desired coefficient. However, for this to work one needs to ensure that on the integration contours the analytic expressions in the integrand in \eqref{Schur_kernel} coincide with the generating series in the right-hand side of \eqref{Schur_kernel_genfunc}. If $|w|<|v|$, then the sum over $k$ in the right-hand side of \eqref{Schur_kernel_genfunc} turns into $w/(v-w)$ which is a part of the integrand. Next, the terms in the right-hand side of \eqref{Schur_kernel_genfunc} which contain poles are
	\begin{equation*}
		\Pi_{\SchurMeas}(w^{-1};-\mathbf{u}^{-1})=
		\frac1{(1+u^{-1}w^{-1})^{\lfloor\tau M\rfloor}}
		,\qquad
		\frac1{\Pi_{\SchurMeas}(w;\rho_{\lfloor\eta M\rfloor})}=
		\frac1{(-a_1^{-1}w;q)_{\infty}(1+w)^{\lfloor \eta M\rfloor-1}},
	\end{equation*}
	and to be able to expand these functions as generating series of complete homogeneous symmetric functions one should have $|u^{-1}w^{-1}|<1$ (hence $|w|>|u^{-1}|$) and $|w|<1$. For sufficiently large $|u|$ we can choose circles centered at zero for the $w$ and $v$ contours. To continue \eqref{Schur_kernel} to all possible $u\in(-\infty,0)$, note that the integral in the right-hand side is rational in $u$. One can show that the kernel $K(i,j)$ is a priori rational in $u$, too. Thus, we can continue identity \eqref{Schur_kernel} in $u$ as long as the integral in the right-hand side represents the same rational function. In particular, contour deformations not crossing poles of the integrand do not change the rational function represented by the integral, and so we see that \eqref{Schur_kernel} holds for all $u$ with the contours described above. 

	Denoting $\tilde K(i,j):=\mathbf{1}_{i=j}-K(i,j)$ (where $\mathbf{1}_{i=j}$ is the indicator that $i=j$) and using complementation as in the proof of \cite[Theorem 6.1]{borodin2016stochastic_MM} (in particular, observing that $-\ell(\la^{(M)})$ is the leftmost particle of the configuration complementary to $\{\la_i^{(M)}-i\}_{i\ge1}$), we can write 
	\begin{equation}\label{Schur_Fredholm_probability}
		\prob\bigl(-\ell(\la^{(M)})>x\bigr)=\det\bigl(\mathbf{1}-\tilde K\bigr)_{\ell^2(x,x-1,x-2,\ldots)},
	\end{equation}
	where $\mathbf{1}$ is the identity operator, and the right-hand side is a Fredholm determinant. The rest of the proof of \Cref{thm:special_q_TASEP_asymptotics} thus reduces to the asymptotic analysis of the kernel $\tilde K$.
		
	Interchanging the integration contours in $K$ \eqref{Schur_kernel} leads to an additional summand equal to the integral over the $v$ contour of the minus residue at $w=v$. Since this additional summand is precisely $\mathbf{1}_{i=j}$, we see that $\tilde K$ is given by the same double integral as in \eqref{Schur_kernel} with interchanged contours and a negative sign. 

	Asymptotic analysis of correlation kernels given by double contour integrals was first performed in \cite{Okounkov2002} and is rather standard by now, so let us indicate main computations without going into full detail. The integrand in $\tilde K(i,j)$ has the following large $M$ asymptotics:
	\begin{equation*}
		(-1)^{i-j}\frac{(-a_1^{-1}v;q)_{\infty}}{(-a_1^{-1}w;q)_{\infty}}
		\exp\Big\{M\Big(G(v;\tfrac iM)-G(w;\tfrac jM)\Big)\Big\}
		\frac{1}{v(w-v)},
	\end{equation*}
	where
	\begin{equation*}
		G(v;x):=\eta \log(1+v)+\tau \log(u+v^{-1})-x \log (-v).
	\end{equation*}
	If we find $x=x_c$ such that $G$ has a double critical point $v_c$, and deform the integration contours so that they come close to each other at $v_c$ and such that $\Re G(v;x_c)<\Re G(v_c;x_c)$, $\Re G(w;x_c)>\Re G(v_c;x_c)$ on the new contours, then the asymptotics of $\tilde K$ is determined by the contribution from a small neighborhood of $v_c$. As we will see later, this contribution produces the Airy kernel.

	A straightforward computation shows that
	\begin{equation*}
		x_c=\frac{\eta -\tau \pm 2 \sqrt{-u\eta\tau}}{1-u}
		,
		\qquad
		v_c=
		\frac{\pm(1-u)\sqrt{-u\eta\tau}-u (\eta +\tau )}{u (\tau +\eta  u)}
		,
	\end{equation*}
	produce a double critical point of $G$. To capture the leftmost point of a random configuration governed by the kernel $\tilde K$, we choose the solution $x_c$ with the minus sign.

	\begin{figure}[htbp]
		\includegraphics[width=.6\textwidth]{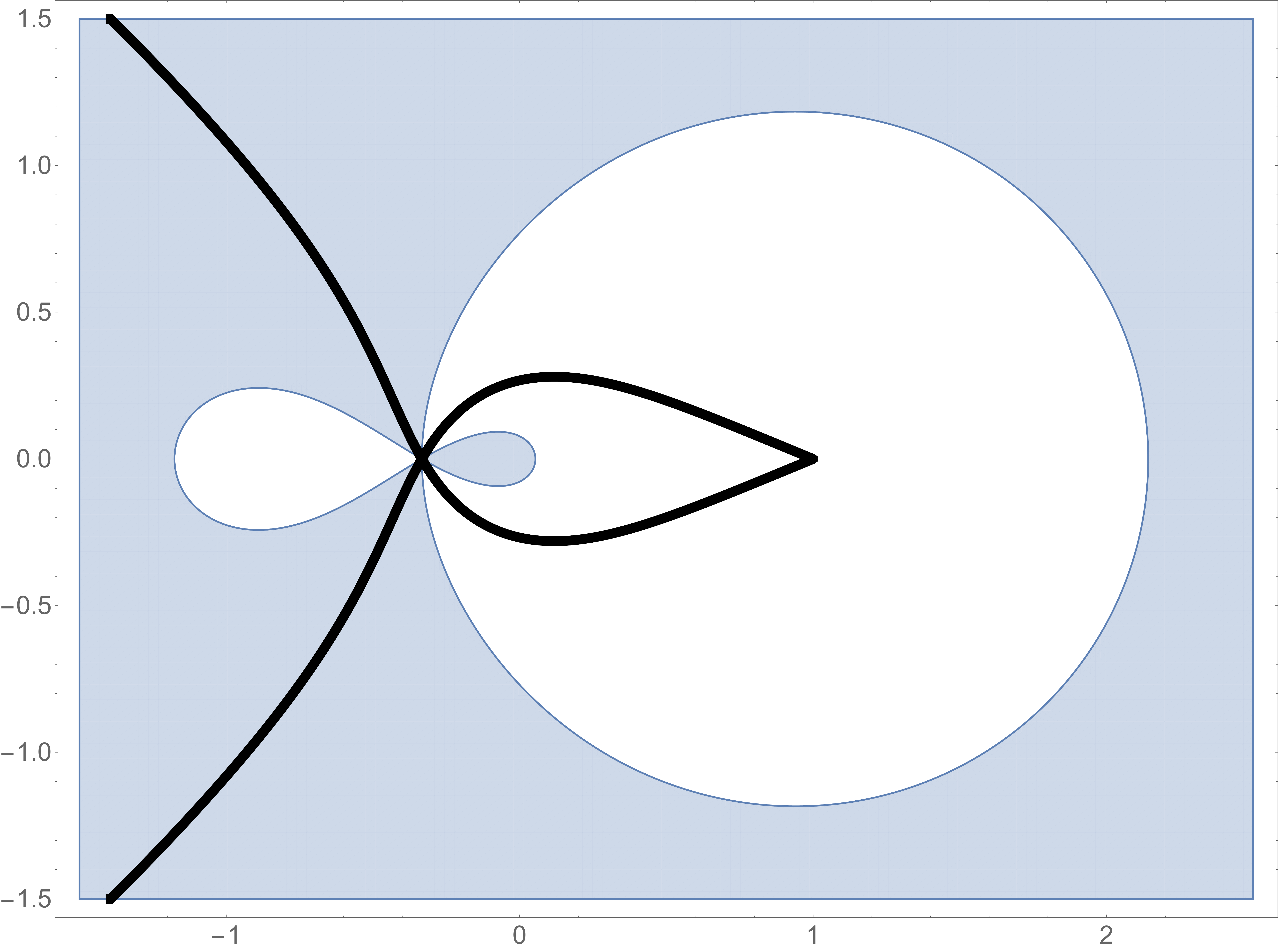}
		\caption{Contours where $\Im G(w;x_c)=\Im G(v_c;x_c)$ and $w$ is not real (they intersect at $v_c$), and regions where $\Re G(w;x_c)>\Re G(v_c;x_c)$ (shaded). The parameters are $u=-1$, $\tau=4\eta=1$, so $x_c=-7/8$ and $v_c=-1/3$.}
		\label{fig:Schur_contours}
	\end{figure}
	If $\tau/\eta>-u^{-1}$, we have $v_c\in(-1,0)$, and the regions where $\Re\bigl(G(w;x_c)-G(v_c;x_c)\bigr)$ have constant sign are shown in \Cref{fig:Schur_contours}. Because the original $v$ and $w$ contours pass between $-1$ and $0$ and $v$ is inside $w$, we can deform the contours in a desired way. Namely, the $v$ contour can be chosen to coincide with the contour $\Im G(v;x_c)=\Im G(v_c;x_c)$ which is inside the unshaded region in \Cref{fig:Schur_contours}. The $w$ contour should coincide with the other contour $\Im G(w;x_c)=\Im G(v_c;x_c)$ in a neighborhood of $v_c$, and then encircle the $v$ contour and stay inside the shaded region in \Cref{fig:Schur_contours}.

	Set 
	\begin{equation}\label{Schur_variance}
		\sigma_{\eta,\tau,u}:=-v_c(G'''(v_c;x_c))^{1/3}=
		\frac{(-u\tau\eta)^{1/6}\bigr(1+\sqrt{-u\eta/\tau}\bigl)^{2/3}
		\bigr(1-\sqrt{-u\tau/\eta}\bigl)^{2/3}}{1-u},
	\end{equation}
	and let
	\begin{equation*}
	\begin{split}
		i=x_cM-\sigma_{\eta,\tau,u}M^{1/3}\tilde x,
		\qquad
		&j=x_cM-\sigma_{\eta,\tau,u}M^{1/3}\tilde y,
		\\
		v=v_c\Big(1-\tilde v M^{-1/3}/\sigma_{\eta,\tau,u}\Big)
		,\qquad
		&w=v_c\Big(1-\tilde w M^{-1/3}/\sigma_{\eta,\tau,u}\Big),
	\end{split}
	\end{equation*}
	where $\tilde x,\tilde y\in\R$, and $\tilde v,\tilde w$ correspond to a change of the integration variables in a neighborhood of $v_c$. Then 
	\begin{equation*}
		MG(v;i/M)=MG(v_c;x_c)+\sigma_{\eta,\tau,u} \tilde x M^{1/3} \log (-v_c)+\frac{\tilde v^3}3-\tilde x\tilde v.
	\end{equation*}
	Thus, with the above scaling of $i$ and $j$ we have
	\begin{equation}\label{Airy_kernel_convergence}
		\lim_{M\to+\infty}(-1)^{i-j}\frac{e^{\sigma_{\eta,\tau,u} \tilde y M^{1/3} \log (-v_c)}}{e^{\sigma_{\eta,\tau,u} \tilde x M^{1/3} \log (-v_c)}}
		\sigma_{\eta,\tau,u}M^{1/3}
		\tilde K(i,j)=K_{\mathrm{Ai}}(\tilde x,\tilde y),
	\end{equation}
	where
	\begin{equation*}
		K_{\mathrm{Ai}}(\tilde x,\tilde y)=
		\frac{1}{(2\pi\i)^2}\int\int
		e^{\tilde v^3/3-\tilde w^3/3-
		\tilde y\tilde v+\tilde x\tilde w}
		\frac{d\tilde vd\tilde w}{\tilde v-\tilde w}
		=\frac{\mathrm{Ai}(\tilde x)\mathrm{Ai}'(\tilde y)
		-
		\mathrm{Ai}'(\tilde x)\mathrm{Ai}(\tilde y)
		}{\tilde x-\tilde y}
	\end{equation*}
	is the Airy kernel. Here $\mathrm{Ai}$ is the Airy function and $\mathrm{Ai}'$ is its derivative, and in the integral representation the $\tilde v$ contour goes from $e^{5\pi\i/3}\infty$ to $e^{\pi\i/3}\infty$ and the $\tilde w$ contour goes from $e^{4\pi\i/3}\infty$ to $e^{2\pi\i/3}\infty$ without intersecting the $\tilde v$ contour. The prefactor $(-1)^{i-j}\frac{e^{\sigma_{\eta,\tau,u} \tilde y M^{1/3} \log (-v_c)}}{e^{\sigma_{\eta,\tau,u} \tilde x M^{1/3} \log (-v_c)}}$ does not change the correlation functions (given by determinants of the kernel) and thus does not affect the asymptotics, and the factor $\sigma_{\eta,\tau,u}M^{1/3}$ comes from the rescaling of the space on which the determinantal point process lives. The convergence \eqref{Airy_kernel_convergence} implies the desired convergence to the GUE Tracy--Widom distribution which is itself given by a Fredholm determinant $F_{GUE}(\tilde x)=\det(\mathbf{1}-K_{\mathrm{Ai}})_{L^2(\tilde x,+\infty)}$ \cite{tracy_widom1994level_airy} (up to certain tail estimates of contour integrals which we omit).

	For $0<\tau/\eta\le -u^{-1}$, the double critical point $v_c$ is between $0$ and $-u^{-1}$, and by a different deformation of contours one can get the second (flat) part of the limit shape for $\x_{\lfloor \eta M \rfloor}(\lfloor \eta M \rfloor,\lfloor \tau M \rfloor)$.
\end{proof}
\begin{remark}
	The case $a_1<1$ in \Cref{thm:special_q_TASEP_asymptotics} would lead to a shock in the corresponding macroscopic limit shape. In terms of contour integrals as in the proof of \Cref{thm:special_q_TASEP_asymptotics} this means that for certain values of $(\tau,\eta,u)$ the double critical point $v_c$ is inaccessible by contour deformations, and the main contribution to the integral comes from a neighborhood of $-a_1$.

	Inside the shock the fluctuations of the particle locations are Gaussian on scale $M^{1/2}$, outside the shock they have the GUE Tracy--Widom distribution on scale $M^{1/3}$ as in \Cref{thm:special_q_TASEP_asymptotics}, and the phase transition between the two regimes is of Baik--Ben Arous--P\'{e}ch\'{e} (BBP) type. Such a phase transition was first observed (in full generality) in \cite{BBP2005phase} for random matrices and shown to hold for the continuous time $q$-TASEP in \cite{barraquand2015phase} and for the ASEP and the stochastic six vertex model in \cite{AmolBorodin2016Phase}. We will not formulate the corresponding results in our case. A computation with Schur measures leading to the BBP phase transition can be found in \cite[Appendix B]{AmolBorodin2016Phase}.

	The BBP phase transition is not the only phase transition present in the stochastic higher spin six vertex model (or its degenerations). The upcoming work \cite{BorodinPetrov2016Exp} will describe phase transitions affecting the macroscopic limit shape but not changing the fluctuation exponents.
\end{remark}
 
% subsection asymptotics (end)

% section asymptotics_of_q_tasep_via_schur_measures (end)

\printbibliography

\end{document}